\documentclass[letter,11pt]{amsart}
\usepackage{geometry}                
\usepackage[latin1]{inputenc}
\usepackage{mathpazo}  
\usepackage[english]{babel}
\usepackage[usenames,dvipsnames]{color}
\usepackage{url}
\usepackage{amsmath,amssymb,latexsym,mathrsfs,amssymb,amsthm}
\usepackage{enumerate,subcaption}  
\usepackage{mathtools} 
\usepackage{epstopdf}
\usepackage{graphicx}
\usepackage[all]{xy}
\usepackage{setspace,cite}
\usepackage[usenames,dvipsnames]{xcolor}
\usepackage[colorlinks=true,linkcolor=PineGreen, citecolor=DarkOrchid, urlcolor=BurntOrange, pagebackref]{hyperref}

\usepackage{color,graphicx,tikz, tkz-euclide, float, pgfplots, overpic}
\usetikzlibrary{arrows,calc, decorations.markings, positioning}
\usetkzobj{all}

\newtheorem{theorem}{Theorem}[section]
\newtheorem{lemma}[theorem]{Lemma}

\newtheorem{rhp}{Riemann--Hilbert Problem}

\newtheorem*{notation}{Notation}

\theoremstyle{definition}    
\newtheorem{definition}[theorem]{Definition}

\theoremstyle{remark}
\newtheorem{remark}[theorem]{Remark}
\newcommand{\rhref}[1]{RH Problem~\ref{#1}}

\let\Re=\undefined\DeclareMathOperator{\Re}{Re}
\let\Im=\undefined\DeclareMathOperator{\Im}{Im}

\let\mb = \mathbf

\DeclareMathOperator{\diag}{diag}

\newcommand{\D}{\ensuremath{\,\mathrm{d}}}
\newcommand{\I}{\ensuremath{\mathrm{i}}}
\newcommand{\E}{\ensuremath{\,\mathrm{e}}}
\newcommand{\defeq}{\vcentcolon=}
\newcommand{\eqdef}{=\vcentcolon}
\newcommand{\sgo}{\sigma_1}

\makeatletter
\renewcommand*\env@matrix[1][\arraystretch]{%
  \edef\arraystretch{#1}%
  \hskip -\arraycolsep
  \let\@ifnextchar\new@ifnextchar
  \array{*\c@MaxMatrixCols c}}
\makeatother

\newcounter{smalllist}

\let\originalleft\left
\let\originalright\right
\renewcommand{\left}{\mathopen{}\mathclose\bgroup\originalleft}
\renewcommand{\right}{\aftergroup\egroup\originalright}

\title[Numerical IST for KDV with discontinuous step-like data]{On numerical inverse scattering for the Korteweg-de Vries equation with discontinuous step-like data}

\author{Deniz Bilman}
\author{Thomas Trogdon}

\address{
University of Michigan, Ann Arbor, MI 48109, USA.
}
\email{bilman@umich.edu}  
\address{
University of California, Irvine, CA.
}
\email{ttrogdon@uci.edu}

\date{\today}
\thanks{The authors would like to thank Percy Deift, Mark Hoefer, and Peter Miller for their input on this work.}

\begin{document}

 \begin{abstract}
   We present a method to compute dispersive shock wave solutions of the Korteweg--de Vries equation that emerge from initial data with step-like boundary conditions at infinity. We derive two different Riemann--Hilbert problems associated with the inverse scattering transform for the classical Schr\"odinger operator with possibly discontinuous, step-like potentials and develop relevant theory to ensure unique solvability of these problems. We then numerically implement the Deift--Zhou method nonlinear steepest descent to compute the solution of the Cauchy problem for small times and in two asymptotic regions.  Our method applies to continuous and discontinuous initial data.
 \end{abstract}
 
\maketitle

\section{Introduction}


Consider the Korteweg--de Vries (KdV) equation in the form
\begin{equation}
u_t + 6 u u_{x} + u_{xxx} = 0, \quad x\in\mathbb{R},
\label{eq:KdV}
\end{equation}
which is completely integrable \cite{GGKM} and admits soliton solutions that decay exponentially fast as $x\to\pm \infty$. 
\begin{figure}[ht]
\begin{overpic}[width=.95\linewidth]{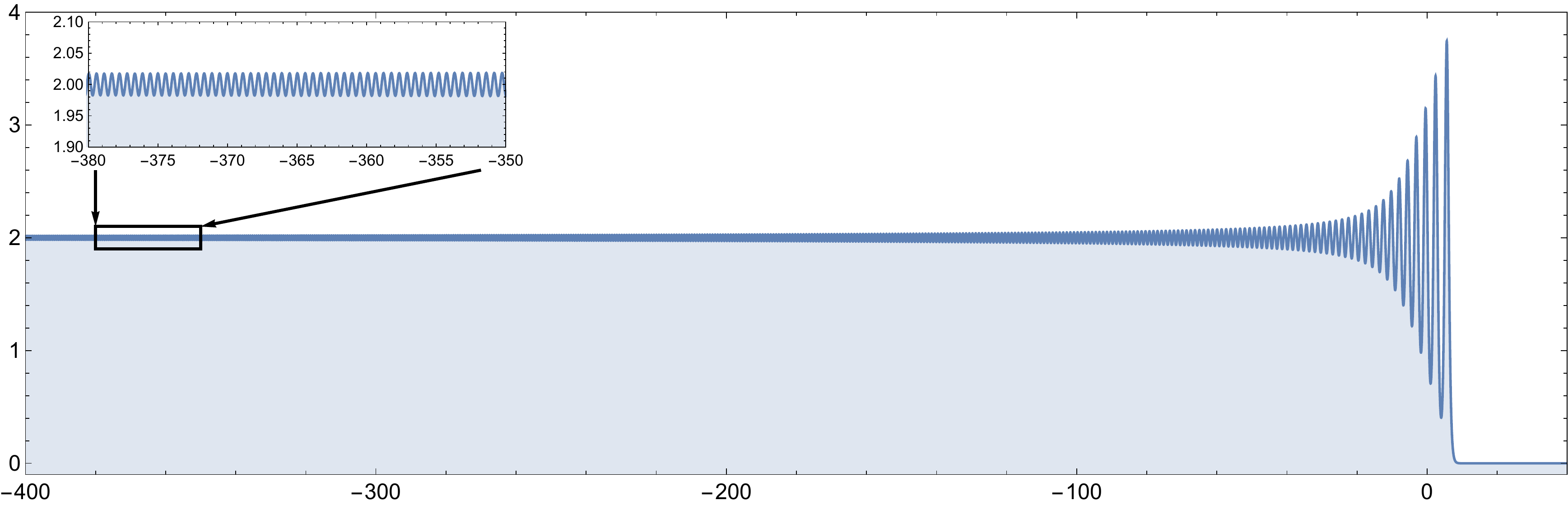}
    \put(50,-2){$x$}
    \put(-3,13){\rotatebox{90}{$u(x,1)$}}
  \end{overpic}
  \caption{ The spatial extent solution of the KdV equation at $t = 1$ when $u(x,0) = c^2,~ x < 0$ and $u(x,0) = 0,~x \geq 0$, $c = \sqrt{2}$.  The initial data is discontinuous and the solution is highly oscillatory for all $t > 0$.  Note that this solution does not satisfy \eqref{eq:step-like-bc} but Remark~\ref{r:boost} gives the method for obtaining this solution directly from one that does.}\label{f:dsw:large}
\end{figure}
For initial initial data with sufficient smoothness and decay on a zero background, the solution of the Cauchy initial-value problem is given asymptotically by a sum of $1$-solitons in the (soliton) region $x/t>C$ for some constant $C>0$ as $t\to +\infty$ \cite{Grunert2009}. Presence of non-zero boundary conditions at infinity, however, gives rise to a fundamentally different long-time solution profile. Monotone initial data $u(x,0)=q(x)$ with boundary conditions
\begin{align}\label{eq:bcs}
  \lim_{x\to-\infty}q(x)=q_\mathrm{l} \quad \mathrm{and} \quad \lim_{x\to+\infty}q(x)=q_\mathrm{r},
\end{align}
   gives rise to generation of a number of \emph{dispersive shock waves} (DSWs) if $q_\mathrm{l}>q_\mathrm{r}$ \cite{Gurevich1974a}. If $q_\mathrm{l} < q_\mathrm{r}$, however, the dynamics generate a \emph{rarefaction fan} and the solution is asymptotically given by $(x-x_0)/(6t)$ for $q_\mathrm{l}t<x-x_0<q_\mathrm{r}t$ as $t\to +\infty$ \cite{Andreiev2016}. An asymptotic description for the solution is much more complicated in the former case, where DSWs emerge \cite{Egorova2013}.            

   The generation of DSWs is also closely related to the regularization of shock waves in Burgers' equation $u_t + 6 uu_x = 0$ using the small-dispersion KdV (sKdV) equation $u_t + 6 u u_{x} + \varepsilon ^2 u_{xxx} = 0, x\in\mathbb{R}, 0<\varepsilon \ll 1$.  
   The initial-value problem for the sKdV equation with so-called ``single hump'' initial data was considered in the seminal work of Lax and Levermore \cite{Lax1979} and the subsequent series of papers \cite{Lax1983,Lax1983a,Lax1983b} where inverse scattering transform methods were used to obtain the limiting solution as $\varepsilon \downarrow 0$ for fixed $t > 0$. The methodology of Lax-Levermore was then extended by Venakides \cite{Venakides1990} to ``single potential-well" initial data where the reflection coefficient plays a significant role as $\varepsilon \downarrow 0$. 
Formation of DSWs, relevant asymptotics and the relation to the boundary conditions \eqref{eq:bcs} in this small dispersion limit $\varepsilon \downarrow 0$ of the sKdV equation were studied numerically in the works of Grava and Klein \cite{Grava2012,grava-klein}.  Recently, the generation of DSWs have been studied in various physical contexts, such as viscous fluid conduits \cite{Maiden2016}. For a review on DSWs, see \cite{Biondini2016a} and the articles in this special issue (in particular, see \cite{El2015, Miller2016, Biondini2016b, Tovbis2016,Trillo2016,El2016}).

We consider solutions of \eqref{eq:KdV} from computational special functions point of view.  Owing to the specialized integrable structure of the KdV equation, solutions should be computable in much the same way as Airy functions are computable in nearly any software package.  This philosophy, when implemented by performing numerical inverse scattering, allows one the freedom of performing nonlinear superpositions that are otherwise beyond reach \cite{TrogdonDressing,Trogdon2013d}.  Specifically, we consider the solution of the KdV equation with Heaviside initial data, as displayed in Figure~\ref{f:dsw:large}, to be a special function.  Taking a more ambitious stance, we aim to compute solutions of \eqref{eq:KdV} with \eqref{eq:bcs} for all $x$ and $t$.  This paper is the first step in that direction.  We anticipate that this full development will allow the investigation and classification of new and well-known phenomena within the KdV equation such as identifying the \emph{spectral signature} of a DSW.

More precisely, we consider solutions of the KdV equation \eqref{eq:KdV} with \emph{step-like} asymptotic profile
\begin{equation}
|u(x,t) - H_c(x)| = o(1),\quad |x|\to\infty,
\label{eq:step-like-bc}
\end{equation}
for all $t\in\mathbb{R}_{\geq 0}$, where
\begin{equation}
H_c(x)\defeq \begin{cases} 
-c^2& x>0,\\
0& x\leq 0, 
\end{cases}
\end{equation}
for $c\in\mathbb{R}_{>0}$. To specify the initial data for the KdV equation, we write
\begin{align}\label{eq:u0}
  u(x,0) = u_0(x) + H_c(x)
\end{align}
and $u_0$ is a real-valued function.  Our theoretical developments require $u_0$, which we refer to as a perturbation, to be in a polynomially-weighted $L^1$ space while our computational results require more: $u_0$ should be at least piecewise smooth and in an exponentially-weighted $L^1$ space. We develop the relevant Riemann--Hilbert (RH) theory for the inverse scattering transform (IST) associated with the KdV equation (i.e., for the classical Schr\"odinger operator with step-like potentials $u(\cdot, t)$) and pose two different RH problems that are amenable to numerical computations using the framework introduced in \cite{TrogdonSOBook}. We then make use of this RH theory to compute the solution of the Cauchy initial-value problem for the KdV equation with the boundary conditions \eqref{eq:step-like-bc} for small $t\geq 0$. Figure~\ref{f:dsw:large} gives the solution of the KdV equation with $u_0(x) = 0$, $c =\sqrt{2}$ at $t = 1$.

\begin{remark}\label{r:boost}
Let $\tilde u$ solve \eqref{eq:KdV} with
\begin{equation}
  \tilde u(x,0) = u(x,0) - a,
\end{equation}
then
\begin{equation}
  u(x,t) = \tilde u(x - 6at,t) + a.
\end{equation}
This is the so-called \emph{Galilean boost} symmetry of the KdV equation.  Using this, any solution $\tilde u$ of \eqref{eq:KdV} satisfying \eqref{eq:bcs} with $q_{\mathrm l} \geq q_{\mathrm r}$  can be obtained from a solution $u$ satisfying \eqref{eq:step-like-bc} by
\begin{equation}
\tilde u(x,t) = u(x- 6 q_{\mathrm l} t,t)+ q_{\mathrm l},\quad c^2=q_{\mathrm l} -q_{\mathrm r}.
\end{equation}
\end{remark}


%

\subsection{Outline of the paper}
In Section~\ref{sec:KdV-Scatter}, we present the necessary scattering theory for Schr\"odinger operators with step-like potentials in context of the \emph{direct scattering transform} for the KdV equation \eqref{eq:KdV}. Some of this material is based on the work of Kappeler and Cohen \cite{Cohen1985, Kappeler1986}, and also on the work of Deift and Trubowitz \cite{Deift1979}. As smoothness and decay properties of various spectral functions are important in obtaining a robust numerical inverse scattering transform, we include the details on scattering theory as they become necessary. In Section~\ref{sec:KdV-RHPs}, we define the right and left reflection coefficients on $\mathbb{R}$, derive their decay and smoothness properties as well as relations between left and right scattering data. We then pose two RH problem formulations of the inverse scattering transform for the KdV equation, one using the left scattering data and another using the right scattering data.  We note that one needs to use both of these problems to have an asymptotically accurate computational method.  This discussion unifies the work in \cite{Egorova2013} with that of Cohen and Kappeler.

 In Section~\ref{sec:KdV-Comp}, we give integrability conditions on the perturbation $u_0$ necessary for the deformations of the RH problems to be made in the subsequent sections and give details on computation of the scattering data. In Section~\ref{sec:KdV-IST-t0} we introduce contour deformations (analytic transformations) of \rhref{rhp:1t} and \rhref{rhp:2t} to apply the Deift--Zhou method of nonlinear steepest descent and compute the inverse scattering transform associated with the KdV equation for all $x\in\mathbb{R}$ at $t=0$. Having done that, we extend these deformations to small $t>0$ in Section~\ref{sec:tpos} to compute the solution $u(x,t)$ of the Cauchy problem for the KdV equation in two asymptotic regions of the $(x,t)$-plane. In Section~\ref{sec:KdV-num-ex} we present the computed solutions $u(x,t)$ for various step-like initial data.

The inclusion of solitons (if any) by incorporating residue conditions in these RH problems and derivation of the time dependence for the scattering data is performed in Appendix~\ref{sec:KdV-Time}.
We prove theorems on the unique solvability of these RH problems in Appendix~\ref{sec:KdV-Unique}.  We apply the dressing method \cite{ZakharovDressing} to establish \emph{a posteriori} that the RH problems we pose produce solutions of the KdV equation, see Theorem~\ref{t:main}.  Establishing unique solvability of the RH problems, without assuming existence of the solution of the KdV equation, is necessary to apply the dressing method.   Additionally, in the process, we show that a singular integral operator that we encounter in the numerical solution of a RH problem is invertible. For these reasons we expend considerable effort in Appendix~\ref{sec:KdV-Unique}.

\begin{remark}  We consider the setting $q_l > q_r$.  The case $q_l< q_r$ can be treated by mapping $(x,t) \mapsto (-x,-t)$ as this leaves \eqref{eq:KdV} invariant, noting that Theorem~\ref{t:main} applies.
  \end{remark}

\begin{notation}
We use the following notational conventions:
\begin{itemize}
\item We denote the following weighted $L^p$ spaces on an oriented (rectifiable) contour $\Gamma$:
\begin{equation}
L^p(\Gamma,\D \mu) = \left\lbrace f\colon \Gamma\to \mathbb{C} ~\Big\vert~ \int_\Gamma |f(s)|^p \D \mu(s) < \infty \right\rbrace.
\end{equation}
Also, $L^p(\Gamma) := L^p(\Gamma,|\D s|)$ where $|\D s|$ refers to arclength measure.
\item We use $\sigma_1$ to denote the first Pauli matrix
\begin{equation}\label{eq:sgo}
\sgo = \begin{bmatrix} 0 & 1 \\ 1 & 0 \end{bmatrix}.
\end{equation}
\item In the discussion of RH problems we use the following notation.  For a function $f$ defined on a subset of $\mathbb C$ with a non-empty interior, we will use $f(z)$ to refer to the values of $f$.  For a function $f$ defined on a contour $\Gamma \subset \mathbb C$ we will use $f(s)$ to refer to values of $f$.
\item Given a point $s$ on an oriented contour $\Gamma\subset\mathbb{C}$, $f^{+}(s)$ (resp.\@ $f^{-}(s)$) denote the non-tangential boundary values of $f(z)$ as $z\to s$ from left (resp.\@ right) with respect to orientation of $\Gamma$.
\item We use bold typeface to denote matrices and vectors with the exception of $\sigma_1$ defined in \eqref{eq:sgo}.
\end{itemize}
\end{notation}


\label{sec:KdV-Intro}

\section{The scattering problem and its solution}
\label{sec:KdV-Scatter}

The spatial part of the Lax pair for the KdV equation is the spectral problem
\begin{equation}
\mathcal{L}\psi = E\psi,~\qquad \mathcal{L}\psi \defeq -\psi_{xx} - u(x,t)\psi, \qquad E = z^2,
\label{eq:spec1}
\end{equation}
where $u$ satisfies the KdV equation \eqref{eq:KdV} and $\mathcal{L}$ is the Schr\"odinger operator.  The temporal part of the Lax pair is the evolution equation
\begin{equation}
\psi_t = \mathcal{P}\psi,~	\qquad\mathcal{P}\psi \defeq -4 \psi_{xxx} -3 u(x,t)\psi -6u(x,t)\psi_x.
\label{eq:spec1t}
\end{equation}
To compute scattering data associated with the given Cauchy initial data we proceed with the construction of the \emph{Jost} solutions of the spectral problem \eqref{eq:spec1}.
We first solve the scattering problem at $t = 0$.  
  {It is convenient to define the complementary functions $u_{0}^{\text{l/r}}(x)$} by 
\begin{equation}
\begin{aligned}
  u^{\text{l}}_0(x) = \begin{cases} u_0(x) & x \leq 0,\\
    u_0(x) -c^2 & x > 0, \end{cases} \quad\text{and}\quad u^{\text{r}}_0(x) = \begin{cases} u_0(x) & x \geq 0,\\
    u_0(x) + c^2 & x < 0. \end{cases}
\end{aligned}
\end{equation}
Recall that we assume that the Cauchy initial data is
\begin{equation}
  u(x,0) = u_0^{\text{l}}(x).
\end{equation}

\subsubsection{Asymptotic spectral problem as $x\to-\infty$}
%
%

On the left-end of the spatial domain, formally, \eqref{eq:spec1} is asymptotically 
\begin{equation}\label{eq:spec2}
\psi_{xx} = - z^2\psi,
\end{equation}
which has a fundamental set of solutions given by $\{\E^{+\I z x},\E^{-\I z x} \}$. Therefore, for $z\in\mathbb{R}$, \eqref{eq:spec1} has the following two independent solutions that are uniquely determined by their asymptotic behavior as $x\to-\infty$:
\begin{align}
\phi^{\text{p}}(z;x) &= \E^{\I z x}(1 + o(1)),\quad {x\to - \infty},\label{eq:phi-p}\\
\phi^{\text{m}}(z;x) &= \E^{-\I z x}(1 + o(1)),\quad {x\to - \infty}\label{eq:phi-m}.
\end{align}
These functions can be defined through Volterra integral equations
\begin{equation}
\begin{aligned}
  \phi^{\text{p}}(z;x) &= \E^{\I z x} + \frac{1}{2 \I z} \int_{-\infty}^x \left(  \E^{\I z (x-\xi)} - \E^{\I z (\xi-x)} \right) u_0^{\text{l}}(\xi) \phi^{\text{p}}(z;\xi) \D \xi,\\
  \phi^{\text{m}}(z;x) &= \E^{-\I z x} - \frac{1}{2 \I z} \int_{-\infty}^x \left(  \E^{ \I z (\xi-x)} - \E^{ \I z (x - \xi)} \right) u_0^{\text{l}}(\xi) \phi^{\text{m}}(z;\xi) \D \xi.  
\end{aligned}
\label{eq:Volterra-phi}
\end{equation}
which can be solved by Neumann series for $z \in \mathbb R$ and $u_0 \in L^1(\mathbb{R}, (1 + |x|) \D x)$. See \cite[Chapter 1]{Cohen1985} and also \cite[Section 2]{Deift1979} for a detailed construction.

\subsubsection{Asymptotic spectral problem as $x\to+\infty$}
Since $u(x)\to -c^2$ as $x\to+\infty$, we consider, formally, the problem \eqref{eq:spec1} asymptotically:
\begin{equation}
\psi_{xx} - c^2\psi = - z^2\psi,
\label{eq:spec3}
\end{equation}
and the eigenvalues associated with this differential equation are doubly-branched. More precisely, we have the fundamental set of bounded solutions to \eqref{eq:spec3} given by $\{\E^{\I \lambda x}, \E^{-\I \lambda x}\}$,
where $\lambda$ depends on $z$ through the algebraic relation $\lambda^2 = z^2 - c^2$ (characteristic equation for the eigenvalues $\I \lambda$ of the constant coefficient equation \eqref{eq:spec3}) which defines a Riemann surface with genus $0$. To be concrete, we define $\lambda(z)$ to be the function analytic for complex $z$ with the exception of a horizontal branch cut 
\begin{equation}
\Sigma_{\mathrm{c}}\defeq [-c,c]\subset\mathbb{R},
\end{equation}
between the branch points $z=\pm c$, whose square coincides with $z^2 - c^2$ and satisfies $\lambda(z)= z + O(z^{-1})$ as $z\to \infty$. With these properties, $\lambda(z)$ is a scalar single-valued complex function that is analytic in the region $\mathbb{C}\setminus \Sigma_\text{c}$. 
We now define two more independent solutions of the problem \eqref{eq:spec1} that are determined, for $\lambda(z)\in\mathbb{R}$ (i.e., $z\in\mathbb{R}\setminus \Sigma_{\text{c}}$), by their asymptotic behavior as $x\to+\infty$:
\begin{align}
\psi^{\text{p}}(z;x) &= \E^{\I \lambda(z) x}(1 + o(1)),\quad {x\to + \infty}\label{eq:psi-p}\\
\psi^{\text{m}}(z;x) &= \E^{-\I \lambda(z) x}(1 + o(1)),\quad {x\to + \infty}\label{eq:psi-m}.
\end{align}
The existence of such solutions is again established through Volterra integral equations
\begin{equation}
\begin{aligned}
  \hat\psi^{\text{p}}(z;x) &= \E^{\I z x} + \frac{1}{2 \I z} \int_x^\infty \left(  \E^{\I z (x-\xi)} - \E^{\I z (\xi-x)} \right) u_0^{\text{r}}(\xi) \hat\psi^{\text{p}}(z;\xi) \D \xi,\\
  \hat\psi^{\text{m}}(z;x) &= \E^{-\I z x} - \frac{1}{2 \I z} \int_x^\infty \left(  \E^{ \I z (\xi-x)} - \E^{ \I z (x - \xi)} \right) u_0^{\text{r}}(\xi) \hat\psi^{\text{m}}(z;\xi) \D \xi.  
\end{aligned}
\label{eq:Volterra-psi}
\end{equation}
with  $\psi^{\text{p/m}}(z;x) = \hat \psi^{\text{p/m}}(\lambda(z);x)$.  Again, the solutions $\hat \psi^{\text{p/m}}(z;x)$ are well-defined for $z \in \mathbb R$, and hence $\psi^{\text{p/m}}(z;x)$ are well-defined for $\lambda(z)\in\mathbb{R}$  (i.e., for $z\in\mathbb{R}\setminus \Sigma_\mathrm{c}$) and $u_0 \in L^1((1+ |x|) \D x)$. See again \cite[Chapter 1]{Cohen1985} and also \cite[Section 2]{Deift1979} for details.

\subsection{Left  and right reflection coefficients}

The left (resp.\@ right) reflection coefficient $R_{\mathrm{l}}$ (resp. $R_{\mathrm{r}}$) are defined through the scattering relations for $z\in\mathbb{R}\setminus\Sigma_\mathrm{c}$
\begin{equation}
\begin{aligned}
   \psi^{\text{p}}(z;x) &=  a(z) \phi^{\text{p}}(z;x)  + b(z) \phi^{\text{m}}(z;x),\\
   \phi^{\text{m}}(z;x) &=  B(z)  \psi^{\text{p}}(z;x)  + A(z) \psi^{\text{m}}(z;x).
\end{aligned}
\label{eq:scattering-rel}
\end{equation}

\begin{remark} \label{r:schr} It is important to note that while $\psi^{\text{p/m}}$ and $\phi^{\text{p/m}}$ are each sets of two linearly independent solutions of the same differential equation for all $x \in \mathbb R$, if $x = 0$, we can replace $u_0^{\text{l/r}}$ with $u_0$ in the associated integral equations \eqref{eq:Volterra-phi} and \eqref{eq:Volterra-psi}.  Then the scattering theory is interpreted as the traditional scattering theory for the one-dimensional Schr\"odinger operator, where one set of eigenfunctions is modified via the $z\mapsto \lambda(z)$ transformation. \end{remark}

The system \eqref{eq:scattering-rel} can be solved for $a(z),b(z)$ and $A(z),B(z)$ using Wronksians $W(f,g) = fg' - gf'$.  Doing so, we define for $z \in \mathbb R$ and $\lambda(z) \in \mathbb R$, 
\begin{align}
  R_{\mathrm{l}}(z) &:= \frac{b(z)}{a(z)} = \frac{W(\psi^{\text{p}}(z;\cdot), \phi^{\text{p}}(z;\cdot))}{W(\psi^{\text{p}}(z;\cdot), \phi^{\text{m}}(z;\cdot))},\label{eq:l-offcut}\\
  R_{\mathrm{r}}(z) &:= \frac{B(z)}{A(z)} = -\frac{W(\phi^{\text{m}}(z;\cdot),\label{eq:r-offcut} \psi^{\text{m}}(z;\cdot))}{W(\phi^{\text{m}}(z;\cdot), \psi^{\text{p}}(z;\cdot))}.
\end{align}
These are the so-called left ($R_{\mathrm{l}}$) and right ($R_{\mathrm{r}}$) reflection coefficients. We note that $W(\phi^{\text{p}}, \phi^{\text{m}}) = - 2 \I z$  and  $W(\psi^{\text{p}}, \psi^{\text{m}}) = - 2 \I \lambda(z)$. Other important formul\ae{} are
\begin{equation}
\begin{aligned}
  a(z) &= \frac{W(\psi^{\text{p}}(z;\cdot), \phi^{\text{m}}(z;\cdot))}{W(\phi^{\text{p}}(z;\cdot), \phi^{\text{m}}(z;\cdot))} = \frac{W(\phi^{\text{m}}(z;\cdot), \psi^{\text{p}}(z;\cdot))}{2 \I z},\\
  b(z) &=  \frac{W(\phi^{\text{p}}(z;\cdot), \psi^{\text{p}}(z;\cdot))}{W(\phi^{\text{p}}(z;\cdot), \phi^{\text{m}}(z;\cdot))} = -\frac{W(\phi^{\text{p}}(z;\cdot), \psi^{\text{p}}(z;\cdot))}{2 \I z},\\
  A(z) &=  \frac{W(\psi^{\text{p}}(z;\cdot), \phi^{\text{m}}(z;\cdot))}{W(\psi^{\text{p}}(z;\cdot), \psi^{\text{m}}(z;\cdot))} = \frac{W(\phi^{\text{m}}(z;\cdot), \psi^{\text{p}}(z;\cdot))}{2 \I \lambda(z)},\\
  B(z) &= \frac{W(\phi^{\text{m}}(z;\cdot), \psi^{\text{m}}(z;\cdot))}{W(\psi^{\text{p}}(z;\cdot), \psi^{\text{m}}(z;\cdot))} = \frac{W(\psi^{\text{m}}(z;\cdot), \phi^{\text{m}}(z;\cdot))}{2 \I \lambda(z)}.  
\end{aligned}
\label{eq:a-b-A-B-Wronskian}
\end{equation}

\begin{remark}
Presence of the step-like boundary conditions rules out the existence of \emph{reflectionless} solutions (e.g., pure solitons). Indeed, setting both reflection coefficients $R_\mathrm{l}(z)$ and $R_\mathrm{r}(z)$ equal to $0$ enforces $\lambda(z)=z$, which holds if and only if $c=0$, resulting in a zero-background (vanishing boundary conditions at infinity). Additionally, $u(x,t)=H_c(x)$ is not a stationary solution of \eqref{eq:KdV}.
\end{remark}
\subsection{Regions of analyticity}

To analyze regions of the complex plane where the functions $\psi^{\text{p/m}}(z;x), \phi^{\text{p/m}}(z;x)$ are analytic in the variable $z$, we consider the Jost functions
\begin{equation}
\begin{aligned}
  N^{\text{p}}(z;x) := \phi^{\text{p}}(z;x) \E^{-\I z x}, \quad N^{\text{m}}(z;x) := \phi^{\text{m}}(z;x) \E^{\I z x},\\
  \hat M^{\text{p}}(z;x) := \hat \psi^{\text{p}}(z;x) \E^{-\I z x}, \quad \hat M^{\text{m}}(z;x) := \hat \psi^{\text{m}}(z;x) \E^{\I z x},\\
  M^{\text{p}}(z;x) := \hat M^{\text{p}}(\lambda(z);x), \quad M^{\text{m}}(z;x) := \hat M^{\text{m}}(\lambda(z);x).
\end{aligned}
\end{equation}
From \eqref{eq:Volterra-phi} and \eqref{eq:Volterra-psi} it immediately follows that the functions $N^{\text{p/m}}(z;x)$ and $M^{\text{p/m}}(z;x)$ satisfy the following Volterra integral equations for $z \in \mathbb R$
\begin{align}
  N^{\text{p}}(z;x) &= 1 + \frac{1}{2 \I z} \int_{-\infty}^x \left(  1 - \E^{2 \I z (\xi-x)} \right) u^{\text{l}}_0(\xi) N^{\text{p}}(z;\xi) \D \xi,\label{e:Np}\\
  N^{\text{m}}(z;x) &= 1 - \frac{1}{2 \I z} \int_{-\infty}^x \left(  1 - \E^{ 2\I z (x - \xi)} \right) u^{\text{l}}_0(\xi) N^{\text{m}}(z;\xi) \D \xi,\label{e:Nm} \\
  \hat M^{\text{p}}(z;x) &= 1 + \frac{1}{2 \I z} \int_x^\infty \left(  1 - \E^{2\I z (\xi-x)} \right) u^{\text{r}}_0(\xi) \hat M^{\text{p}}(z;\xi) \D \xi,\label{e:Mp}\\
  \hat M^{\text{m}}(z;x) &= 1 - \frac{1}{2 \I z} \int_x^\infty \left(  1 - \E^{ 2\I z (x - \xi)} \right) u^{\text{r}}_0(\xi) \hat M^{\text{m}}(z;\xi) \D \xi.\label{e:Mm}
\end{align}
For \eqref{e:Np} and \eqref{e:Nm} $x- \xi \geq 0$ and $x- \xi \leq 0$ for \eqref{e:Mp} and \eqref{e:Mm}.  This immediately implies that \eqref{e:Np} and \eqref{e:Mm} can be analytically continued for $\Im z < 0$ while \eqref{e:Nm} and \eqref{e:Mp} can be analytically continued for $\Im z > 0$.  It also follows from the asymptotics of $\lambda(z)$ that $(\Im z)(\Im \lambda(z)) > 0$ for $z \not\in \mathbb R$.  We note that these considerations immediately imply that $a(z)$ and $A(z)$ are analytic for $\Im z > 0$.

We now consider the large $z$ asymptotics of the above solutions, $N^{\text{p/m}}$ and $\hat M^{\text{p/m}}$ assuming $z$ is in the appropriate region of analyticity.

\begin{lemma}\label{l:neumann}
  If $u_0 \in L^1(\mathbb R)$ then for fixed $x\in\mathbb{R}$, $N^{\text{p/m}}(z;x) = 1 + O(z^{-1})$ and $\hat M^{\text{p/m}}(z;x) = 1 + O(z^{-1})$ as 
$z\to\infty$.
\end{lemma}
\begin{proof}
  We concentrate on one function, $N^\text{m}$, as the proof is the same for all. For $|z| > 1$ consider the Volterra integral equation
  \begin{equation}
    N^{\text{m}}(z;x) + \frac{1}{2 \I z} \int_{-\infty}^x \left(  1 - \E^{ 2\I z (x - \xi)} \right) u_0^{\text{l}}(\xi) N^{\text{m}}(z;\xi) \D \xi = 1,
    \label{e:Nm-VIE}
    \end{equation}
    which can be rewritten as $(\mathcal{I} + \mathcal K_z)N^{\text{m}}(z;\cdot) = 1$, where $\mathcal K_z$ is the Volterra integral operator given as
 \begin{equation}
[\mathcal{K}_z f](z;x) \defeq \frac{1}{2 \I z} \int_{-\infty}^x \left(  1 - \E^{ 2\I z (x - \xi)} \right) u_0^{\text{l}}(\xi) f(z;\xi) \D \xi.
 \end{equation} 
We proceed by showing that the Neumann series for the inverse operator $(\mathcal{I} + \mathcal K_z)^{-1}$ converges in the operator norm on $C^0((-\infty,X])$ for fixed $X\in\mathbb{R}$. Standard estimates yield
\begin{equation}
\begin{aligned}
    \| \mathcal K_z^n \|_{C^0((-\infty,X])} &\leq  \int_{-\infty}^X \int_{s_1}^{X} \int_{s_2}^{X} \cdots \int_{s_{n-1}}^{X} \prod_{j=1}^n |u^{\text{l}}_0(s_j)| \D s_n \cdots \D s_1\\
 & = - \int_{-\infty}^X \int_{s_1}^{X} \int_{s_2}^{X} \cdots  \int_{s_{n-\ell}}^X \frac{1}{\ell!} \frac{\D}{\D s_{n-\ell + 1}} \left( \int_{s_{n-\ell + 1}}^{X} |u_0^{\text{l}}(s)| \D s \right)^\ell  \D s_{n-\ell + 1} \prod_{j=1}^{n-\ell} |u^{\text{l}}_0(s_j)| \D s_{n-\ell} \cdots \D s_1\\
    & \leq \frac{1}{n!} (\|u_0\|_{L^1(\mathbb R)} + c^2|X|)^n, \quad n\in\mathbb{Z}_{>0}.
\end{aligned}
\end{equation}
This implies that $\|(\mathcal I + \mathcal K_z)^{-1}\|_{C^0((-\infty,X])} \leq \E^{\|u_0\|_{L^1(\mathbb R)} + c^2 |X|}$ for $|z| > 1$.
Then directly estimating \eqref{e:Nm-VIE}, we have that
\begin{equation}
      |N^{\text{m}}(z;x) -1| \leq \frac{\|u_0\|_{L^1(\mathbb R)} + c^2|X|}{|z|} \E^{\|u_0\|_{L^1(\mathbb R)} + c^2|X|}, \quad |z| > 1,
    \end{equation}
    proving the result for $N^{\text{m}}$.   Note that for $X < 0$, we can omit the $c^2|X|$ term from these estimates.
  \end{proof}
  
  \begin{remark}
  The reason it is enough to assume $u_0\in L^{1}(\mathbb{R})$ to prove Lemma~\ref{l:neumann} is because $z$ is away from zero. The additional decay assumption $u_0\in L^{1}(\mathbb{R},(1+|x|)\D x)$ in construction of the Jost solutions is required to handle the case when $z=0$, i.e., in general, for $z\in\mathbb{R}$.
  \end{remark}
We now compute the coefficients of the terms that are proportional to $z^{-1}$ in the large-$z$ asymptotic series expansions of these functions.
  \begin{lemma}\label{l:J-largez}
    For fixed $x$, As $|z| \to \infty$, $\Im z > 0$,
  \begin{equation}
\begin{aligned}
      2 \I z(N^{\text{m}}(z;x) -1) \to \int_{-\infty}^x u_0^{\text{l}}(\xi) \D \xi,\\
      2 \I z(\hat M^{\text{p}}(z;x) -1) \to \int_{x}^\infty u_0^{\text{r}}(\xi) \D \xi.
    \end{aligned}
\end{equation}

    For fixed $x$, As $|z| \to \infty$, $\Im z < 0$,
   \begin{equation}
\begin{aligned}
      2 \I z(N^{\text{p}}(z;x) -1) \to -\int_{-\infty}^x u_0^{\text{l}}(\xi) \D \xi,\\
      2 \I z(\hat M^{\text{m}}(z;x) -1) \to -\int_{x}^\infty u_0^{\text{r}}(\xi) \D \xi.
\end{aligned}
\end{equation}
  \end{lemma}
  \begin{proof}
    We only prove this for $N^{\text{m}}$.  The proofs for other functions are similar.  Consider, as $|z| \to \infty$, $\Im z > 0$,
   \begin{equation}
\begin{aligned}
      2 \I z(N^{\text{m}}(z;x) -1) &=  \int_{-\infty}^x \left(  1 - \E^{ 2\I z (x - \xi)} \right) u_0^{\text{l}}(\xi) (1 + O(z^{-1}) ) \D \xi\\
                                   & =  \int_{-\infty}^x \left(  1 - \E^{ 2\I z (x - \xi)} \right) u_0^{\text{l}}(\xi) \D \xi + O(z^{-1})  \\
                                   & = \int_{-\infty}^x u_0^{\text{l}}(\xi) \D \xi - \int_{-\infty}^x \E^{2 \I z (x - \xi)} u_0^{\text{l}}(\xi) \D \xi + O(z^{-1}).
   \end{aligned}
\end{equation}
    The claim follows if we show $\int_{-\infty}^x \E^{2 \I z (x - \xi)} u_0^{\text{l}}(\xi) \D \xi = o(1)$ as $|z| \to \infty$, $\Im z > 0$.  Indeed, this is the case since setting $y\defeq \xi-x$ we have
    \begin{equation}
      \int_{-\infty}^0 \E^{-2 \I z y} u_0^{\text{l}}(y + x) \D y  \to 0,\quad |z|\to\infty
    \end{equation}
    by the Riemann--Lebesgue lemma.
  \end{proof}
  \noindent It is important to note that from this lemma we obtain
  \begin{equation}
\begin{aligned}
    \lim_{|z| \to \infty, ~\Im z > 0} 2 \I z (M^{\text{p}}(z;x) -1) 
                                                                   & = \lim_{|z| \to \infty, ~\Im z > 0} 2\I \frac{z}{\lambda(z)} \lambda(z) (\hat M^{\text{p}}(\lambda(z);x) -1)
                                                                     = \int_{x}^\infty u_0^{\text{r}}(\xi) \D \xi,\\
    \lim_{|z| \to \infty, ~\Im z < 0} 2 \I z (M^{\text{m}}(z;x) -1) &= -\int_{x}^\infty u_0^{\text{r}}(\xi) \D \xi.
 \end{aligned}
\end{equation}

  \begin{lemma}\label{l:a-largez}
    If $u_0 \in L^1(\mathbb R)$, for $\Im z > 0$, $a(z) = 1 + O(z^{-1})$ as $z\to\infty$.  Furthermore
    \begin{equation}
      \lim_{z \to \infty, ~\Im z > 0} 2\I z (a(z) -1) = \int_{-\infty}^\infty u_0(\xi) \D \xi.
    \end{equation}
  \end{lemma}
  \begin{proof}
    We use the representation of $a(z)$ given in \eqref{eq:a-b-A-B-Wronskian} in terms of a Wronskian
    \begin{equation}
      a(z) = \frac{W(\phi^{\text{m}}(z;\cdot), \psi^{\text{p}}(z;\cdot))}{2 \I z},
    \end{equation}
    with
    \begin{equation}
    \begin{aligned}
      \phi^{\text{m}}(z;x) &= \E^{-\I z x} N^{\text{m}}(z;x),\\ \frac{\partial}{\partial x} \phi^{\text{m}}(z;x) &= \E^{-\I z x} \frac{\partial}{\partial x} N^{\text{m}}(z;x) - \I z \E^{-\I z x} N^{\text{m}}(z;x),\\
      \psi^{\text{p}}(z;x) &= \E^{\I \lambda(z)  x} M^{\text{p}}(z;x), \\ \frac{\partial}{\partial x} \psi^{\text{p}}(z;x) &= \E^{\I \lambda(z) x} \frac{\partial}{\partial x} M^{\text{p}}(z;x) + \I \lambda(z) \E^{\I \lambda(z) x} x M^{\text{p}}(z;x).
          \end{aligned}
    \end{equation}
    We find, by evaluating at $x = 0$,
    \begin{equation}
    \begin{aligned}
      a(z) &= \frac{1}{2 \I z} \left(\phi^{\text{m}}(z;x)\frac{\partial}{\partial x} \psi^{\text{p}}(z;x) - \psi^{\text{p}}(z;x) \frac{\partial}{\partial x} \phi^{\text{m}}(z;x)\right) \\
           & =\left( \frac{z + \lambda(z)}{2 \I z} \right) N^{\text{m}}(z;0) M^{\text{p}}(z;0) \\
           &+ \frac{1}{2 \I z} \left(  N^{\text{m}}(z;x) \frac{\partial}{\partial x} M^{\text{p}}(z;0)  - M^{\text{p}}(z;x)\frac{\partial}{\partial x} N^{\text{m}}(z;0)  \right).
           \end{aligned}
    \end{equation}
    It then follows that $\frac{\partial}{\partial x} N^{\text{m}}(z;0) = O(z^{-1})$ and  $\frac{\partial}{\partial x} M^{\text{p}}(z;0) = O(z^{-1})$ so that
    \begin{equation}
      \lim_{|z| \to \infty} 2 \I z(a(z) - 1) = \int_{-\infty}^\infty u_0(\xi) \D \xi.
    \end{equation}
  \end{proof}

  \subsection{Differentiability with respect to $z$ on $\mathbb R$}

  We now consider the conditions on $u_0$ under which $\hat \psi^{\text{p/m}}$ and $\phi^{\text{p/m}}$ and their first-order $x$, derivatives both evaluated at $x = 0$, are differentiable $k$ times with respect to $z$ for $z \in \mathbb R$.
  \begin{lemma}\label{l:diff}
    Let $k$ be a non-negative integer and suppose that $u_0 \in L^1(\mathbb{R},(1+|x|)^{k+1} \D x)$.  Then for each fixed $x\in\mathbb{R}$
    \begin{equation}
      \hat \psi^{\text{p/m}}(\cdot;x), ~~ \hat \psi^{\text{p/m}}_x(\cdot;x), ~~ \phi^{\text{p/m}}(\cdot;x), ~\phi_x^{\text{p/m}}(\cdot;x) \in C^k(\mathbb R).
    \end{equation}
    Furthermore, for fixed $x$, the $\ell$-th derivative with respect to $z$, $\ell \leq k$, is continuous as a function of $u_0 \in L^1(\mathbb{R},(1+|x|)^{\ell+1} \D x)$ and $z \in \mathbb R$.
  \end{lemma}
  \begin{proof}
    We prove this only for $\phi^{\text{m}}(z;x)$ as the proofs for the others are similar.  And to prove this for $\phi^{\text{m}}(z;x)$, it suffices to prove this for the renormalized function $N^{\mathrm m}(z;x)$.  We begin with rewriting the Volterra integral equation \eqref{e:Nm} as
  \begin{equation}
     N^{\text{m}}(z;x) - \int_{-\infty}^x K(z;x-\xi) u_0^{\text{l}}(\xi) N^{\text{m}}(z;\xi) \D \xi =  1,\quad K(z;x)\defeq  \frac{1}{2\I z} \left( \E^{2 \I z x} -1\right),
  \end{equation}
  which has the form $(\mathcal{I} + \mathcal{K}_z)[N^\text{m}(z;\cdot)]=1$ with $\mathcal{K}_z$ denoting the Volterra integral operator 
  \begin{equation}
  \mathcal{K}_z[f](x)\defeq - \int_{-\infty}^x K(z;x-\xi) u_0^{\text{l}}(\xi) f(\xi) \D \xi.
  \end{equation}
  For $h \neq 0$, the difference function $N_h^{\text{m}}(z;x) \defeq  N^{\text{m}}(z+h;x)- N^{\text{m}}(z;x)$ satisfies the equation
 \begin{equation}
    N_h^{\text{m}}(z;x) - \int_{-\infty}^x K(z;x-\xi) u_0^{\text{l}}(\xi) N_h^{\text{m}}(z;\xi) \D \xi = \int_{-\infty}^x \left[ K(z+h;x-\xi) - K(z;x-\xi) \right]u_0^{\text{l}}(\xi) N^{\text{m}}(z+h;\xi) \D \xi  .
    \label{e:VIE-Nh}
  \end{equation}
  Because the operator $(\mathcal{I} + \mathcal{K}_z)$ on the left-hand side is invertible on $C^0((-\infty,X])$, for any fixed $X \in \mathbb R$, uniform continuity of $N^\text{m}(z;x)$ in the spectral variable $z$ follows if we show that the right-hand side tends uniformly to zero as $h\to 0$. We fix $X\in\mathbb{R}$. The modulus of the expression on the right-hand side of \eqref{e:VIE-Nh} is bounded above by
  \begin{equation}
    I(x) \defeq \int_{-\infty}^X \left| K(z+h;x-\xi) - K(z;x-\xi) u_0^{\text{l}}(\xi) N^{\text{m}}(z+h;\xi) \right| \D \xi,\quad x\in(-\infty,X),
 \end{equation}
  since $z\in\mathbb{R}$. Thus, we will show that $I(x) \to 0$ as $h \to 0$. We write $K(z;x) \eqdef \kappa(zx)x$, with
  \begin{equation}
    \kappa(s) := \begin{cases} \frac{\E^{2\I s} -1}{2 \I s} & s\in\mathbb{R}\setminus \{ 0 \},\\
    1& s=0,
    \end{cases}
  \end{equation}
  which is bounded and differentiable for $s\in\mathbb{R}$, with all of its derivatives being also bounded for all $s\in\mathbb{R}$. Now, since for any fixed $x$, $N^{\text{m}}(z;x)$ is bounded uniformly in $z \in \mathbb R$ (see the proof of Lemma~\ref{l:neumann}) by, say, $M>0$, we have
  \begin{equation}
  I(x) \leq M \int_{-\infty}^X |\kappa( (z + h)(x-\xi)) - \kappa(z(x-\xi))| \frac{|x-\xi|}{1 + |\xi|} |u_0^{\mathrm l}(\xi)| (1 + |\xi|) \D \xi,\quad x\in(-\infty, X].
  \end{equation}
  Now, let $\epsilon > 0$.   Because $\kappa$ is a bounded function and $u_0 \in L^1((1+|x|) \D x)$ there exists $\ell = \ell(\epsilon) \leq X$ such that  
  \begin{equation}
 M \int_{-\infty}^\ell |\kappa( (z + h)(x-\xi)) - \kappa(z(x-\xi))| \frac{|x-\xi|}{1 + |\xi|} |u_0^{\mathrm l}(\xi)| (1 + |\xi|) \D \xi < \epsilon
  \end{equation}
  for all $x\leq X$. Therefore
  \begin{equation}
  I(x) \leq \epsilon +  M \int_\ell^X |\kappa( (z + h)(x-\xi)) - \kappa(z(x-\xi))| |x-\xi| |u_0^{\mathrm l}(\xi)| \D \xi,\quad x\in(-\infty, X]
  \label{e:I0-bound}
  \end{equation}
  since $z\in\mathbb{R}$. On the other hand, by the Fundamental Theorem of Calculus we have
  \begin{equation}
    \kappa( (z + h)(x-\xi)) - \kappa(z(x-\xi)) = (x-\xi)\int_z^{z+h} \kappa'(s(x-\xi))\D s,
  \end{equation}
  which tends to zero, uniformly for $\xi \in [\ell,x]$, for any $x \leq X$, as $h\to 0$ because $\kappa'$ is bounded ($|\kappa'(s)|\leq 1$ for all $s\in\mathbb{R}$).  {Since $\epsilon>0$ in \eqref{e:I0-bound} can be made arbitrarily small,} this establishes uniform continuity of $N^\text{m}(z;x)$ with respect to $z \in \mathbb R$.

 To generalize this to existence and continuity of the $z$-derivatives of $N^\text{m}(z;x)$ for $z\in\mathbb{R}$, we first use boundedness of $\kappa$ and all of its derivatives on $\mathbb R$ and immediately obtain the estimate
  \begin{equation}
    |\partial_z^j K(z;x)| \leq C_j |x|^{j+1}, \quad j = 0,1,2,\ldots.
  \end{equation}
We then use integral equation satisfied by the difference quotient ${\tilde N_h^{\text m}}(z;x) \defeq N_h^{\text m}(z;x)/h$:
  \begin{align}\label{eq:diffq}
    \tilde N_h^{\text{m}}(z;x) - \int_{-\infty}^x K(z;x-\xi) u_0^{\text{l}}(\xi) \tilde N_h^{\text{m}}(z;\xi) \D \xi = \int_{-\infty}^x \frac{ K(z+h;x-\xi) - K(z;x-\xi)}{h} u_0^{\text{l}}(\xi) N^{\text{m}}(z+h;\xi) \D \xi\,.
  \end{align}  
  If we can show that the right-hand side converges to
  \begin{align}\label{eq:rhs-diff}
 \int_{-\infty}^x \partial_z K(z;x-\xi) u_0^{\text{l}}(\xi) N^{\text{m}}(z;\xi) \D \xi
  \end{align}
  in $C^0((-\infty,X])$, for fixed $X \in \mathbb R$, as $h \to 0$ then we have shown that $\partial_z N^{\text m}(z;x)$ exists, and is given by
   \begin{equation}\label{eq:dz}
    \partial_z N^{\text m}(z;x) = (\mathcal{I} + \mathcal K_z)^{-1} \int_{-\infty}^x \partial_z K(z;x-\xi) u_0^{\text{l}}(\xi) N^{\text{m}}(z;\xi) \D \xi.
  \end{equation}
  To establish this, we proceed as before. Fix $X\in\mathbb{R}$, $x\leq X$, and consider the difference
\begin{equation}
 \int_{-\infty}^x \partial_z K(z;x-\xi) u_0^{\text{l}}(\xi) N^{\text{m}}(z;\xi) \D \xi - \int_{-\infty}^x \partial_z K(z;x-\xi) u_0^{\text{l}}(\xi) N^{\text{m}}(z;\xi) \D \xi
\end{equation}
  whose modulus is bounded above by
  \begin{equation}
  I_1(x)\defeq \int_{-\infty}^X \left | \left(\frac{ K(z+h;x-\xi) - K(z;x-\xi)}{h} - \partial_z K(z;x-\xi) \right) u_0^{\text{l}}(\xi) N_h^{\text{m}}(z+h;\xi)\right| \D \xi.
  \end{equation}
Using the bound 
\begin{equation}
    \frac{ |K(z+h;x-\xi) - K(z;x-\xi)|}{|h|} \leq C_1 |x-\xi|^{2},
\end{equation}
for $h\neq 0$ and the fact that $u_0 \in L^1((1 + |x|)^2 \D x)$, for any $\epsilon >0$ there exists $\ell = \ell(\epsilon) \leq X$ such that
\begin{equation}
\int_{-\infty}^\ell \left | \left(\frac{ K(z+h;x-\xi) - K(z;x-\xi)}{h} - \partial_z K(z;x-\xi) \right) u_0^{\text{l}}(\xi) N_h^{\text{m}}(z+h;\xi)\right| \D \xi < \epsilon
\end{equation}
and hence
\begin{equation}
I_1(x) \leq \epsilon +  \int_{\ell}^X \left | \left(\frac{ K(z+h;x-\xi) - K(z;x-\xi)}{h} - \partial_z K(z;x-\xi) \right) u_0^{\text{l}}(\xi) N_h^{\text{m}}(z+h;\xi)\right| \D \xi.
\end{equation}
Multiplying and dividing by the factor $(1 + |\xi|)^2$ inside the integral, using $u_0\in L^1((1+|x|)^2 \D x)$ and boundedness of $N^\text{m}(z;x)$ for $x\in\mathbb{R}$, it now remains to show that
\begin{equation}
\lim_{h\to 0}
    \sup_{\ell \leq \xi \leq x \leq X} \frac{1}{(1 + |\xi|)^2}\left | \frac{K(z+h;x-\xi) - K(z;x-\xi)}{h} - \partial_z K(z;x-\xi) \right| = 0.
\end{equation}
To this end, we set $s = x -\xi>0$, and observe that
\begin{equation}
  \begin{aligned}
    \frac{K(z+h;s) - K(z;s)}{h} - \partial_z K(z;s) &= s\left(\frac{\kappa((z+h)s) - \kappa(zs)}{h} - \kappa'(zs)s \right) \\
    & = \frac{s^2}{h} \int_{z}^{z+h}(\kappa'(\tau s)-\kappa'(zs))\D \tau,
  \end{aligned}
  \end{equation}
  and since $z\leq \tau \leq z+h$, by the Mean Value Theorem $\kappa'(\tau s) = \kappa'(zs) + \kappa''(\tau_0)(\tau s-zs)$ for some $\tau_0\in(zs,\tau s)$. Then, since $\kappa''$ is bounded on $\mathbb{R}$, say, by $L\in\mathbb{R}$, we have
  \begin{equation}
    \left|\frac{s^2}{h} \int_{z}^{z+h}(\kappa'(\tau s)-\kappa'(zs))\D \tau  \right|=\left| \frac{s^3}{h} \int_{z}^{z+h}\kappa''(\xi)(\tau -z)\D \tau \right| \leq L\frac{|s|^3}{|h|} \int_{z}^{z+h} |\tau -z|\D \tau  = \frac{L}{2} |s|^3|h|.
  \end{equation} 
Therefore $I_1(x) \to 0$ as $h \to 0$, and we have indeed shown that the right-hand side of \eqref{eq:diffq} converges in $C^0((-\infty,X])$, {implying that $\partial_z N^{\text m}(z;x)$ exists and is given by \eqref{eq:dz}}.  We also, then note that for fixed $x$, \eqref{eq:dz} is continuous as a function of $u_0 \in  L^1(\mathbb{R},(1+|x|)^2\D x)$ and $z \in \mathbb R$ because $\mathcal K_z$, as an operator on $C^0((-\infty,X])$, is continuous as a function of these same variables, and \eqref{eq:rhs-diff}, as an element of $C^0((-\infty,X])$ is then continuous as a function of $u_0 \in L^1(\mathbb{R},(1+|x|)^2\D x)$ and $z \in \mathbb R$.

We then can proceed as before, to show that $\partial_z N^{\text m}(z;x)$ is (uniformly) continuous and then show that $\partial^2_z N^{\text m}(z;x)$ exists and is uniformly continuous if $L^1(\mathbb{R},(1+|x|)^3 \D x)$.  Higher derivatives follow, inductively, in a similar manner because all derivatives of $\kappa$ with respect to $s$ are bounded.

  \end{proof}

\section{Two Riemann--Hilbert problems}
\label{sec:KdV-RHPs}

In this section we assume that $a(z) \neq 0$ for $z \in \overline{\mathbb C^+}$ (hence there are no solitons in the solution of the Cauchy problem), and relax this assumption in the following sections. See the notational remark at the end of Section~\ref{sec:KdV-Intro} for the notational conventions.

We continue with some basic definitions for Riemann--Hilbert problems.  The following sequence of definitions can essentially be found in \cite{TrogdonSOBook}.
\begin{definition}  
  \begin{enumerate}
  \item As a point of reference, we first define the classical Hardy spaces on the upper- and lower-half planes.  The Hardy spaces $H^2(\mathbb C^\pm)$ consists of analytic functions $f: \mathbb C^\pm \to \mathbb C$ which satisfy the estimate
  \begin{equation}
    \sup_{r > 0} \| f(\cdot \pm \I r) \|_{L^2(\mathbb R)} < \infty.
  \end{equation}
\item   $\Gamma \subset \mathbb C$ is said to be an admissible contour if it is finite union of oriented, differentiable curves $\Gamma = \Gamma_1 \cup \cdots \cup \Gamma_k$, called component contours, which intersect only at their endpoints and tend to straight lines at infinity, the connected components of $\mathbb C \setminus \Gamma$ can be grouped into two classes $C_+$ and $C_-$ such that for $\Omega_1,\Omega_2 \in C_\pm$ the arclength of $\partial \Omega_1 \cap \partial \Omega_2$ is zero, and $- \Gamma = \{-s : s \in \Gamma\} = \Gamma$ with a reversal of orientation.
\item  For a connected component $\Omega \subset \mathbb C \setminus \Gamma$, the class $\mathcal E^2(\Omega)$ is defined to be the set of all analytic functions $f$ in $\Omega$ such that there exists a sequence of curves $(\gamma_n)_{n \geq 1}$ in $\Omega$ satisfying
  \begin{equation}
    \sup_n \int_{\gamma_n} \frac{|\D s|}{|s-a|^2} < \infty, \quad \text{for some}\quad a \in \mathbb C \setminus \overline{\Omega},
  \end{equation}
  that tend to $\partial \Omega$ in the sense that $\gamma_n$ eventually surrounds every compact subset of $\Omega$ such that
  \begin{equation}
    \sup_{n} \int_{\gamma_n} |f(s)|^2 |\D s| < \infty.
  \end{equation}
\item For an admissible contour $\Gamma$, define the Hardy space $H_\pm(\Gamma)$ to be the class of all analytic functions $f: \mathbb C \setminus \Gamma \to \mathbb C$ such that $f|_{\Omega} \in \mathcal E^2(\Omega)$ for every connected component $\Omega$ of $\mathbb C \setminus \Gamma$.  This is a generalization of (1).  We also use the notation $H_\pm^2(\Gamma)$ if just modification of the orientations of the component contours make $\Gamma$ admissible.
  \end{enumerate}
\end{definition}
For $f \in L^2(\Gamma)$, define the Cauchy integral
\begin{equation}
  \mathcal C_{\Gamma} f(z) = \frac{1}{2 \pi \I} \int_{\Gamma} \frac{f(s)}{s-z} \D s, \quad z \not\in \Gamma.
\end{equation}
We have the following standard facts.
\begin{enumerate}
\item From standard theory (see, \cite{TrogdonSOBook}, for example) it follows that $\mathcal C_{\Gamma} : L^2(\Gamma) \to H^2_\pm(\Gamma)$.
\item Furthermore, the Cauchy operator $\mathcal{C}_\Gamma$ maps $L^2(\mathbb R)$ onto $H^2_\pm(\Gamma)$, and therefore every function $f \in H^2_\pm(\Gamma)$ has two $L^2(\Gamma)$ boundary values on $\Gamma$, one taken from $C_+$ and the other taken from $C_-$.  We use $\mathcal C^\pm_\Gamma f(s)$ to denote these boundary values, and note the identity that $\mathcal C^+_\Gamma f(s) - \mathcal C^-_\Gamma f(s) = f(s)$ for a.e. $s \in \Gamma$.
\item The last fact we need is that $\mathcal C^\pm_\Gamma$ are bounded operators on $L^2(\Gamma)$ if $\Gamma$ is admissible\footnote{The necessary and sufficient condition is that $\Gamma$ is a \emph{Carleson curve} \cite{Bottcher1997}.}.
\end{enumerate}

\begin{definition}
  An $L^2$ solution $\mb N$ to an RH problem on an admissible contour $\Gamma$
  \begin{equation}
    \mb N^+(s) = \mb N^-(s) \mb J(s), \quad s \in  \Gamma, \quad \mb N(z) = \begin{bmatrix} 1 &  1 \end{bmatrix} + O(z^{-1}),
  \end{equation}
  is a solution $\mb N(\cdot) - \begin{bmatrix} 1 &  1 \end{bmatrix} \in H_{\pm}^2(\Gamma)$ such that $\mb N^+(s) = \mb N^-(s) \mathbf{J}(s)$ is satisfied for a.e. $s \in\Gamma$.
\end{definition}
Note that an $L^2$ solution does not necessarily satisfy the uniform $O(z^{-1})$ condition at infinity.

\subsection{Left Riemann-Hilbert Problem}

We use the scattering relation, combined with another equation, for $z \in \mathbb R$ and $\lambda(z) \in \mathbb R$,
\begin{align}\label{eq:ab}
  \begin{split}
  \psi^{\text{p}}(z;x) &=  a(z) \phi^{\text{p}}(z;x)  + b(z) \phi^{\text{m}}(z;x)\\
  \psi^{\text{m}}(z;x) &=  \hat b(z) \phi^{\text{p}}(z;x)  + \hat a(z) \phi^{\text{m}}(z;x).
\end{split}
\end{align}
These two equations are used to formulate jump conditions for a sectionally analytic function.
\begin{remark}
To deduce properties of $a,b,\hat a,\hat b$ we need only evaluate this relation at $x = 0$ and recall Remark~\ref{r:schr} and apply Lemma~\ref{l:diff} at $x = 0$, for example.
\end{remark}

\subsubsection{Jump relation for $s^2 > c^2$}

First, note that $\psi^{\text{m}}(-z;x) = \psi^{\text{p}}(z;x)$ since $\lambda(z)$ is odd for $z \not\in (-c,c)$.  Additionally, there is a conjugate symmetry because $u_0$ is real-valued:  $ \overline{\psi^{\text{m}}(z;x)} = \psi^{\text{p}}(z;x)$, and  $\psi^{\text{p/m}}(z;x)$ enjoys the same symmetry.  Thus, we find that $\overline{b(z)} = b(-z) = \hat b(z)$ and $\overline{a(z)} = a(-z) = \hat a(z)$.  We also know that $\phi^{\mathrm{p}}$ and $\psi^{\mathrm{m}}$ are analytic functions of $z$ in the lower-half plane while the others, $\phi^{\mathrm{m}}$ and $\psi^{\mathrm{p}}$,  are analytic in the upper-half plane.  Define the sectionally-analytic function
\begin{equation}
  \mb L_1(z) = \mb L_1(z;x) := \begin{cases} \begin{bmatrix}  \psi^{\mathrm{p}}(z;x) & \phi^{\mathrm{m}}(z;x) \end{bmatrix} & \Im z > 0, \\
    \\
     \begin{bmatrix} \phi^{\mathrm{p}}(z;x) & \psi^{\mathrm{m}}(z;x) \end{bmatrix} & \Im z < 0.
    \end{cases}
\end{equation}
Then assuming that $a(z) \neq 0$ for $\Im z \geq 0$, we have for $s^2 \geq c^2$
\begin{equation}
\begin{aligned}
  \mb L_1^+(s) & = \begin{bmatrix}  \psi^{\mathrm{p}}(s;x) & \phi^{\mathrm{m}}(s;x) \end{bmatrix} \\
           & = \begin{bmatrix} \left[1 - \frac{b(s)}{a(s)} \frac{b(-s)}{a(-s)} \right]  \phi^{\text{p}}(s;x) + \frac{b(s)}{a(s)} \frac{1}{a(-s)} \psi^{\mathrm{m}}(s;x) & \frac{1}{a(-s)} \psi^{\text{m}}(s;x) -   \frac{b(-s)}{a(-s)} \phi^{\text{p}}(s;x)  \end{bmatrix} \begin{bmatrix} a(s) & 0 \\ 0 & 1 \end{bmatrix}\\
  & = \mb L_1^-(s) \begin{bmatrix} 1 & 0 \\ 0 & \frac{1}{a(-s)} \end{bmatrix} \begin{bmatrix} 1 - |R_{\mathrm{l}}(s)|^2 & - R_{\mathrm{l}}(-s) \\ R_{\mathrm{l}}(s) & 1  \end{bmatrix} \begin{bmatrix} a(s) & 0 \\ 0 & 1 \end{bmatrix}.
  \end{aligned}
\end{equation}
We now define
\begin{equation}
  \mb K_1(z) = \begin{cases} \mb L_1(z) \begin{bmatrix} \frac{1}{a(z)} & 0 \\ 0 & 1 \end{bmatrix} & \Im z > 0, \\
    \\
     \mb L_1(z)\begin{bmatrix} 1 & 0 \\ 0 & \frac{1}{a(-z)} \end{bmatrix}  & \Im z < 0,
    \end{cases}
\end{equation}
which is analytic on $\mathbb C \setminus \mathbb R$ and satisfies
\begin{align}
  \mb K^+_1(s) = \mb K^-_1(s) \begin{bmatrix} 1 - |R_{\mathrm{l}}(s)|^2 & - R_{\mathrm{l}}(-s) \\ R_{\mathrm{l}}(s) & 1  \end{bmatrix}, \quad s^2 \geq c^2.
  \label{eq:K-jump}
\end{align}

\subsubsection{Jump relation for $-c \leq s \leq c$}
We find that for $-c \leq s \leq c$
\begin{align} \label{eq:psi}
  \psi(s;x) : = \lim_{\epsilon \downarrow 0} \psi^{\mathrm{p}}(s + \I \epsilon; x) = \lim_{\epsilon \downarrow 0} \psi^{\mathrm{m}}(s - \I \epsilon; x).
\end{align}
Then, again for $-c \leq s \leq c$ we define $\tilde a(s)$ and $\tilde b(s)$ by
\begin{align}\label{eq:psi=}
  \psi(s;x) &=  \tilde a(s) \phi^{\text{p}}(s;x)  + \tilde b(s) \phi^{\text{m}}(s;x),
\end{align}
because $\psi$ is a solution of \eqref{eq:spec1}. From this, it follows that
\begin{equation}
  \mb K_1^+(s) = \mb K_1^-(s) \begin{bmatrix} 0 & -\frac{\tilde a(s)}{\tilde b(s)} \\ \frac{a^+(-s)}{a^+(s)} & \frac{a^+(-s)}{\tilde b(s)} \end{bmatrix}, \quad -c \leq s \leq c.
  \label{eq:K-jump-cut}
\end{equation}
But then we solve for $\tilde b$ and $\tilde a$ to find
\begin{equation}
\begin{aligned}
\tilde b(s) & =  \frac{W(\psi(s;\cdot), \phi^{\text{p}}(s;\cdot))}{2 \I s},\\
 \tilde a (s) &=  \frac{W(\phi^{\text{m}}(s;\cdot), \psi(s;\cdot))}{2 \I s}.
\end{aligned}
\end{equation}
Since both $\psi$ and $\phi^{\text{p}}$ have analytic continuations for $\Im z < 0$, it follows that $\tilde b(z)$ has an analytic continuation for $\Im z < 0$.  And then for $\Im z > 0$
\begin{equation}
  \tilde b(-z) =  -\frac{W(\psi^{\text{m}}(-z;\cdot), \phi^{\text{p}}(-z;\cdot))}{2 \I z} = \frac{W(\phi^{\text{m}}(z;\cdot), \psi^{\text{p}}(z;\cdot))}{2 \I z} = a(z).
\end{equation}
This implies that $a^+(s) = {\tilde b}^-(-s) = \tilde b(-s)$.  It also follows that $\tilde a(z) = a(z)$ for $\Im z > 0$ so that $ \tilde a^+(s) = a^+(s)$.  So,
\begin{equation}
  \mb K_1^+(s) = \mb K_1^-(s) \begin{bmatrix} 0 & -\frac{a^+(s)}{a^+(-s)} \\ \frac{a^+(-s)}{a^+(s)} & 1 \end{bmatrix}, \quad -c \leq s \leq c.
\end{equation}
To finish the setup of the RH problem we extend the definition of $R_{\mathrm l}$ as
\begin{equation}\label{eq:l-oncut}
  R_{\mathrm{l}}(s) = \begin{cases} \frac{b(s)}{a(s)} & s^2 > c^2,\\
    \frac{a^+(-s)}{a^+(s)} & -c \leq s \leq c, \end{cases}
\end{equation}
and define
\begin{equation}
\mb N_1(z) = \mb K_1(z) \E^{- \I x z \sigma_3}, \quad z \not \in \mathbb R.
\label{eq:N1-def}
\end{equation}
\begin{remark}
  This definition of $R_{\mathrm{l}}(s)$ for $-c \leq s \leq c$ can be justified by noting that if $u_0$ decays exponentially so that $\hat\psi^{\text{p/m}}$ and $\phi^{\text{p/m}}$ have analytic extensions to a strip containing the real axis then, $b(z)$ has an extension to a set $ (B\setminus [-c,c]) \cap \mathbb C^+$ where $[-c,c] \subset B$, $B$ is open, and $R_{\mathrm{l}}(s) = \frac{b^+(s)}{a^+(s)}$ for $-c \leq s \leq c$.
\end{remark}

\begin{theorem}[\!\cite{Deift1979,Cohen1985}] For all $x \in \mathbb R$, componentwise, we have
  \begin{equation}
    \mb N_1(\cdot) - \begin{bmatrix} 1 & 1 \end{bmatrix} \in H^2_\pm(\mathbb R).
  \end{equation}
\end{theorem}

Using the jump conditions \eqref{eq:K-jump} and \eqref{eq:K-jump-cut} satisfied by $\mathbf{K}_1$ and the extension of $R_\mathrm{l}$ given in \eqref{eq:l-oncut} we have arrived at the following RH problem satisfied by $\mathbf{N}_1$. 

\begin{rhp}\label{rhp:1}
  The function $\mb N_1: \mathbb C \setminus \mathbb R \to \mathbb C^{1\times 2}$ is analytic on its domain and satisfies
  \begin{equation}
    \mb N_1^+(s) = \mb N_1^-(s) \begin{bmatrix} 1 - |R_{\mathrm{l}}(s)|^2  & -\overline{R_{\mathrm{l}}(s)} \E^{2 \I s x}  \\ {R_{\mathrm{l}}}(s) \E^{-2 \I s x} & 1 \end{bmatrix}, \quad s \in \mathbb R, \quad \mb N_1(z) = \begin{bmatrix} 1 & 1 \end{bmatrix} + O(z^{-1}), \quad z \in \mathbb C \setminus \mathbb R,
  \end{equation}
  with the symmetry condition
  \begin{equation}
    \mb N_1(-z) = \mb N_1(z) \sgo, \quad z \in \mathbb C \setminus \mathbb R.
  \end{equation}
\end{rhp}

\begin{remark}
  While setting up this RH problem one verifies that
  \begin{equation}
    \lim_{|z| \to \infty, ~ \Im z > 0} z\left( \mb N_1(z) -\begin{bmatrix} 1 & 1 \end{bmatrix} \right) = \lim_{|z| \to \infty, ~ \Im z < 0} z\left( \mb N_1(z) -\begin{bmatrix} 1 & 1 \end{bmatrix} \right),
  \end{equation}
using Lemmas~\ref{l:J-largez} and \ref{l:a-largez}.  This is a necessary condition for the solution of a singular integral equation we pose below in \eqref{eq:numer-op} to have an integrable solution.  We are  purposefully vague about in what sense the limits $\mb N_1^{\pm}$ exist as this is made precise below.
\end{remark}

%

  \subsection{Right Riemann-Hilbert Problem}
 We now use the other scattering relation, combined with yet another equation, for $z \in \mathbb R, \lambda(z) \in \mathbb R$,
\begin{equation}
\begin{aligned}
  \phi^{\text{m}}(z;x) &=  B(z) \psi^{\text{p}}(z;x)  + A(z) \psi^{\text{m}}(z;x)\\
  \phi^{\text{p}}(z;x) &=  \hat A(z) \psi^{\text{p}}(z;x)  + \hat B(z) \psi^{\text{m}}(z;x).
\end{aligned}
\end{equation}
We find that $\overline{B(z)} = B(-z) = \hat B(z)$ and $\overline{A(z)} = A(-z) = \hat A(z)$ and $A(z)$ has an analytic continuation into the upper-half plane.  This is now used to determine the jump relations for another sectionally analytic function.

\subsubsection{Jump relation for $s^2 > c^2$}
Define the sectionally-analytic function
\begin{equation}
  \mb L_2(z) = \mb L_2(z;x) := \begin{cases} \begin{bmatrix} \phi^{\mathrm{m}}(z;x)  & \psi^{\mathrm{p}}(z;x) \end{bmatrix} & \Im z > 0, \\
    \\
     \begin{bmatrix}  \psi^{\mathrm{m}}(z;x) & \phi^{\mathrm{p}}(z;x)\end{bmatrix} & \Im z < 0.
    \end{cases}
\end{equation}
Then assuming that $A(z) \neq 0$ for $\Im z \geq 0$, we have for $s^2 \geq c^2$
\begin{equation}
\begin{aligned}
  \mb L_2^+(s) & = \begin{bmatrix}   \phi^{\mathrm{m}}(s;x) & \psi^{\mathrm{p}}(s;x) \end{bmatrix} \\
                   &= \begin{bmatrix}    \left[ A(s) - \frac{B(s) \hat B(s)}{\hat A(s)} \right] \psi^{\text{m}}(s;x) +  \frac{B(s)}{\hat A(s)} \phi^{\text{p}}(s;x) & \frac{\phi^{\text{p}}(s;x)}{\hat A(s)} - \frac{\hat B(s)}{\hat A(s)} \psi^{\text{m}}(s;x) \end{bmatrix} \\
                   &= \mb L_2^-(s) \begin{bmatrix} 1 & 0 \\ 0 & \frac{1}{\hat A(s)} \end{bmatrix} \begin{bmatrix}  1 - \frac{B(s) \hat B(s)}{A(s) \hat A(s)} & -\frac{\hat B(s)}{\hat A(s)} \\ \frac{B(s)}{A(s)} & 1 \end{bmatrix} \begin{bmatrix} A(s) & 0 \\ 0 & 1 \end{bmatrix}. \\
\end{aligned}
\end{equation}
In a similar way as above define
\begin{equation}
  \mb K_2(z) = \begin{cases} \mb L_2(z) \begin{bmatrix} \frac{1}{A(z)} & 0 \\ 0 & 1 \end{bmatrix} & \Im z > 0,\\
  \\
  \mb L_2(z) \begin{bmatrix} 1 & 0 \\ 0 & \frac{1}{\hat A(z)} \end{bmatrix} & \Im z < 0. \end{cases}
\end{equation}
so that for $s^2 \geq c^2$ we have the jump relation
\begin{align}
\mb  K^+_2(s) = \mb K_2^-(s) \begin{bmatrix} 1 - |R_{\mathrm{r}}(s)|^2 & - R_{\mathrm{r}}(-s) \\ R_{\mathrm{r}}(s) & 1 \end{bmatrix}.
\label{eq:K2-jump}
\end{align}

\subsubsection{Jump relation for $-c \leq s \leq c$}

For $-c \leq s \leq c$, the second entry of $\mb L_2^+(s)$ is equal to the first entry of $\mb L^-_2(s)$.  From \eqref{eq:psi=} it follows that
\begin{equation}
  \phi^{\text{m}}(s;x) = \frac{1}{a^+(-s)} \psi(s;x) - \frac{a^+(s)}{a^+(-s)} \phi^{\text{p}}(s;x).
\end{equation}
For $-c \leq s \leq c$ this gives the relations
\begin{equation}
\begin{aligned}
  \mb L_2^+(s) &= \mb L_2^-(s) \begin{bmatrix} \frac{1}{a^+(-s)} & 1 \\  - \frac{a^+(s)}{a^+(-s)} & 0 \end{bmatrix}                      
\end{aligned}
\end{equation}
which implies
\begin{equation}
\begin{aligned}
  \mb K_2^+(s) &= \mb K_2^-(s) \begin{bmatrix} \frac{1}{a^+(-s) A^+(s) } & 1 \\  - \frac{a^+(s)  A^+(-s) }{a^+(-s) A^+(s)} & 0 \end{bmatrix},                                                            
\end{aligned}
\end{equation}                                                                                                
and then 
\begin{equation}
\begin{aligned}
  \mb K_2^+(s) &= \mb K_2^-(s) \begin{bmatrix} \frac{1}{a^+(-s) A^+(s) } & 1 \\  1 & 0 \end{bmatrix}.                                                 
\end{aligned}
\end{equation}

The definition of $R_{\mathrm r}(s)$ for $-c \leq s \leq c$ is much more complicated that for $R_{\mathrm l}(s)$. We first establish an identity involving $R_{\mathrm r}(s)$ and $a^+(s), A^+(s)$ under the assumption that $u_0(x)$ decays exponentially as $x\to\pm\infty$ implying that there exists neighborhoods of $V_c$, $V_{-c}$ of $c$ and $-c$, respectively, such that $R_{\mathrm r}(s)$ has an analytic extension to $V_{\pm c} \setminus [-c,c]$. For $\epsilon > 0$ sufficiently small, and for $-c \leq s \leq -c + \epsilon$ we claim
\begin{align}\label{eq:claimR}
  \lim_{\epsilon \downarrow 0} R_{\mathrm r}(s + \I \epsilon) - \lim_{\epsilon \downarrow 0} R_{\mathrm r}( - (s - \I \epsilon) ) = \frac{1}{a^+(-s) A^+(s)}.
\end{align}
The left-hand side is equal to
\begin{align}
  R_{\mathrm r}^+(s) - R_{\mathrm r}^+(-s) = \frac{B^+(s) A^+(-s) - B^+(-s)A^+(s)}{A^+(-s)A^+(s)}.
\end{align}
We then use $a^+(s) = \frac{s}{\lambda^+(s)} A^+(s)$ to write
\begin{equation}
R_{\mathrm r}^+(s) - R_{\mathrm r}^+(-s) = -\frac{\lambda^+(s)}{s} \frac{B^+(s) A^+(-s) - B^+(-s)A^+(s)}{a^+(-s)A^+(s)}.
\end{equation}
From the Wronskian representations we obtain $A^+(-s) = -\frac{s}{\lambda^+(s)} b^+(s)$ from which it follows that
\begin{equation}
R_{\mathrm r}^+(s) - R_{\mathrm r}^+(-s) =  \frac{B^+(s) b^+(s) + B^+(-s)b^+(-s)}{a^+(-s)A^+(s)}.
\end{equation}
Then working with Wronskians for functions $f,g,h,k$ we find by brute force
\begin{equation}
  W(f,h)W(g,k) - W(g,h)W(f,k) = W(f,g)W(h,k).
\end{equation}
Then using that the boundary values from above of $\psi^{\mathrm{p/m}}$ are even in $s$ and $\phi^{\mathrm{p}}(-s;x) = \phi^{\mathrm{m}}(s;x)$, we find that $B^+(s) b^+(s) + B^+(-s)b^+(-s) = 1$ and the claim \eqref{eq:claimR}  follows. Then, compute
\begin{equation}
  \begin{bmatrix} 1 & -R_{\mathrm r}^+(-s) \\ 0 & 1 \end{bmatrix} \begin{bmatrix} 0 & 1 \\ 1 & 0 \end{bmatrix} \begin{bmatrix} 1 &  0 \\ R_{\mathrm r}^+(s) & 1 \end{bmatrix}  = \begin{bmatrix} \frac{1}{a^+(-s) A^+(s) } & 1 \\  1 & 0 \end{bmatrix}.
\end{equation}
Note that $R_{\mathrm r}^+(-s)$ can be extended to an open set in the lower-half plane, $R_{\mathrm r}^+(s)$ can be extended to an open set in the upper-half plane.  The same factorization holds near $c$ on $[c- \epsilon, c]$.

Removing the assumption of exponential decay of $u_0$, but keeping the condition $u_0 \in L^1(\mathbb{R},(1+|x|)^3 \D x)$, we extend the definition of $R_{\mathrm r}$ to $[-c,c]$ so that it has an approximate analytic extension.  This extension is given by 
\begin{align}
  R_{\mathrm r}(s) = \begin{cases} \frac{B(s)}{A(s)} & s^2 \geq c^2,\\
    \frac{\frac{1}{2} + \frac{\ell(s)}{s \sqrt{c^2-s^2}}}{a^+(-s) A^+(s)} & -c + \epsilon \leq s \leq c - \epsilon,\\
    \frac{B^+(s)}{A^+(s)} &s \in (-c,-c + \epsilon) \cup (c-\epsilon,c),\end{cases} \label{eq:r-oncut}
\end{align}
where $\ell(s)$ is an even function of $s$ on $(-c,c)$.  The intent of this definition is for it to make sense even when $\epsilon = 0$, and the third case never applies.  Furthermore, we have
\begin{align}
R_{\mathrm r}(s) - R_{\mathrm r}(-s) = \frac{1}{a^+(-s) A^+(s)}, \quad s \in (-c,c).
\end{align}
Now choose $\ell$ to match the behavior of $R_{\mathrm r}$ at $-c$ in the following way.  Set
\begin{equation}
  \ell(s) = \alpha s^2 + \beta \sqrt{c^2-s^2},
\end{equation}
and assume
\begin{equation}
  \frac{1}{a^+(-s)A^+(s)} = \kappa_1 \sqrt{s+c} + \kappa_2 (s+c) + O(|s+c|^{3/2})
\end{equation}
as $s \to -c$, $s > -c$.  Such an expansion is valid by Lemma~\ref{l:diff}.  We find
\begin{equation}
\frac{\frac{1}{2} + \frac{\ell(s)}{s \sqrt{c^2-s^2}}}{a^+(-s) A^+(s)} = - \kappa_1 \frac{\alpha \sqrt{c}}{\sqrt{2}} + \left( \frac{\kappa_1}{2} - \kappa_1 \frac{\beta}{c}  - \kappa_2 \frac{\alpha \sqrt{c}}{\sqrt{2}} \right)\sqrt{s+c} + O(|s+c|), \quad s \to -c, ~ s > -c.
\end{equation}
We choose $\alpha$ so that $\kappa_1 \frac{\alpha \sqrt{c}}{\sqrt{2}} = 1$ and choose $\beta$ so that $\frac{\kappa_1}{2} - \kappa_1 \frac{\beta}{c}  - \kappa_2 \frac{\alpha \sqrt{c}}{\sqrt{2}} = - \I \gamma$
where $\gamma$ is determined by
\begin{equation}
  R_{\mathrm r}(s) = -1 + \gamma \sqrt{-s-c} + O(|s+c|), \quad s \to -c, ~~s < -c.
\end{equation}
This process succeeds because $\kappa_1 \neq 0$.  This implies that\footnote{The fact that $\lim_{s \to -c, ~s < -c} R_{\mathrm r}(s) = -1$ is established in Theorem~\ref{t:smooth} below directly from a ratio of Wronskians.}
\begin{align}\label{eq:match}
  \begin{split}
  R_{\mathrm r}(s) &= -1 + \gamma g(s) + O(|s+c|), \quad s \to -c, ~~s < -c,\\
  \frac{\frac{1}{2} + \frac{\ell(s)}{s \sqrt{c^2-s^2}}}{a^+(-s) A^+(s)} &= -1 + \gamma g_+(s) + O(|s+c|), \quad s \to -c, ~~s > -c,
  \end{split}
\end{align}
where $g(z) = \sqrt{-z-c}$ has an analytic extension to the upper-half plane, using the principal branch of the square root.

Then, as a consequence of $R_{\mathrm r}(-s) = \overline{R_{\mathrm r}(s)}$, $s^2 > c^2$ and the fact that $a^+(-s)A^+(s)$ is an odd function of $s$, we have
\begin{equation}
\begin{aligned}
  R_{\mathrm r}(s) &= -1 + \bar{\gamma} \sqrt{s-c} + O(|s-c|), \quad s \to c, ~~s > c,\\
  \frac{1}{a^+(-s)A^+(s)} &= -\kappa_1 \sqrt{c-s} + -\kappa_2 (c-s) + O(|s-c|^{3/2}), \quad s \to c, ~~s < c,
\end{aligned}
\end{equation}
and therefore
\begin{equation}\label{eq:match2}
  \begin{split}
  \frac{\frac{1}{2} + \frac{\ell(s)}{s \sqrt{c^2-s^2}}}{a^+(-s) A^+(s)} &= - \kappa_1 \frac{\alpha \sqrt{c}}{\sqrt{2}} + \left( -\frac{\kappa_1}{2} - \kappa_1 \frac{\beta}{c}  - \kappa_2 \frac{\alpha \sqrt{c}}{\sqrt{2}} \right)\sqrt{c-s} + O(|c-s|), \quad s \to c, ~ s < c,\\
  &= - 1 + \left( - \I \gamma - \kappa_1 \right)\sqrt{c-s} + O(|c-s|), \quad s \to c, ~ s < c,\\
  &= - 1 + \I \bar \gamma \sqrt{c-s} + O(|c-s|), \quad s \to c, ~ s < c,
  \end{split}
\end{equation}
because the following lemma holds.

\begin{lemma}
If $u_0 \in L^1(\mathbb{R},(1+|x|)^2 \D x)$, $-\I \gamma - \kappa_1 = \I \bar{\gamma}$.
\end{lemma}
\begin{proof}
  If the initial condition has compact support, we have local analytic continuations of $R_{\mathrm r}$ to the upper-half plane in the neighborhood of $\pm c$ and therefore using that $R_{\mathrm r}(-z) = \overline{ R_{\mathrm r}(z)}$
  \begin{equation}
\begin{aligned}
    R_{\mathrm r}^+(s) &= -1 -\I \gamma \sqrt{s + c} + O(|s+c|), \quad s \to -c, ~~ s > -c.\\
    R_{\mathrm r}^+(s) &= -1 + \I \overline{\gamma} \sqrt{c-s} + O(|s-c|), \quad s \to c, ~~ s < c.
\end{aligned}
\end{equation}
  The identity
  \begin{equation}
     R_{\mathrm r}^+(s) - R_{\mathrm r}^+(-s) = \frac{1}{a^+(-s)A^+(s)},
  \end{equation}
  establishes the claim for initial data with compact support.  For general data, we approximate it in $L^1(\mathbb{R},(1+|x|)^2 \D x)$ with data having compact support and then Lemma~\ref{l:diff} implies the claim in the limit because $\gamma$ and $\kappa_1$ are continuous as functions on $L^1(\mathbb{R},(1+|x|)^2 \D x)$.
  
\end{proof}



\begin{remark}
  The definition of $R_{\mathrm r}$ on $[-c + \epsilon, c - \epsilon]$ can be modified, assuming $u_0 \in L^1(\mathbb{R},(1+|x|)^{k + 1} \D x)$, so that more terms in its series expansion at $\pm c$ match from the left and right.
\end{remark}

We finally define
\begin{equation}
  \mb N_2(z) = \mb K_2(z) \E^{\I \lambda(z) x \sigma_3}
\end{equation}
and arrive at the following problem satisfied by $\mathbf{N}_2$.

\begin{rhp} \label{rhp:2} The function $\mb N_2: \mathbb C \setminus \mathbb R \to \mathbb C^{1\times 2}$ is analytic on its domain and satisfies
  \begin{align*}
  \mb N_2^+(s) &= \mb N_2^-(s) \begin{bmatrix} 1 - |R_{\mathrm{r}}(s)|^2 & - R_{\mathrm{r}}(-s) \E^{ -2 \I  \lambda(s) x} \\ R_{\mathrm{r}}(s) \E^{2 \I \lambda(s) x} & 1 \end{bmatrix}, \quad s^2 > c^2,\\
    \mb N_2^+(s) &= \mb N_2^-(s) \begin{bmatrix} \displaystyle \frac{\E^{2 \I \lambda^+(s) x}}{a^+(-s) A^+(s) } & 1 \\  1 & 0 \end{bmatrix} = \mb N_2^-(s) \begin{bmatrix} 1 & -R_{\mathrm r}(-s) \E^{-2 \I \lambda^-(s) x}  \\ 0 & 1 \end{bmatrix}  \sgo  \begin{bmatrix} 1 &  0 \\ R_{\mathrm r}(s) \E^{2 \I \lambda^+(s) x }& 1 \end{bmatrix}, \quad -c \leq s \leq c,\\
     \mb N_2(z) &= \begin{bmatrix} 1 & 1 \end{bmatrix} + O(z^{-1}), \quad z \in \mathbb C \setminus \mathbb R,
  \end{align*}
  with the symmetry condition
  \begin{equation}
    \mb N_2(-z) = \mb N_2(z)\sgo, \quad z \in \mathbb C \setminus \mathbb R.
  \end{equation}
\end{rhp}

\subsection{Decay properties of $R_{\mathrm{l/r}}$ on $\mathbb R$}

  \begin{definition}
    Define $\mathcal D_n$, $n \geq 2$ to be the class of functions $f$ on $\mathbb R$ such that $f \in L^1(\mathbb{R},(1+|x|)\D x)$ has $n-1$ absolutely continuous derivatives in $L^1(\mathbb R)$, $f^{(n)}$ is piecewise absolutely continuous\footnote{A function $f$ is piecewise absolutely continuous on $\mathbb R$ if there exists a partition $-\infty = x_0 < x_1 < \ldots < x_N = + \infty$ such that $f|_{[x_n,x_{n+1}]}$ can be made absolutely continuous by modifying the values of $f(x_n)$ and $f(x_{n+1})$.} and in $L^1(\mathbb R)$, and $f^{(n+1)} \in L^1(\mathbb{R})$.

    If $n = 1$ define $\mathcal D_n$ to be the class of functions $f$ on $\mathbb R$ such that $f \in L^1(\mathbb{R},(1+|x|)\D x)$, $f$ is absolutely continuous and $f^{(1)}$ is piecewise absolutely continuous and in $L^1(\mathbb R)$, and $f^{(2)} \in L^1(\mathbb{R})$.

    If $n = 0$ define $\mathcal D_n$ to be the class of functions $f$ on $\mathbb R$ such that $f \in L^1(\mathbb{R},(1+|x|)\D x)$,  $f$ is piecewise absolutely continuous and $f^{(1)} \in L^1(\mathbb{R})$.
    
  \end{definition}

  \begin{lemma}[\cite{Kappeler1986}]\label{l:decay}
    For $n \geq 0$, suppose that $u_0 \in \mathcal D_n$.  Then
    \begin{equation}
      R_{\mathrm{l/r}}(s) = O(|s|^{-2-n}) \text{  as  } s \to \pm \infty.
    \end{equation}
    \end{lemma}

    \begin{remark} In \cite{Kappeler1986} the author imposes moment conditions on derivatives of $u_0$ in the proof of a more general version of Lemma~\ref{l:decay} that gives decay rates of the derivatives of the reflection coefficients. Since we only focus on the decay rate of the function itself in the present work, these conditions are unnecessary. \end{remark}

\subsection{Relations between left and right scattering data}

In some of the calculations that follow, it is convenient to have specific equalities that relate $A,B,a$ and $b$. First, consider the system \eqref{eq:ab} for $z \in \mathbb R$, $\lambda(z) \in \mathbb R$, combined with its derivative with respect to $x$
\begin{equation}   
\begin{aligned} 
  \begin{bmatrix} \psi^{\text{p}}(z;x) & \psi^{\text{m}}(z;x) \\  \psi_x^{\text{p}}(z;x) & \psi_x^{\text{m}}(z;x)\end{bmatrix}  &=  \begin{bmatrix} \phi^{\text{p}}(z;x) & \phi^{\text{m}}(z;x) \\ \phi_x^{\text{p}}(z;x) & \phi_x^{\text{m}}(z;x) \end{bmatrix} \begin{bmatrix} a(z) & b(-z) \\ b(z) & a(-z) \end{bmatrix}.
  \end{aligned}
\end{equation}
This gives
\begin{equation}
\begin{aligned}
  a(z)a(-z) -b(z) b(-z) &= \frac{W(\psi^{\text{p}}(z;\cdot), \psi^{\text{m}}(z;\cdot))}{W(\phi^{\text{p}}(z;\cdot), \phi^{\text{m}}(z;\cdot))} = \frac{\lambda(z)}{z},\\
  A(z)A(-z) -B(z) B(-z) &=  \frac{z}{\lambda(z)}.
\end{aligned}
\end{equation}
From this, one finds,
\begin{equation}\label{eq:transl}
1 - R_{\mathrm l}(z)R_{\mathrm l}(-z) = \frac{\lambda(z)}{z} \frac{1}{a(-z)a(z)} = \frac{1}{A(z)a(-z)}
\end{equation}
Next, we claim that for $z \in \mathbb R$, $\lambda(z) \in \mathbb R$
\begin{equation}\label{eq:Btob}
B(z) = - \frac{b(-z) A(-z)}{a(-z)} = - b(-z) \frac{z}{\lambda(z)}. 
\end{equation}
This follows because $\psi^{\text{p}}(-z;\cdot) = \psi^{\text{m}}(z;\cdot)$, $\phi^{\text{p}}(-z;\cdot) = \phi^{\text{m}}(z;\cdot)$, and $A(z) = a(z) \frac{z}{\lambda(z)}$.  
  
\subsection{Smoothness properties of $R_{\mathrm{l/r}}$ on $\mathbb R$}

  \begin{definition}
    The initial perturbation $u_0(x) =  u(x,0)  - H_c(x)$ is said to be generic if
    \begin{equation}
      W(\phi^{\text{m}}(c;\cdot), \psi^{\text{p}}(c;\cdot)) \neq 0 \quad\text{and}\quad W(\psi^{\text{m}}(0;\cdot), \phi^{\text{p}}(0;\cdot)) \neq 0.
    \end{equation}
  \end{definition}
  The term genericity is used because this is expected to hold on a open, dense subset of initial data \cite{Deift1979}.  We note that this fact was not established in \cite{Kappeler1986}.  We do not establish this here because we can verify it numerically in all cases we consider.  It will be considered in a future work.
  
  Genericity implies, by evaluating at $x = 0$,
  \begin{equation}
    W(\phi^{\text{m}}(c;\cdot), \hat \psi^{\text{p}}(0;\cdot)) \neq 0,
  \end{equation}
  giving
  \begin{align}\label{eq:Wron}
    0 \neq \overline{W(\phi^{\text{m}}(c;\cdot), \hat \psi^{\text{p}}(0;\cdot))} = W(\phi^{\text{p}}(c;\cdot), \hat \psi^{\text{m}}(0;\cdot)) = W(\phi^{\text{m}}(- c;\cdot), \hat \psi^{\text{p}}(0;\cdot)).
  \end{align}
  Next, by again evaluating at $x = 0$,
  \begin{align}\label{eq:Wron2}
   0 \neq W(\psi^{\text{m}}(0;\cdot), \phi^{\text{p}}(0;\cdot)) =   W(\hat \psi^{\text{m}}(c;\cdot), \phi^{\text{p}}(0;\cdot)).
  \end{align}
  Here $\hat \psi^{\text{m}}$ and $\phi^{\text{p}}$ are solutions of the same Schr\"odinger equation with decaying potential $u_0(x)$.  We find that \eqref{eq:Wron2} with $u_0(x)$ replaced with $u_0(-x)$ is the same condition as \eqref{eq:Wron}.

\begin{theorem}\label{t:smooth}
  Let $k$ be a non-negative integer and suppose that $u_0 \in L^1(\mathbb{R},(1+|x|)^{k+1} \D x)$ and assume $u_0$ is generic.  Then $R_{\mathrm{l}}(s)$ satisfies\footnote{A similar condition at $s = c$ is implied by $R_{\mathrm l}(-s) = \overline R_{\mathrm l}(s)$.}
\begin{align}\label{eq:smoothtoorder}\begin{split}
    R_{\mathrm l}(s) = \sum_{j=0}^k c_{j} (\sqrt{-s-c})^j + o(|s+c|^{k/2}), \quad s \to -c,~ s < -c,\\
    R_{\mathrm l}(s) = \sum_{j=0}^k \tilde c_{j} (\sqrt{-s-c})_+^j + o(|s+c|^{k/2}), \quad s \to -c,~ s > -c,
  \end{split}
  \end{align}
and $c_{j} = \tilde c_{j}$ for $j = 0,1,\ldots,k$.   Furthermore,  $R_{\mathrm{l/r}}$ are $C^k$ functions on $\mathbb R \setminus \{c ,-c\}$ satisfying
  \begin{equation}
    R_{\mathrm r}(\pm c ) = -1,\quad
    R_{\mathrm l}(0) = -1.
  \end{equation}
\end{theorem}
\begin{proof}
  Recall that from \eqref{eq:l-offcut} and \eqref{eq:l-oncut}
  \begin{equation}
    R_{\mathrm l}(s) = \begin{cases} \frac{b(s)}{a(s)} & |s| > c,\\
      \frac{a^+(-s)}{a^+(s)} & |s| \leq c. \end{cases}
  \end{equation}
  Consider the truncation $u_{0,L}(x) = u_0(x) \xi_{\{|x|\leq L\}}(x)$, $L > 0$, which has compact support so that
  \begin{equation}
\begin{aligned}
    R_{\mathrm l}(s;L) &= \sum_{j=0}^k c_{j,L} (\sqrt{-s-c})^j + o(|s+c|^{k/2}), \quad s \to -c,~ s < -c,\\
    R_{\mathrm l}(s;L) &= \sum_{j=0}^k \tilde c_{j,L} (\sqrt{-s-c})_+^j + o(|s+c|^{k/2}), \quad s \to -c,~ s > -c
 \end{aligned}
\end{equation}
  and $c_{j,L} = \tilde c_{j,L}$ for $j = 0,1,\ldots,k$.  Next, we show that these expressions remain valid as $L \to \infty$, implying \eqref{eq:smoothtoorder}.  Indeed this follows by Lemma~\ref{l:diff} as the limit can be applied term-by-term in the Taylor expansion.  A similar argument holds at $+c$.  The argument for $R_{\mathrm r}$ is simpler as once we know the Taylor expansions exist, \eqref{eq:r-oncut} gives the result.

  Now,
  \begin{equation}
    a^+(s) = \frac{W(\psi^{\text{m}}(s;\cdot), \phi^{\text{p}}(s;\cdot))}{2 \I s}
  \end{equation}
  so that for $s \neq 0$ we have
  \begin{equation}
    \frac{a^+(s)}{a^+(-s)} = -\frac{W(\psi^{\text{m}}(s;\cdot), \phi^{\text{p}}(s;\cdot))}{W(\psi^{\text{m}}(-s;\cdot), \phi^{\text{p}}(-s;\cdot))}
  \end{equation}
  Then, under the condition that $W(\psi^{\text{m}}(0;\cdot), \phi^{\text{p}}(0;\cdot)) \neq 0$ we find that $R_{\mathrm l}(0) = -1$.  Then to establish the required equalities at $\pm c$ we consider for $s^2 > c^2$, assuming the corresponding denominators do not vanish
  \begin{equation}
  R_{\mathrm{r}}(\pm c) = \frac{W(\phi^{\text{m}}(\pm c;\cdot), \hat \psi^{\text{m}}(0;\cdot))}{W(\phi^{\text{m}}(\pm c;\cdot), \hat \psi^{\text{p}}(0;\cdot))}.
  \end{equation}
But then $\hat \psi^{\text{p}}(0;\cdot) = \hat \psi^{\text{m}}(0;\cdot)$ so that $R_{\mathrm{r}}(\pm c) = 1$.
  \begin{equation}
    R_{\mathrm{r}}(\pm c)  = -\frac{W(\hat \psi^{\text{p}}(0;\cdot), \phi^{\text{p}}(\pm c;\cdot))}{W(\hat \psi^{\text{p}}(0;\cdot), \phi^{\text{m}}(\pm c;\cdot))}.
  \end{equation}

\end{proof}

\subsection{The final Riemann--Hilbert problems}

To finalize the setup of the RH problems, we must introduce time-dependence and residue conditions from the existence of solitons in the solution whenever $a(z)$ has a simple zero.  This process is detailed in Appendix~\ref{sec:KdV-Time}.  Specifically, it follows from the decay assumptions on $u_0$ that $a(z) = a(z;0)$ does not vanish on $\mathbb R$ and has a finite number of simple poles $\{z_1,\ldots,z_n\}$ in the open upper-half plane, all lying on the imaginary axis \cite{Cohen1985}.  Then define $\Sigma_1,\ldots \Sigma_n$ to be disjoint circular contours in the open upper-half plane of radius $\delta>0$ with $z_1,\ldots,z_n$ as their centers and clockwise orientation.  Additionally, give $- \Sigma_j := \{ -z : z \in \Sigma_j\}$ counter-clockwise orientation.

\begin{rhp}\label{rhp:1t}
  The function $\mb N_1: \mathbb C \setminus \mathbb R \to \mathbb C^{1\times 2}$, $\mb N_1(z) = \mb N_1(z;x,t)$ is analytic on its domain and satisfies
  \begin{equation}
\begin{aligned}
    \mb N_1^+(s) &= \mb N_1^-(s) \begin{bmatrix} 1 - |R_{\mathrm{l}}(s)|^2  & -R_{\mathrm{l}}(-s) \E^{2 \I s x + 8 \I s^3 t}  \\ {R_{\mathrm{l}}}(s) \E^{-2 \I s x- 8 \I s ^3 t} & 1 \end{bmatrix}, \quad s \in \mathbb R,\\
    \mb N_1^+(s) & = \mb N_1^-(s) \begin{bmatrix} 1 & 0 \\ -\frac{c(z_j)}{s - z_j}\E^{-2 \I z_j x - 8 \I z_j^3 t} & 1 \end{bmatrix},\quad s \in \Sigma_j,\\
    \mb N_1^+(s) & = \mb N_1^-(s) \begin{bmatrix} 1 & -\frac{c(z_j)}{s + z_j}\E^{-2 \I z_j x - 8 \I z_j^3 t} \\ 0 & 1 \end{bmatrix},\quad s \in -\Sigma_j,\\
    \mb N_1(z) &= \begin{bmatrix} 1 & 1 \end{bmatrix} + O(z^{-1}), \quad z \in \mathbb C \setminus \mathbb R,
  \end{aligned}
\end{equation}
  with the symmetry condition
  \begin{equation}
    \mb N_1(-z) = \mb N_1(z) \sgo, \quad z \in \mathbb C \setminus \Gamma, \quad \Gamma = \mathbb R \cup \bigcup_j (\Sigma_j \cup -\Sigma_j).
  \end{equation}
  
  \end{rhp}



\begin{theorem}\label{t:uniqueRH1}
  There exists a unique $L^2$ solution of \rhref{rhp:1t} provided $R_{\mathrm l}$ is any function on $\mathbb R$ that is continuous, decays at infinity and satisfies $\overline {R_{\mathrm l}(-s)} = R_{\mathrm l}(s)$.
\end{theorem}
\noindent For the proof of Theorem~\ref{t:uniqueRH1}, see Appendix~\ref{sec:rhp:1t}.

\begin{rhp} \label{rhp:2t} The function $\mb N_2: \mathbb C \setminus \mathbb R \to \mathbb C^{1\times 2}$, $\mb N_2(z) = \mb N_2(z;x,t)$ is analytic on its domain and satisfies
 \begin{equation}
\begin{aligned}
  \mb N_2^+(s) &= \mb N_2^-(s) \begin{bmatrix} 1 - |R_{\mathrm{r}}(s)|^2 & - R_{\mathrm{r}}(-s) \E^{ -2 \I  \lambda(s) x - 8 \I \varphi(s) t} \\ R_{\mathrm{r}}(s) \E^{2 \I \lambda(s) x + 8 \I  \varphi(s) t} & 1 \end{bmatrix}, \quad s^2 > c^2,\\
    \mb N_2^+(s) &=  
           \mb N_2^-(s) \begin{bmatrix} 1 & -R_{\mathrm r}(-s) \E^{-2 \I \lambda^-(s) x - 8 \I \varphi^-(s) t}  \\ 0 & 1 \end{bmatrix}  \sgo  \begin{bmatrix} 1 &  0 \\ R_{\mathrm r}(s) \E^{2 \I \lambda^+(s) x + 8 \I \varphi^+(s) t }& 1 \end{bmatrix}, \quad -c \leq s \leq c,\\
\mb N_2^+(s) & = \mb N_2^-(s) \begin{bmatrix} 1 & 0 \\ -\frac{C(z_j)}{s - z_j}\E^{2 \I \lambda(z_j) x + 8 \I \varphi(z_j) t} & 1 \end{bmatrix},\quad s \in \Sigma_j,\\
    \mb N_2^+(s) & = \mb N_2^-(s) \begin{bmatrix} 1 & -\frac{C(z_j)}{s + z_j}\E^{2 \I \lambda(z_j) x + 8 \I \varphi(z_j) t} \\ 0 & 1 \end{bmatrix},\quad s \in -\Sigma_j,\\
     \varphi(s) &=\lambda^3(s) + \tfrac{3}{2} c^2 \lambda(s),
  \end{aligned}
\end{equation}
  with the symmetry condition
  \begin{equation}
    \mb N_2(-z) = \mb N_2(z) \sgo, \quad z \in \mathbb C \setminus \mathbb R.
  \end{equation}
\end{rhp}

\begin{theorem}\label{t:uniqueRH2}

  Assume
  \begin{enumerate}
  \item $a,b,A,B : \mathbb R \setminus [-c,c] \to \mathbb C$ are 1/2-H\"older continuous functions such that $a(s)$ and $b(s)$, can be extended to 1/2-H\"older continuous functions on $\mathbb R \setminus (-c,c)$.
  \item The symmetries \eqref{eq:transl} and \eqref{eq:Btob} hold for $s^2 > c^2$.
  \item For $s^2 > c^2$, $\overline a(s) = a(-s)$ and $\overline b(s) = b(-s)$
  \item $a^+,A^+: (-c,c)\to \mathbb C$ are 1/2-H\"older functions such that $sa^+(s),\lambda_+(s)A^+(s)$ can be extended to 1/2-H\"older continuous functions on $[-c,c]$ and $a^+(\pm c) = a(\pm c)$.
  \item $a,b$ satisfy
    \begin{equation}
\begin{aligned}
      a(s) &= \alpha_{1,-}  + \alpha_{2,-} \sqrt{-s-c} + O(|s+c|), \quad s \to -c, \quad s^2 > c^2,\\
      b(s) &=  - \alpha_{1,-}  + \beta_{2,-} \sqrt{-s-c} + O(|s+c|), \quad s \to -c, \quad s^2 > c^2,\\
      a(s) &= \alpha_{1,+}  + \alpha_{2,+} \sqrt{s-c} + O(|s-c|), \quad s \to c, \quad s^2 > c^2,\\
      b(s) &=  -\alpha_{1,+}  + \beta_{2,-} \sqrt{s-c} + O(|s-c|), \quad s \to c, \quad s^2 > c^2,
  \end{aligned}
\end{equation}
    for some  $\alpha_{j,\pm}$, $\beta_{j,\pm} \in \mathbb C$.
  \item $a^+$ satisfies
    \begin{equation}
\begin{aligned}
      a^+(s) &= \zeta_{1,-} + \zeta_{2,-} \sqrt{s+c}  + O(|s+c|), \quad s \to -c, \quad s > - c,\\
      a^+(s) &=  -\zeta_{1,-} + \zeta_{2,+} \sqrt{c-s}  + O(|s-c|), \quad s \to c, \quad s < c,
   \end{aligned}
\end{equation}
    for some $\zeta_1$ and $\zeta_{2,\pm} \in \mathbb C$.
    
  \item $A^+(s) = a^+(s)\frac{s}{\lambda_+(s)}$ for $s \in (-c,c)$
  \item Neither $a(s)$ nor $s a^+(s)$ vanish within their domains of definition.

  \item $R_{\mathrm l}(s)$ is given by \eqref{eq:l-oncut}.
  \item $R_{\mathrm r}(s)$ is given by \eqref{eq:r-oncut} and \eqref{eq:match} and \eqref{eq:match2} hold.
  \item $R_{\mathrm r/\mathrm l}(s) = O(s^{-1})$ as $|s| \to \infty$.
  \end{enumerate}
Then there exists a unique $L^2$ solution of \rhref{rhp:2t}.
\end{theorem}
\noindent For the proof of Theorem~\ref{t:uniqueRH2}, see Section~\ref{sec:rhp:2t} of the Appendix. We can now prove our theorem about the existence of solutions of the KdV equation via RH problems.

\begin{theorem}\label{t:main}  
Suppose $u_0$ is generic.  Then the following hold:
\begin{enumerate}
\item If $u_0 \in L^1(\mathbb R, (1+|x|) \D x)$ then \rhref{rhp:1t} has a unique solution.
\item If $u_0 \in L^1(\mathbb R, (1+|x|)^3 \D x)$ then \rhref{rhp:2t} has a unique solution.
\item If either $u(\cdot,0) \in \mathcal D_3$ or $u_0 \in L^1(\mathbb R, \E^{\delta |x|} \D x)$ for some $\delta > 0$ then by the Dressing Method these solutions produce the solution of the KdV equation for $t > 0$:
  \begin{align}\label{eq:recover}
   \begin{split}
  \lim_{z \to \infty} 2 \I z (\mb N_1(z) - \begin{bmatrix} 1 & 1 \end{bmatrix}) &= \begin{bmatrix} - \int_{-\infty}^x u(x',t) \D x' & \int_{-\infty}^x u(x',t) \D x' \end{bmatrix},\\
\lim_{z \to \infty} 2 \I z (\mb N_2(z) - \begin{bmatrix} 1 & 1 \end{bmatrix}) &= \begin{bmatrix} - \int_{x}^\infty [u(x',t) + c^2] \D x' & \int_{x}^\infty [u(x',t) + c^2] \D x' \end{bmatrix}.
\end{split}\end{align}
\end{enumerate}
\end{theorem}
\begin{proof}
Parts (1) and (2) follow from Lemma~\ref{l:diff} and Theorems~\ref{t:uniqueRH1} and \ref{t:uniqueRH2}.  Part (3) is the application of the Dressing Method and the conditions imposed are sufficient for the solution of the RH problem to be differentiable both in $x$ and $t$ the required number of times.  For $u(\cdot,0) \in \mathcal D_3$ see Lemma~\ref{l:decay} and for $u_0 \in L^1(\mathbb R, \E^{\delta |x|} \D x)$, see the deformations in Section~\ref{sec:tpos} which induces exponential decay of the jump matrix.
\end{proof}

\begin{remark}
  It is important to note that if one solves \rhref{rhp:1t} for large values of $x$, the recovery formula \eqref{eq:recover} produces a quantity that grows as $x$ increases.  This indicates that the operator one is inverting is not well-conditioned in this limit.  Thus there is a reason based on numerical stability for including both \rhref{rhp:1t} and \rhref{rhp:2t}.
  \end{remark}


\section{Contour deformations and numerical inverse scattering}  
\label{sec:KdV-Comp}

Throughout this section we assume $u_0 \in L^1(\E^{2 \nu |x|}\D x)$ for some $\nu > 0$.  This immediately implies that, in addition to other analyticity properties, $\phi^{\text{p/m}}$ and $\hat \psi^{\text{p/m}}$ and their $x$-derivatives have analytic extensions as functions of $z$ within the open strip $S_\nu := \{z \in \mathbb C: |\Im z| < \nu\}$ and continuous in the closure.  Define
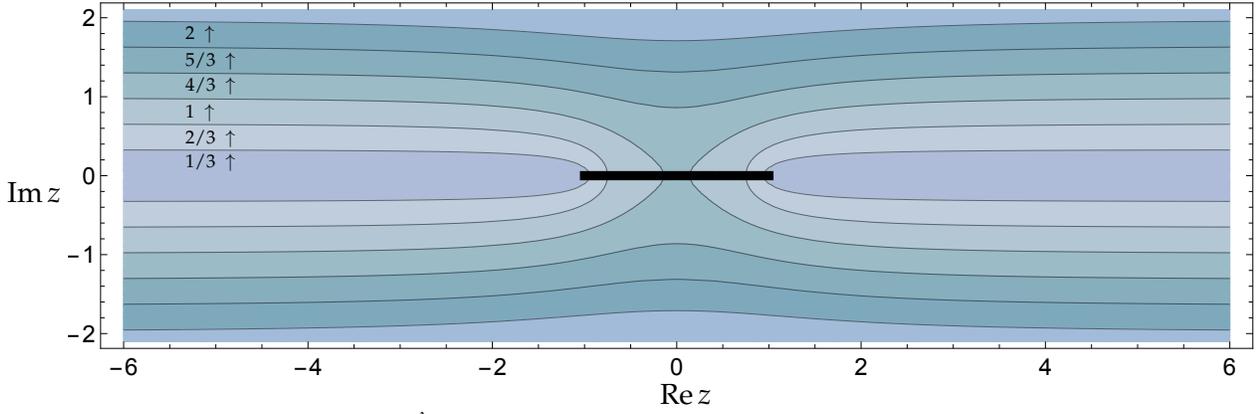
\begin{figure}[tbp]
  \centering
  \begin{overpic}[width=.9\linewidth]{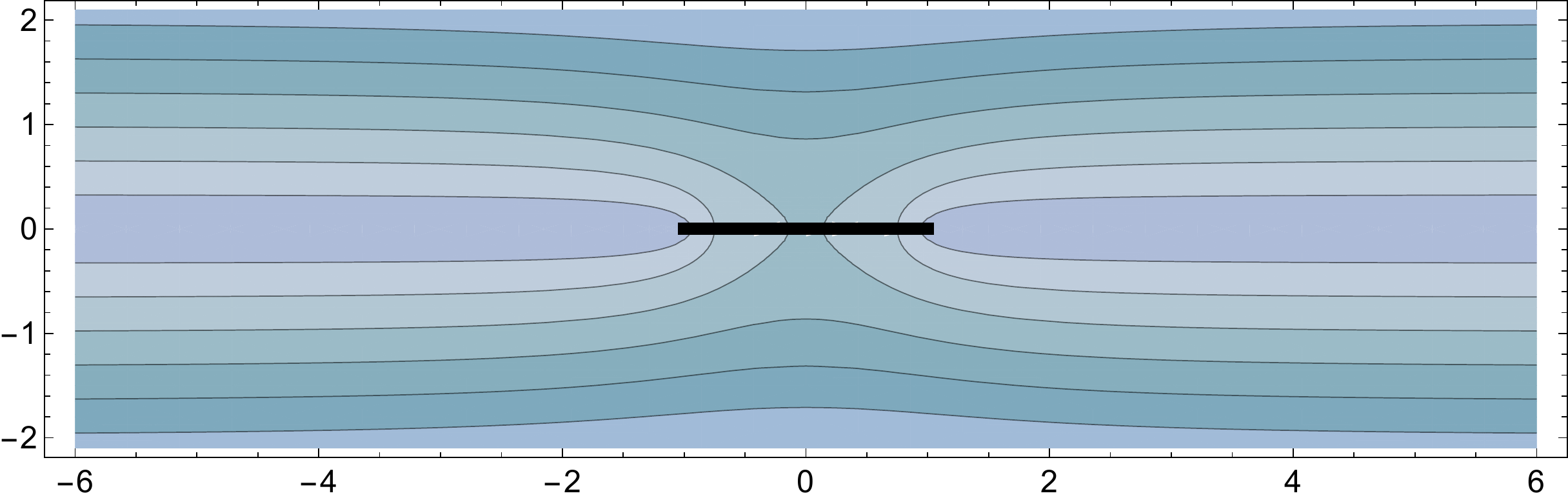}
    \put(50,-2){$\Re z$}
    \put(-5,15){$\Im z$}
    \put(10,18){\tiny $1/3~\uparrow$}
    \put(10,20){\tiny $2/3~\uparrow$}
    \put(10,22.2){\tiny $1~\uparrow$}
    \put(10,24.4){\tiny $4/3~\uparrow$}  
    \put(10,26.5){\tiny $5/3~\uparrow$}
    \put(10,28.8){\tiny $2~\uparrow$}
  \end{overpic}
  \caption{The domain $S_{\nu}^\lambda$ for varying values of $\nu$ when $c = 1$.  Specifically, this plot gives the level curves of $|\Im \lambda(z)|$.}\label{f:Snu}
\end{figure}
\begin{equation}
  S_\nu^\lambda = \{z \in \mathbb C: |\Im \lambda(z)| < \nu\}.  
\end{equation}
See Figure~\ref{f:Snu} for a plot.  It is clear that $\mathbb R \setminus [-c,c] \subset S_{\nu}^\lambda$ for any choice of $\lambda$.  Then, for example, it follows that $\psi^{\text{p}}(z;x)$ is an analytic function of $z$ within the region
\begin{equation}
  S_\nu^{\lambda,+} := \mathbb C^+ \cup S_{\nu}^\lambda \setminus [-c,c],
\end{equation}
while $\psi^{\text{m}}(z;x)$ is an analytic function of $z$ within the region
\begin{equation}
  S_\nu^{\lambda,-} :=\mathbb C^- \cup S_{\nu}^\lambda \setminus [-c,c].
\end{equation}
It then follows that $R_{\mathrm l}(s)$ has a meromorphic extension to $S_\nu^{\lambda,+}$ while $R_{\mathrm r}(s)$ has a meromorphic extension to only $S_\nu^{\lambda,+} \cap S_\nu^{\lambda,-}$.  These regions of analyticity are sufficient to make all the deformations outlined below.

\subsection{Computing $R_{\text{r/l}}$}  We note that the computation of the reflection coefficients is no different than that in the case of decaying data \cite{TrogdonSOKdV}.  Indeed, we compute the scattering data by evaluating at $x = 0$, see Remark~\ref{r:schr}.

\subsection{Computing $\{z_j\}$, $C(z_j)$ and $c(z_j)$}  The authors in \cite{TrogdonSOKdV} used Hill's method \cite{hill} to compute the (negative) eigenvalues of the operator \eqref{eq:spec1} at $t = 0$ and therefore find the zeros $a(z)$ in the upper-half plane.  This required initial data with decay, so that one can approximate the eigenvalues with those from a operator on a space of periodic functions.  Here, we choose $L >0$ so that $|u_0(x)| < \epsilon$ for $|x| > L$ and $\epsilon$ is on the order of machine precision.  Then \eqref{eq:spec1} can be approximated by
\begin{align}\label{eq:spec1N}
  -D_{N,L}^2 - \diag u(\vec x_{N,L},0)
\end{align}
where $D_{N,L}$ is the first-order Chebyshev differentiation matrix \cite{TrefethenSpectral} for $\vec x_{N,L}$, the vector of $N$th-order Chebyshev points scaled to the interval $[-L,L]$.  For sufficiently large $L,N$, the eigenvalues of \eqref{eq:spec1N} near the negative real axis approximate the eigenvalues of \eqref{eq:spec1}.

\subsection{The numerical solution of Riemann--Hilbert problems}

The numerical solution of an $L^2$ RH problem is based around the representation of $H^2_\pm(\Gamma)$ functions as the Cauchy integral of $L^2(\Gamma)$ functions and consequently, the equivalency between solving the RH problem for $\mb N$ and solving the singular integral equation
\begin{align} \label{eq:numer-op}
  \mb u - \mathcal C_{\Gamma}^- \mb u \cdot (\mb G - \mb I) = \mb G - \mb I, \quad \mb N = \mathcal C_{\Gamma} \mb u + \mb I.
\end{align}
This integral equation is discretized (see \cite{SORHFramework,TrogdonSOBook}) using mapped Chebyshev polynomials.  The convergence rate is closely tied to the smoothness of solutions \cite{TrogdonSONNSD} and invertibility  of the associated operator on high-order Sobolev spaces is required \cite{TrogdonSOBook}.  Fortunately, this is immediate following Theorems~\ref{t:uniqueRH1} and \ref{t:uniqueRH2}, and the fact that the jump matrix $\mb G$ we encounter, after deformation, will satisfy the $k$th-order product condition \cite[Definition 2.55]{TrogdonSOBook} for every $k$.  Full details on the numerical solution of RH problems is relegated to the references, particularly \cite{TrogdonSOBook}.

The deformation of a RH problem is an explicit transformation $(\mb G,\Gamma) \mapsto (\tilde {\mb G}, \tilde \Gamma)$ such the solutions of the two problems are in correspondence.  The goal is for the operator $\mb u \mapsto \mb u - \mathcal C_{\tilde \Gamma}^- \mb u \cdot (\tilde {\mb G} - \mb I)$ to be better conditioned than the original operator \eqref{eq:numer-op}, i.e. have a smaller condition number.  To have any analytic expressions for the solution, one needs the condition number to tend to one in an asymptotic limit, while numerically, one just aims to have a bounded quantity.

\subsection{Recovery of $u(x,t)$}

Once the solution of \eqref{eq:numer-op} has been computed, one then seeks $\partial_x \mb u = \mb u_x$, see \eqref{eq:recover}. To do this, we solve the equation solved by $\mb u_x$:
\begin{equation}
\mb u_x - \mathcal C_{\Gamma}^- \mb u_x \cdot (\mb G - \mb I) = (\mathcal C_{\Gamma}^- \mb u + \mb I) \mb G_x, \quad \mb N_x = \mathcal C_{\Gamma} \mb u_x.
\end{equation}
And then, formally,
\begin{equation}
  \lim_{z \to \infty} z \mb N_x(z) = -\frac{1}{2 \I \pi }\int_{\Gamma} \mb u_x(s) \D s.
\end{equation}
Assuming the operator in  \eqref{eq:numer-op} is invertible, these formal manipulations are justified provided $\mb G_x \in L^1 \cap L^\infty(\Gamma)$ and $\mathcal C_{\Gamma}^- \mb u + \mb I \in L^\infty(\mathbb R)$.

\section{Numerical Inverse scattering at $t = 0$}
\label{sec:KdV-IST-t0}
We divide this computation into two cases, $x<0$ and $x \geq 0$.  We first ignore the jumps on the contours $\Sigma_j$, $-\Sigma_j$.  

\subsection{$x < 0$}

Under our assumptions, $R_{\mathrm l}$ has a meromorphic extension to $\nu \geq \Im z > 0$, decaying at infinity within this strip.  And because $R_{\mathrm l}$ has a finite number of poles in this strip, we can use the factorization
\begin{equation}
  \begin{bmatrix}
    1 - R_{\mathrm l}(s)R_{\mathrm l}(-s) & R_{\mathrm l}(-s)\E^{- 2 \I x s} \\
    -R_{\mathrm l}(s) \E^{2 \I x s} & 1 \end{bmatrix} = \mb M_{1}(s) \mb P_1^{-1}(s) = 
                        \begin{bmatrix} 1 & R_{\mathrm l}(-s)\E^{- 2 \I x s}  \\ 0 & 1 \end{bmatrix}
                                                                     \begin{bmatrix} 1 & 0 \\ -R_{\mathrm l}(s)\E^{2 \I x s}  & 1 \end{bmatrix},                                                                                              
\end{equation}
noting that $\overline{R_{\mathrm l}(s)} = R_{\mathrm l}(-s)$, to deform \rhref{rhp:1t} within a possibly smaller strip $\alpha \leq\delta$. One does this by the so-called \emph{lensing} process:  Given $\mb N_1$ define
\begin{equation}
  \tilde {\mb N}_1(z) = \begin{cases} \mb N_1(z)\mb P_1(z) & 0 < \Im z < \alpha,\\
    \mb N_1(z)\mb M_1(z) & -\alpha < \Im z < 0, \end{cases}
\end{equation}
and then $\tilde {\mb N}_1(z)$ satisfies the RH problem depicted in Figure~\ref{f:rhp1-t=0}.  The jumps matrices decay exponentially to the identity matrix as $x \to - \infty$.  
                  
\begin{figure}[ht]
  \vspace{.3in}
  \centering
  \begin{overpic}[width=\linewidth]{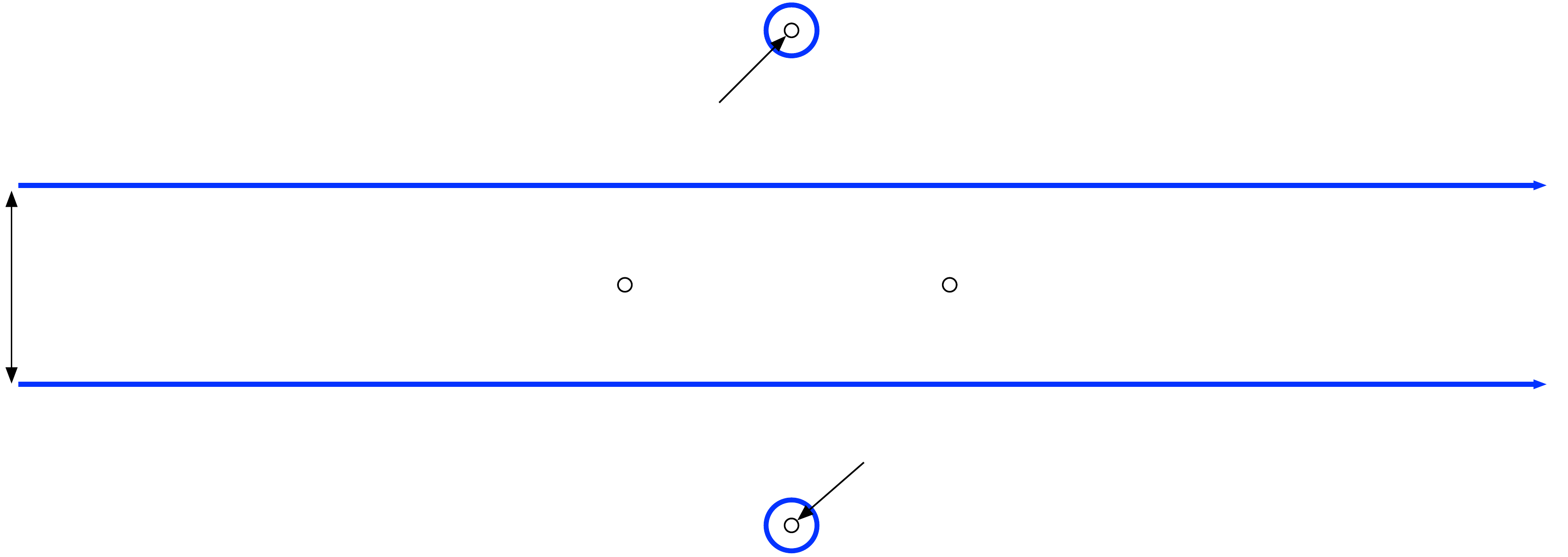}
    \put(40,18){$-c$}
    \put(59,18){$c$}
    \put(34,25){\Large $\mb P_1^{-1}$}
    \put(34,12){\Large $\mb M_1$}
    \put(1.5,16){\Large $2 \alpha$}
    \put(43,27){\large $z_j$}
    \put(56,7){\large $-z_j$}
  \end{overpic}
  \caption{The initial deformation of \rhref{rhp:1t} for $t = 0$, $x < 0$.  The jumps on the contours $\Sigma_j$ and $- \Sigma_j$ are unchanged at this stage.}\label{f:rhp1-t=0}
\end{figure}

\subsection{$x \geq 0$}

The situation for $x \geq 0$ is more complicated because the jump condition in \rhref{rhp:2t} is discontinuous.  Furthermore, we can only lens the jump matrix within as subregion of $S_{\nu}^{\lambda}$.  See Figure~\ref{f:rhp2-t=0:1} for a depiction of the jump contours and jump matrices after lensing.  But this RH problem, even though it is uniquely solvable in an $L^2$ sense, has a jump matrix that is not smooth, in the sense of the product condition \cite[Definition 2.55]{TrogdonSOBook} at $\pm c$. A local deformation is required, using \eqref{eq:W} below 
with jump matrices and jump contours depicted in Figure~\ref{f:rhp2-t=0:1}. Then define two neighborhoods $B_{\pm c}$ of $\pm c$, by first defining $B_c$ shown in Figure~\ref{f:rhp2-t=0:2} and setting $B_{-c} = \{-z : z \in B_c\}$. Now, define a new unknown
\begin{equation}
  \hat {\mb N}_2(z) = \tilde {\mb N}_2(z) \begin{cases} \mb W^{\mp 1}(z) & z \in B_{\pm c},\\
    \mb I & \text{otherwise}. \end{cases}
\end{equation}
where $\mb W$ is defined in \eqref{eq:W}. We point out that this definition is made to both solve the jump on the small intervals near $\pm c$ and to preserve the symmetry condition:  If a function satisfies $\mb N(-z) = \mb N(z) \sigma_1$ and we want a new function $\hat {\mb  N}(z) = \mb N(z) \mb C(z)$ to satisfy the same condition, then:
\begin{equation}
  \hat {\mb  N}(-z) = \mb N(-z) \mb C(-z) = \mb N(z) \sigma_1 \mb C(-z),
\end{equation}
and one concludes that $\sigma_1 \mb C(-z) = \mb C(z) \sigma_1$ is a sufficient condition.  In the case of $\mb W$, we see that $\sigma_1 \mb W^{-1}(-z) \sigma _1 = \mb W(z)$.

\begin{figure}[ht]
  \vspace{.3in}
  \centering
  \begin{overpic}[width=\linewidth]{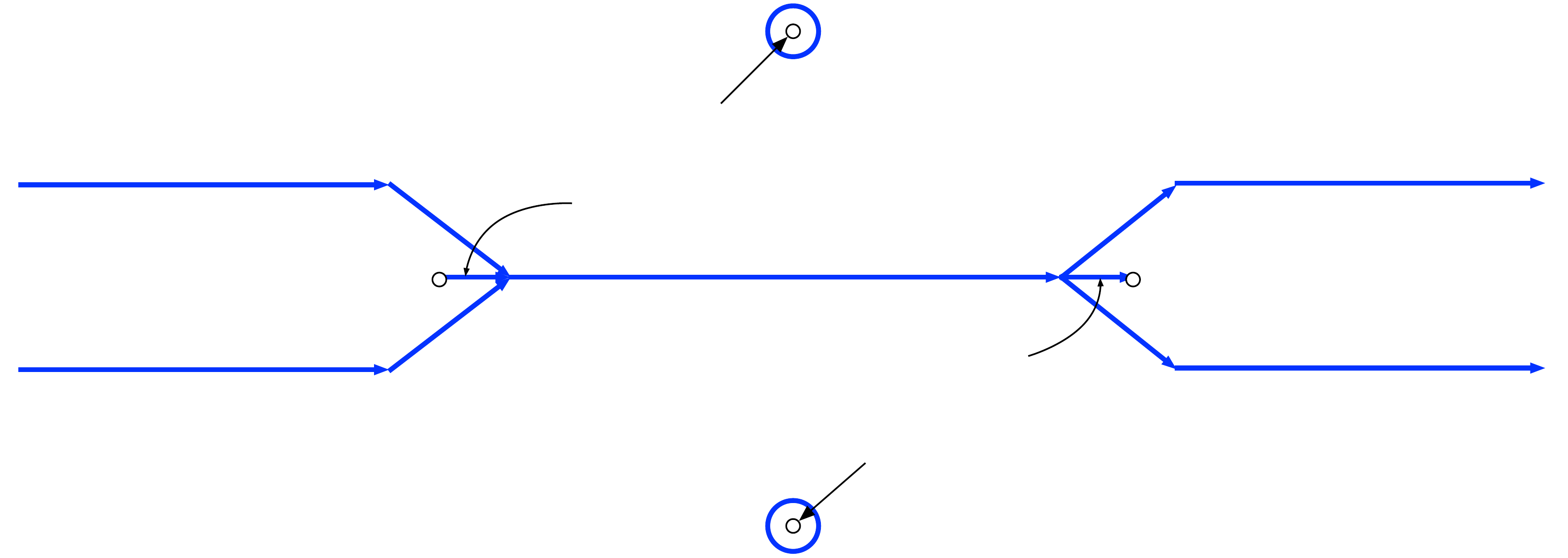}
    \put(24,18){$-c$}
    \put(74,18){$c$}
    \put(32,25.3){ $\begin{bmatrix} 0 & 1 \\ 1 & 0\end{bmatrix}$}
    \put(63.5,8.3){ $\begin{bmatrix} 0 & 1 \\ 1 & 0\end{bmatrix}$}
    \put(43,27){\large $z_j$}
    \put(56,7){\large $-z_j$}
    \put(14,25){\Large $\mb P_2^{-1}$}
    \put(14,13){\Large $\mb M_2$}
    \put(84,25.1){\Large $\mb P_2^{-1}$}
    \put(84,13){\Large $\mb M_2$}
    \put(43,14){$\mb J_2 = \mb M_1 \begin{bmatrix} 0 & 1 \\ 1 & 0\end{bmatrix} \mb P_2^{-1}$}
  \end{overpic}
  \caption{The initial deformation of \rhref{rhp:2t} for $t = 0$, $x \geq 0$.  The jumps on the contours $\Sigma_j$ and $- \Sigma_j$ are unchanged at this stage.}\label{f:rhp2-t=0:1}
\end{figure}

\begin{figure}[ht]
  \vspace{.3in}
  \centering
  \begin{overpic}[width=\linewidth]{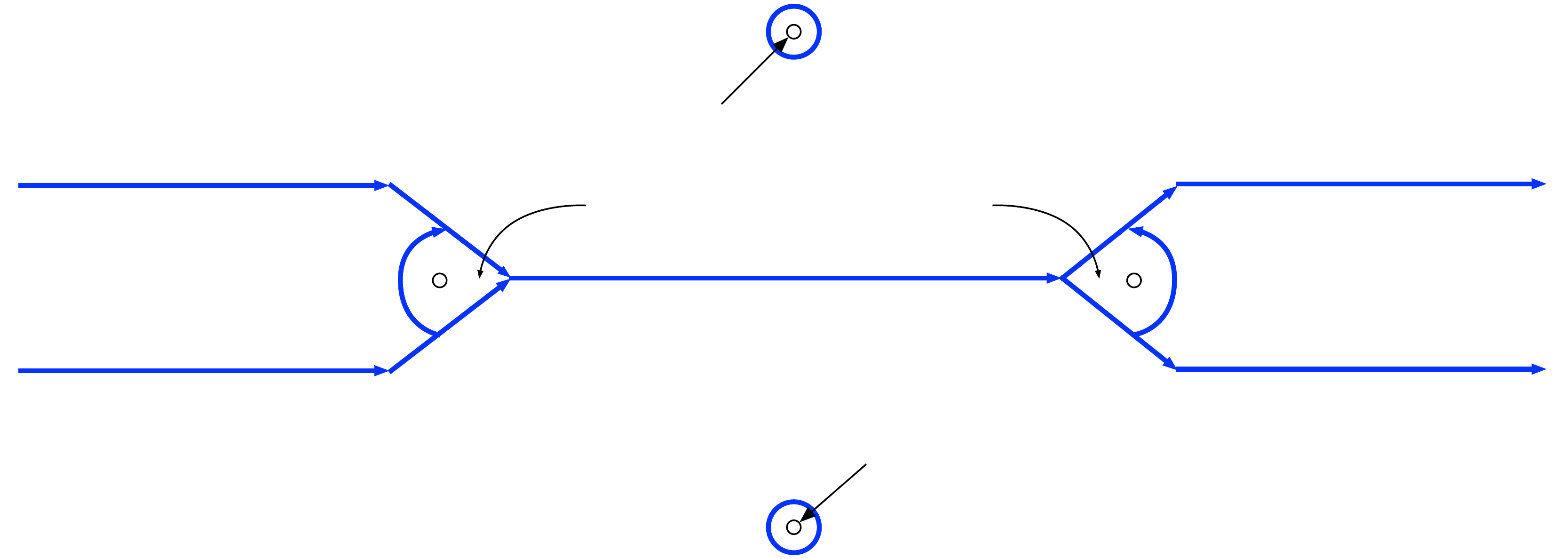}
    \put(37,22){ $B_{-c}$}
    \put(60,22){ $B_{c}$}
    \put(43,27){\large $z_j$}
    \put(56,7){\large $-z_j$}
    \put(14,25){\Large $\mb P_2^{-1}$}
    \put(14,13){\Large $\mb M_2$}
    \put(84,25.1){\Large $\mb P_2^{-1}$}
    \put(84,13){\Large $\mb M_2$}
    \put(43,14){$\mb J_2 = \mb M_2 \begin{bmatrix} 0 & 1 \\ 1 & 0\end{bmatrix} \mb P_2^{-1}$}
  \end{overpic}
  \caption{The second deformation of \rhref{rhp:2t} for $t = 0$, $x \geq 0$.  The jumps on the contours $\Sigma_j$ and $- \Sigma_j$ are unchanged at this stage.}\label{f:rhp2-t=0:2}
\end{figure}

\begin{figure}[ht]
  \vspace{.3in}
  \centering
  \begin{overpic}[width=.6\linewidth]{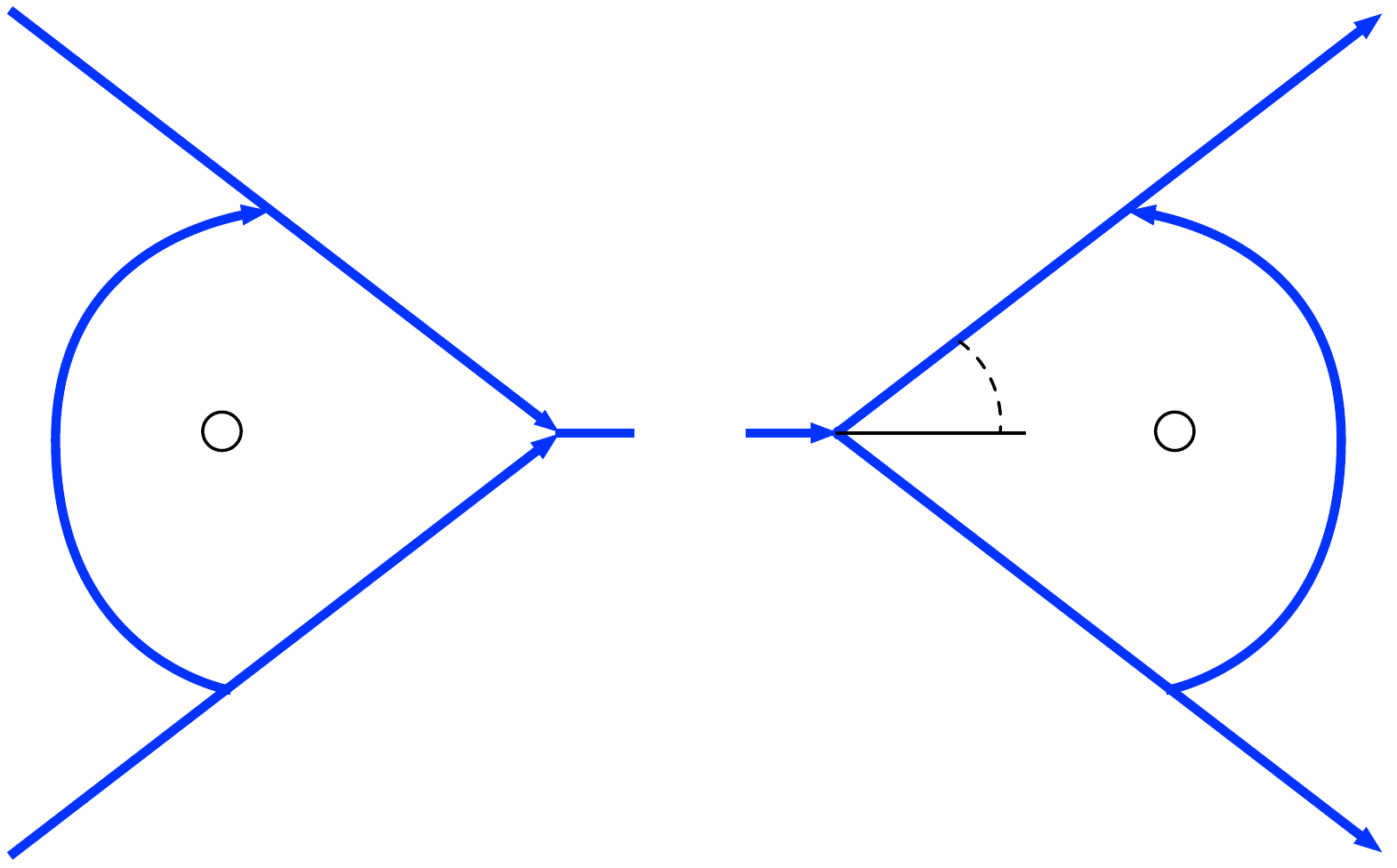}
    \put(18,30){ $-c$}
    \put(77,30){ $c$}
    \put(6,59){\Large $\mb P_2^{-1}$}
    \put(6,1){\Large $\mb M_2$}
    \put(87,59){\Large $\mb P_2^{-1}$}
    \put(87,1){\Large $\mb M_2$}
    \put(5,30){\Large $\mb W$}
    \put(90,30){\Large $\mb W$}
    \put(60,18){\Large $\mb M_2 \mb W$}
    \put(58,44){\Large $\mb W^{-1} \mb P_2^{-1}$}
    \put(26,44){\Large $\mb W \mb P_2^{-1}$}
    \put(28,18){\Large $\mb M_2 \mb W^{-1}$}
    \put(71,34){ $\pi/3$}
  \end{overpic}
  \caption{A zoomed view of the second deformation of \rhref{rhp:2t} for $t = 0$, $x \geq 0$.  All contours intersecting the real axis make the same angle with the real axis.  The angle $\pi/3$ is chosen so that $\E^{\pm \I \lambda^3(z)}$ decays exponentially, for large $z$, in the appropriate quadrants.}\label{f:rhp2-t=0:zoom}
\end{figure}
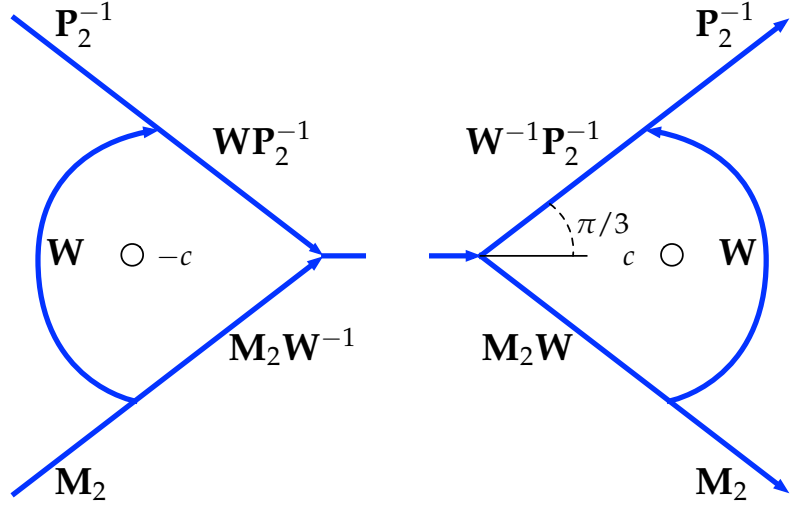

\subsection{Jump matrices on $\Sigma_j$}\label{sec:sigma}

Consider a RH problem with jump conditions the form
\begin{equation}  
  \mb N^+(s) = \mb N^-(s) \begin{cases} \begin{bmatrix} 1 & 0 \\ \frac{\alpha}{s-z_j} & 1\end{bmatrix} & s \in \Sigma_j,\\ \\
    \begin{bmatrix} 1 &  \frac{\beta}{s+z_j} \\ 0 & 1 \end{bmatrix} & s \in - \Sigma_j.\end{cases}
\end{equation}
Define
\begin{equation}
\begin{aligned}
  \mb Q(z) &= \begin{bmatrix} \frac{z - z_j}{z + z_j} & 0 \\ 0 & \frac{z + z_j}{z - z_j} \end{bmatrix}, \quad \mb M(z) = \mb N(z) \mb T(z;z_j,\alpha,\beta), \\ \mb T(z;z_j,\alpha,\beta) &= \begin{cases} \mb Q(z) & z \text{ outside } \Sigma_j \text{ and } -\Sigma_j,\\
    \begin{bmatrix} \frac{z-z_j}{z+z_j} & \frac{1}{\alpha(z+z_j)} \\ {- \alpha (z + z_j)} & 0 \end{bmatrix} & z \text{ inside } \Sigma_j,\\ \\
    \begin{bmatrix} 0 &\beta (z - z_j)  \\ -\frac{1}{\beta(z-z_j)}  & \frac{z+z_j}{z-z_j} \end{bmatrix} & z \text{ inside } -\Sigma_j.
  \end{cases}
\end{aligned}
\end{equation}
Then the jump conditions satisfied by $\mb M(z)$ are given by
\begin{equation}
\begin{aligned}
  \mb M^+(s)  &= \mb M^-(s)  \begin{cases}\mb Q^{-1}(s) \begin{bmatrix} 1 & 0 \\ \frac{\alpha}{s-z_j} & 1 \end{bmatrix} \begin{bmatrix} \frac{s-z_j}{s+z_j} & \frac{1}{\alpha(s+z_j)} \\ - \alpha (s + z_j) & 0 \end{bmatrix} & s \in \Sigma_j,\\ \\
    \begin{bmatrix} \frac{s + z_j}{s - z_j} &-\beta (s - z_j)  \\ \frac{1}{\beta(s-z_j)}  & 0  \end{bmatrix}\begin{bmatrix} 1 &  \frac{\beta}{s+z_j} \\ 0 & 1 \end{bmatrix}\mb Q(s) & s \in - \Sigma_j,\end{cases}
  = \mb M^-(s)  \begin{cases} \begin{bmatrix}1 & \frac{1}{\alpha(s-z_j)} \\ 0 &  1 \end{bmatrix} & s \in \Sigma_j,\\ \\
    \begin{bmatrix} 1 & 0 \\ \frac{1}{\beta(s+z_j)} &  1  \end{bmatrix} & s \in - \Sigma_j.\end{cases}
\end{aligned}
\end{equation}
When $\alpha$ and $\beta$ are both large, this transformation allows us to convert the jump to one that is near-identity.  We will only need to apply this transformation in the case $\alpha = \beta$, in which case we use the notation $\mb T(z;z_j,\alpha) = \mb T(z;z_j,\alpha,\beta)$.

To see how to employ this in the context of the KdV equation define two index sets, depending on $x$ and $t$
\begin{equation}
  S_1(x,t) = \{j : |c(z_j) \E^{-2 \I z_j - 8 \I z_j^3}| > 1\}, \quad S_2(x,t) = \{j : |C(z_j) \E^{-2 \I \lambda(z_j) - 8 \I \varphi(z_j)}| > 1\},
\end{equation}
and two matrix functions defined on $\mathbb C \setminus \left( \bigcup_j (\Sigma_j \cup -\Sigma_j)\right)$
\begin{equation}
  \mb Q_1(z) = \prod_{j \in S_1(x,t)} \mb T(z;z_j,-c(z_j)\E^{-2 \I z_j - 8 \I z_j^3}), \quad \mb Q_2(z) = \prod_{j \in S_2(x,t)} \mb T(z;z_j,-C(z_j)\E^{-2 \I \lambda(z_j) - 8 \I \varphi(z_j)}).
\end{equation}
Our final step before solving the RH problem for $\mb N_j$ will be to instead consider the RH problem for $\mb N_j \mb Q_j$.  This includes our calculations for $t > 0$ below.  We do not present the final RH problem, after this modification, as the preceding calculations allow one to directly derive the new jumps.

\section{Numerical inverse scattering for two asymptotic regions}\label{sec:tpos}

We now discuss simple deformations that lead to asymptotically accurate computations in two regions.  The full deformation of the RH problem to compute asymptotic solutions in the entire $(x,t)$-plane will be presented in a forthcoming work.

\subsection{$x \geq -2 c^2 t$}

We begin with a simple but important calculation.  For $s \in (-c,c)$ and $\zeta \in \mathbb R$ consider  
\begin{equation}
  h(s) = {2 \I \lambda^+(s) \zeta + 8 \I \varphi^+(s)} = -\sqrt{c^2-s^2} \left[ 2 \zeta + 12 c^2 - 8 (c^2 - s^2) \right].
\end{equation}
This function, evidently, has a local minimum at $s = 0$ where $h(0) = -|c|(2 \zeta + 4 c^2)$.  This remains non-positive provided that $\zeta \geq  - 2 c^2$.  Thus the jump in \rhref{rhp:2t} on $(-c,c)$ has its $(1,1)$ entry less than unity, in absolute value, provided that $x \geq -2c^2t$.  For this regime, we can use the deformation depicted in Figures~\ref{f:rhp2-t=0:2} and \ref{f:rhp2-t=0:zoom}, using \rhref{rhp:2t}.

Before the deformed RH problem is solved numerically, the deformation detailed in Section~\ref{sec:sigma} is performed.  

\subsection{$\sqrt{\frac{-x}{12t}} \geq c + \delta$}

In this region we use \rhref{rhp:1t} exclusively.  Recalling that $\overline R_{\mathrm l}(s) = R_{\mathrm l}(-s)$ we consider, formally,
\begin{equation}
\begin{aligned}
  &\begin{bmatrix}
    1 - R_{\mathrm l}(s)R_{\mathrm l}(-s) & R_{\mathrm l}(-s)\E^{- 2 \I s x- 8 \I  s^3 t} \\
    -R_{\mathrm l}(s) \E^{2 \I s x + 8 \I s^3 t} & 1 \end{bmatrix} = \mb M_{1}(s) \mb P_1^{-1}(s) = 
                        \begin{bmatrix} 1 & R_{\mathrm l}(-s)\E^{- 2 \I x s- 8 \I  s^3 t}  \\ 0 & 1 \end{bmatrix}
                                                                                                  \begin{bmatrix} 1 & 0 \\ -R_{\mathrm l}(s)\E^{2 \I x s+ 8 \I  s^3 t}  & 1 \end{bmatrix}\\
  & = \mb L(s) \mb D(s) \mb U^{-1}(s) = \begin{bmatrix} 1 & 0 \\  -\frac{R_{\mathrm l}(s)}{T(s)}\E^{2 \I x s+ 8 \I  s^3 t} & 1 \end{bmatrix}\begin{bmatrix} T(s) & 0 \\ 0 & 1/T(s) \end{bmatrix} \begin{bmatrix} 1 & \frac{R_{\mathrm l}(-s)}{T(s)}\E^{- 2 \I x s- 8 \I  s^3 t} \\  0 & 1 \end{bmatrix},
  \end{aligned}
\end{equation}
with
\begin{equation}
  T(s) := 1 - |R_{\mathrm l}(s)|^2 = 1 - R_{\mathrm l}(s)R_{\mathrm l}(-s).
\end{equation}
The first factorization is valid for $s \in \mathbb R$.  The second factorization fails when $|R_{\mathrm l}(s)| = 1$ which occurs for $s \in [-c,c]$.

As is customary, we use the stationary phase points $z^* = \pm \sqrt{-x/(12 t)}$ to guide the deformation.  Given $\alpha >0$ define six polygonal regions in $\mathbb C$:
\begin{equation}
\begin{aligned}
  \Omega_1 &= \{ z : 0 < \Im z < \alpha, ~~ \Im z < \Re z - z^*\},\\
  \Omega_2 &= \{ z : 0 < \Im z < \alpha, ~~ \Im z < -\Re z + z^*, ~~ \Im z < \Re z + z^*\},\\
  \Omega_3 &= \{ z : 0 < \Im z < \alpha, ~~ \Im z < -\Re z - z^*\},\\
  \Omega_4 &= \{ z : -\alpha < \Im z < 0, ~~ \Im z > -\Re z + z^*\},\\
  \Omega_5 &= \{ z : -\alpha < \Im z < 0, ~~ \Im z > \Re z - z^*, ~~ \Im z > -\Re z - z^*\},\\
  \Omega_6 &= \{ z : -\alpha < \Im z < 0, ~~ \Im z > \Re z - z^*\}.
\end{aligned}
\end{equation}
There exists $\alpha> 0$, sufficiently small, so that $\mb L$ has an analytic extension to $\Omega_4 \cup \Omega_6$ and $\mb U$ has an analytic extension to $\Omega_1 \cup \Omega_3$.  Similarly, $\mb P_1$ and $\mb M_1$ have analytic extensions to $\Omega_2$ and $\Omega_5$, respectively.  So, define
\begin{align}\label{eq:N1}
  \tilde {\mb N}_1(z) = \mb N_1(z) \begin{cases} \mb U(z) & z \in \Omega_1 \cup \Omega_3,\\
    \mb P_1(z) & z \in \Omega_2,\\
    \mb L(z) & z \in \Omega_4 \cup \Omega_6,\\
    \mb M_1(z) & z \in \Omega_3. \end{cases}
\end{align}
The jump contours and jump matrices for the $\tilde {\mb N}_1$ are depicted in Figure~\ref{f:rhp1-disp1}.  
\begin{figure}[ht]
  \vspace{.3in}
  \centering
  \begin{overpic}[width=\linewidth]{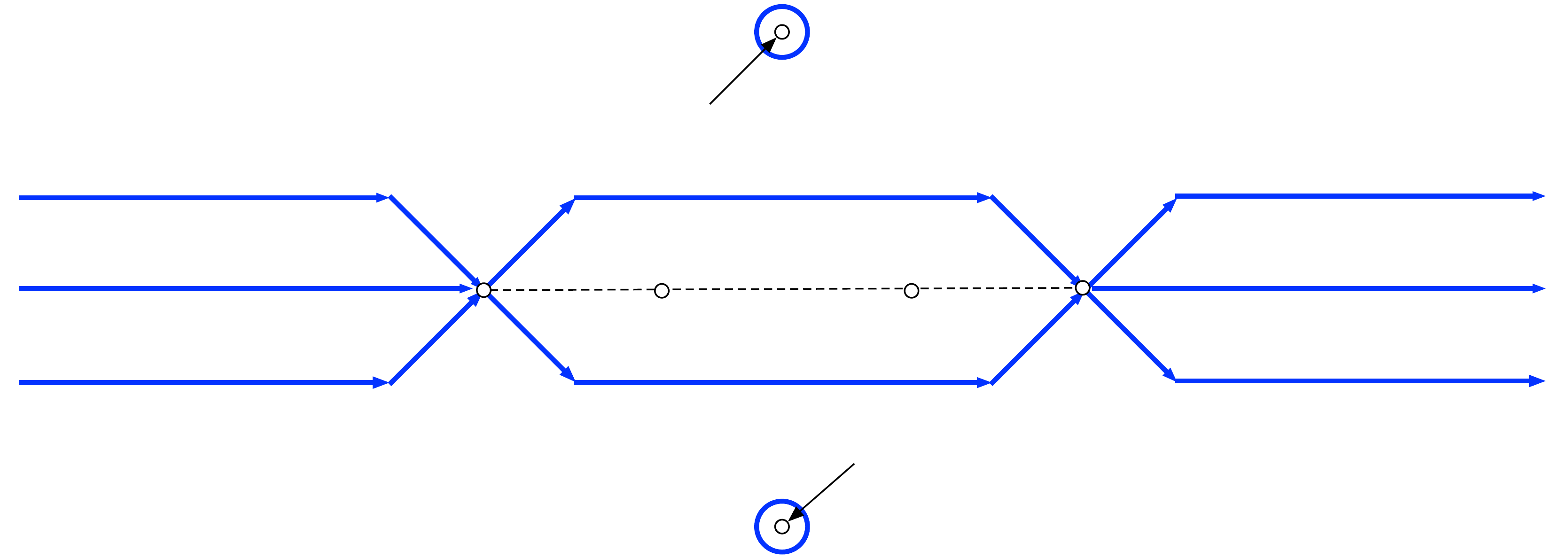}
    \put(28,12){ $-z^*$}
    \put(68,12){ $z^*$}
    \put(42,15){ $-c$}
    \put(55,15){ $c$}
    \put(35,18){ $\Omega_2$}
    \put(35,15){ $\Omega_5$}
    \put(24,18){ $\Omega_3$}
    \put(24,15){ $\Omega_4$}
    \put(72,18){ $\Omega_1$}
    \put(72,15){ $\Omega_6$}
    \put(82,24){\Large $\mb U^{-1}$}
    \put(82,18){\Large $\mb D$}
    \put(82,12){\Large $\mb L$}
    \put(15,24){\Large $\mb U^{-1}$}
    \put(15,18){\Large $\mb D$}
    \put(15,12){\Large $\mb L$}
    \put(49,24.3){\Large $\mb P_1^{-1}$}
    \put(49,12){\Large $\mb M_1$}
    \put(43,27){\large $z_j$}
    \put(55,7){\large $-z_j$}
  \end{overpic} 
  \caption{The jump contours and jump matrices for the unknown $\tilde {\mb N}_1$ defined in \eqref{eq:N1}. The contours are deformed within a strip of width $2 \alpha$. }\label{f:rhp1-disp1}
\end{figure}
We aim to have jumps that are localized at $\pm z^*$, and need to remove the jump on $(-\infty,-z^*) \cup (z^*,\infty)$.  Consider the RH problem
\begin{equation}
  \mb \Delta^+(s) = \mb \Delta^-(s) \mb D(s), \quad s \in (-\infty,-z^*) \cup (z^*,\infty), \quad \Delta(s) = I + O(s^{-1}) \quad s \to \infty.
\end{equation}
This is easily solved via the Cauchy integral
\begin{equation}
 \mb  \Delta(z) = \diag( \Delta(z), \Delta^{-1}(z) ), \quad \log \Delta(z) = \frac{1}{2 \pi \I} \int_{(-\infty,-z^*) \cup (z^*,\infty)} \frac{\log T(s)}{s-z} \D s.
\end{equation}
Now, fix $0 < r <  \delta$, and define
\begin{equation}
  \mb  \Sigma(z) = \begin{cases} \mb \Delta^{-1}(z) & z \not \in (-\infty,-z^*) \cup (z^*,\infty), ~~|z \pm z^*| > r,\\
    \mb I & |z+z^*| < r, ~~ \frac{3 \pi}{4} <\arg (z + z^*)  < \pi,\\
    \mb D(z) & |z+z^*| < r, ~~ -\pi < \arg (z + z^*)  < - 3 \pi/4,\\
    \mb L(z) \mb D(z) & |z+z^*| < r, ~~ -\frac{3 \pi}{4} < \arg (z + z^*)  < -\frac{\pi}{4},\\
    \mb P(z)\mb U^{-1}(z) & |z+z^*| < r, ~~ -\frac{\pi}{4} < \arg (z + z^*)  < \frac{\pi}{4},\\
    \mb U^{-1}(z) & |z+z^*| < r, ~~ \frac{\pi}{4} < \arg (z + z^*)  < \frac{3\pi}{4},\\
    \mb I & |z-z^*| < r, ~~ 0 <\arg (z - z^*)  < \frac{\pi}{4},\\
    \mb D(z) & |z-z^*| < r, ~~ -\frac{\pi}{4} < \arg (z - z^*)  < 0,\\
    \mb L(z) \mb D(z) & |z-z^*| < r, ~~ -\frac{3 \pi}{4} < \arg (z - z^*)  < -\frac{\pi}{4},\\
    \mb P(z)\mb U^{-1}(z) & |z-z^*| < r, ~~ -\pi < \arg (z - z^*)  < -\frac{3\pi}{4},\\
    \mb P(z)\mb U^{-1}(z) & |z-z^*| < r, ~~ \frac{3\pi}{4} < \arg (z - z^*)  \leq \pi,\\
    \mb U^{-1}(z) & |z-z^*| < r, ~~ \frac{\pi}{4} < \arg (z - z^*)  < \frac{3\pi}{4}.
    \end{cases}
\end{equation}
From this we define
\begin{align}\label{eq:hN1}
  \hat {\mb N}_1(z) = \tilde {\mb N}_1(z) \mb \Sigma(z).
\end{align}
The jump contours and jump matrices for $\hat {\mb N}_1(z)$ are displayed in Figure~\ref{f:rhp1-disp2} with a zoomed view given in Figure \ref{f:rhp1-dispzoom}.  Before this RH problem is discretized and solved, the transformation discussed in Section~\ref{sec:sigma} is performed.

\begin{figure}[ht]
  \vspace{.3in}
  \centering
  \begin{overpic}[width=\linewidth]{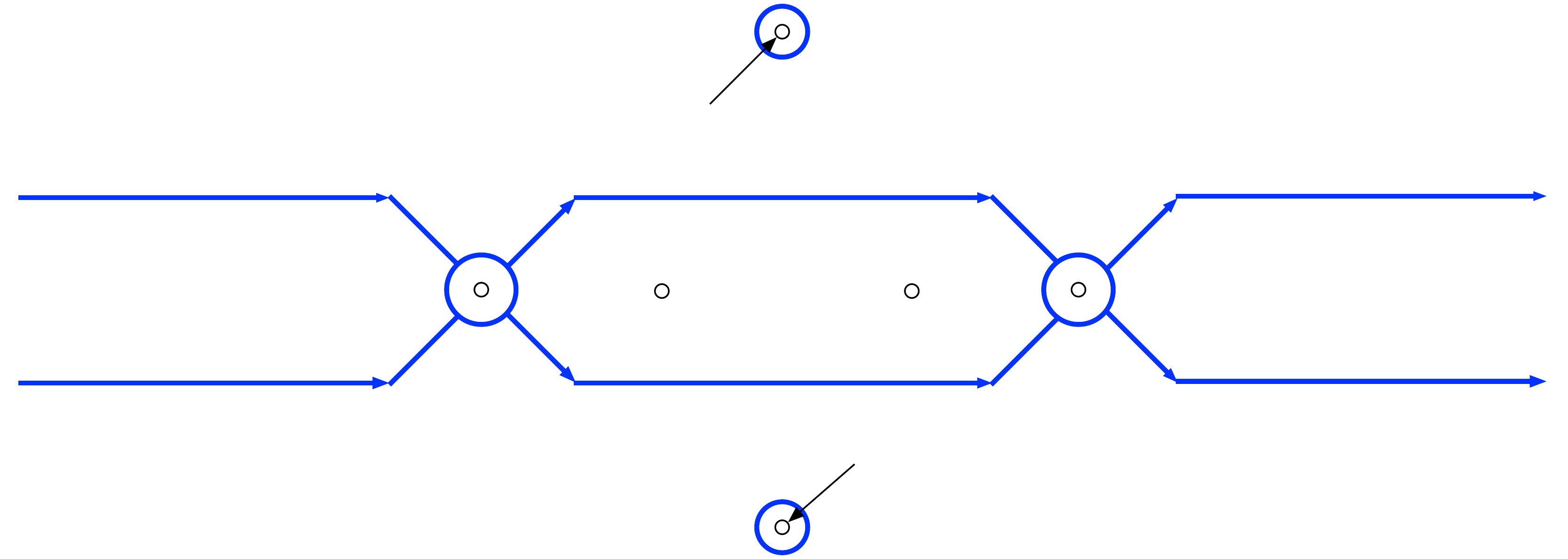}
    \put(28,12){ $-z^*$}
    \put(68,12){ $z^*$}
    \put(42,15){ $-c$}
    \put(55,15){ $c$}
    \put(82,24){\Large $\mb \Delta\mb U^{-1}\mb \Delta^{-1}$}
    \put(82,12){\Large $\mb \Delta\mb L\mb \Delta^{-1}$}
    \put(15,24){\Large $\mb \Delta\mb U^{-1}\mb \Delta^{-1}$}
    \put(15,12){\Large $\mb \Delta\mb L\mb \Delta^{-1}$}
    \put(45,24.3){\Large $\mb \Delta\mb P_1^{-1}\mb \Delta^{-1}$}
    \put(45,12){\Large $\mb \Delta\mb M_1\mb \Delta^{-1}$}
    \put(43,27){\large $z_j$}
    \put(55,7){\large $-z_j$}
  \end{overpic} 
  \caption{The jump contours and jump matrices for the unknown $\hat {\mb N}_1$ defined in \eqref{eq:hN1}.  The contours are deformed within a strip of width $2 \alpha$. }\label{f:rhp1-disp2}
\end{figure}

\begin{figure}[ht]
  \vspace{.5in}
  \centering
  \begin{overpic}[width=\linewidth]{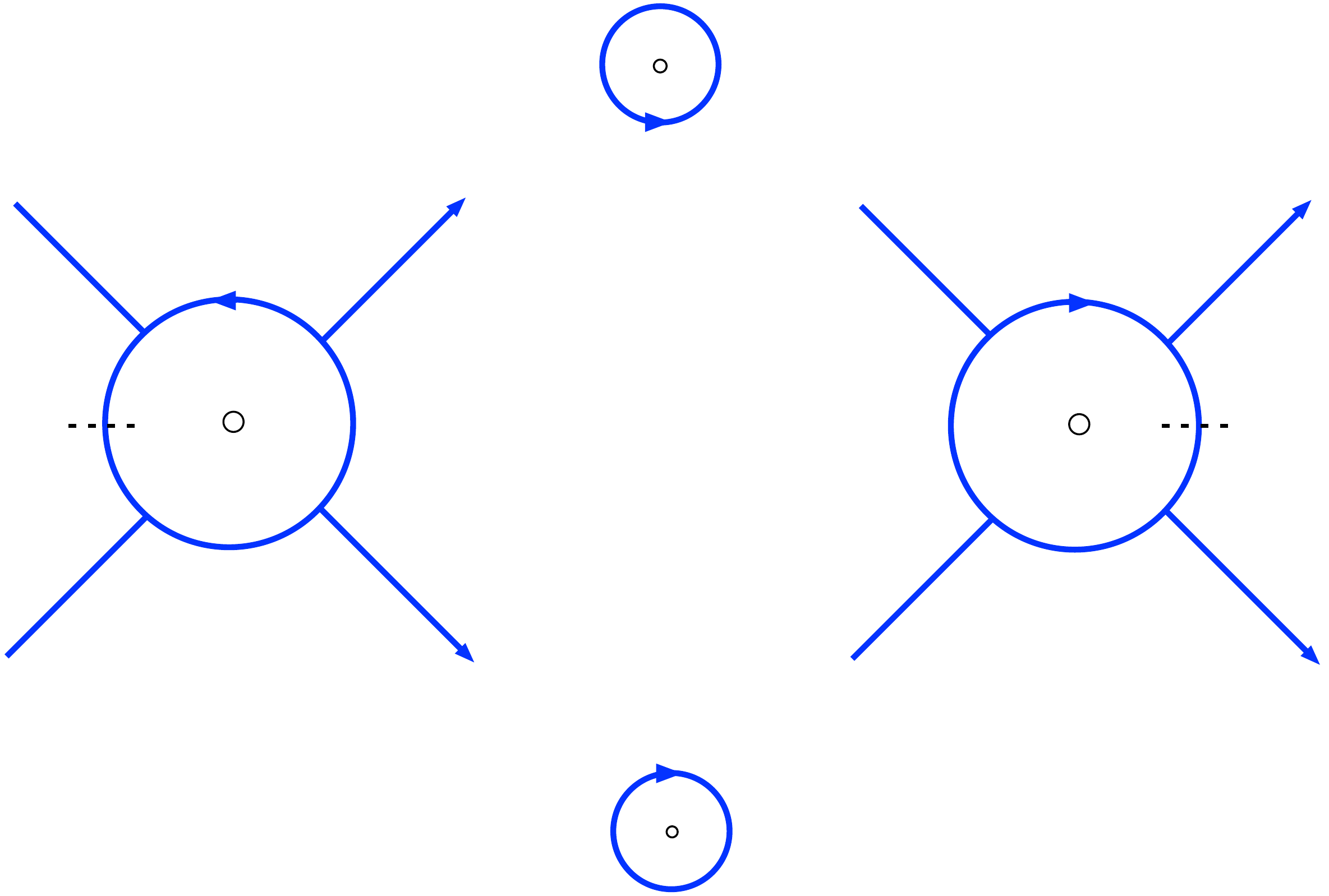}
    \put(17,32){ $-z^*$}
    \put(78,32){ $z^*$}
    \put(94,53){\Large $\mb \Delta\mb U^{-1}\mb \Delta^{-1}$}
    \put(94,14){\Large $\mb \Delta\mb L\mb \Delta^{-1}$}
    \put(0,53){\Large $\mb \Delta\mb U^{-1}\mb \Delta^{-1}$}    
    \put(0,14){\Large $\mb \Delta\mb L\mb \Delta^{-1}$}

    \put(14,46){\Large $\mb \Delta\mb U^{-1}$}
    \put(9,37){\Large $\mb \Delta$}
    \put(9,32){\Large $\mb \Delta \mb D$}
    \put(27,34.5){\Large $\mb P \mb U^{-1}$}
    \put(14,23){\Large $\mb \Delta\mb L \mb D$}

    \put(78,46){\Large $\mb U\mb \Delta^{-1}$}
    \put(84,37){\Large $\mb \Delta^{-1}$}
    \put(91,32){\Large $\mb D^{-1} \mb \Delta^{-1}$}
    \put(64,34.5){\Large $\mb U \mb P^{-1}$}
    \put(74,23){\Large $\mb D^{-1}\mb L^{-1} \mb \Delta^{-1}$}

    \put(32,53){\Large $\mb \Delta\mb P_1^{-1}\mb \Delta^{-1}$}
    \put(32,14){\Large $\mb \Delta\mb M_1\mb \Delta^{-1}$}

    \put(60,53){\Large $\mb \Delta\mb P_1^{-1}\mb \Delta^{-1}$}
    \put(60,14){\Large $\mb \Delta\mb M_1\mb \Delta^{-1}$}
    
    \put(50,63.5){\large $z_j$}
    \put(50,5.3){\large $-z_j$}  

    \put(55,63.5){\Large $\mb \Delta \mb S_j \mb \Delta^{-1}\quad \mb S_j(z) = \begin{bmatrix} 1 & 0 \\ - \frac{c(z_j)}{z-z_j} \E^{- 2 \I z_j x - 8 \I z_j^3 t} &1 \end{bmatrix}$}
    \put(56,5.3){\Large $\mb \Delta \mb S_j^{-T}(-\cdot) \mb \Delta^{-1}$}
  \end{overpic} 
  \caption{ A zoomed view of the jump contours and matrices for $\hat {\mb N}_1$. }\label{f:rhp1-dispzoom}
\end{figure}
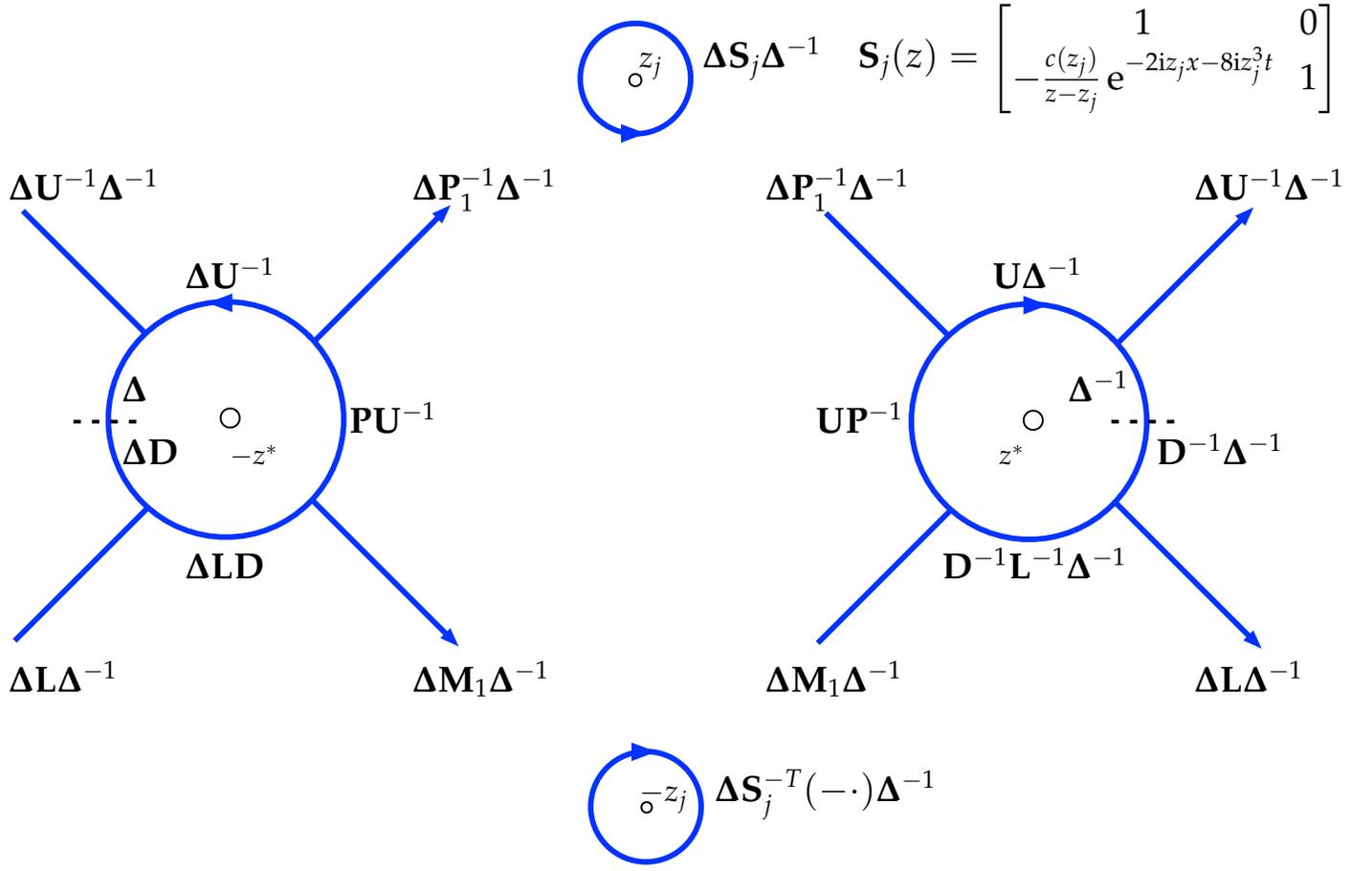

This deformation, following the arguments in \cite{TrogdonSOBook}, give accurate computations for all $(x,t)$ such that $z^*\geq c + \delta$, even as $t \to \infty$.  As $t$ increases, one has to vary $r$ and $r \sim t^{-1/2}$ is seen to be an acceptable choice \cite{TrogdonSOKdV}.

\section{Numerical examples}
\label{sec:KdV-num-ex}
Combining the two deformations discussed in the previous section, numerical computations will be accurate asymptotically\footnote{This means that computations will be accurate for all $x$ and $t$ in these regions including both large and small values.} for
\begin{equation}
  x \leq -12 (c+ \delta)^2 \quad \text{and} \quad -2c^2t \leq x.
\end{equation}
This leaves a rather large sector of the $(x,t)$ plane unaccounted for.  A future work will focus on properly filling this gap.

Nevertheless, we can compute the entire solution profile for a restricted interval of $t$ values, provided that $c$ is not too large.  To accomplish this, we made an \emph{ad hoc} modification of $z^*$:
\begin{equation}
  z^*_m = \max\{z^*, c+ \delta\},
\end{equation}
where, in practice we set $\delta = 1/10$.  And then we use the deformation and RH problem displayed in Figure~\ref{f:rhp1-disp2} for $x < -2c^2t$ with $z^*$ replaced with $z^*_m$ and the deformation and RH problem displayed in Figure~\ref{f:rhp2-t=0:2} for $x \geq -2c^2t$.

The initial data $u(x,0)$ in our examples satisfies
\begin{equation}
  u(x,0) \to c^2, \quad x \to -\infty \quad \text{and}\quad u(x,0) \to 0, \quad x \to + \infty.
\end{equation}
It is simple to use the Galilean boost to map such a solution to one satisfying \eqref{eq:u0}, see Remark~\ref{r:boost}. 

\begin{remark}
  Evaluating $u(x,t)$ for small $t$ can be difficult if $R_{\mathrm l}(z)$ and $R_{\mathrm r}(z)$ do not decay quickly as $z\to\pm\infty$.  This issue is analogous to computing the Fourier transform of a function that decays slowly at infinity --- one cannot truncate the domain of integration enough to allow for the capturing of oscillation.  But for $t > 0$, the deformations outlined in the previous section induce exponential decay, alleviating this issue to an extent.  Indeed, as $t \downarrow 0$ the additional decay is reduced.

  For infinitely smooth initial data $u(x,0)$, from Lemma~\ref{l:decay}, this is not an issue even as $t$ approaches zero.  So, we are able to evaluate the solution profile for all $x$ and $t \in [0,T]$.  In our computations $T \approx 1$.

  For discontinous initial data $u(x,0)$, $t \downarrow 0$ is a singular limit and the deformations described only allow for the computation for all $x$ but $t \in [\epsilon,T]$, $\epsilon > 0$.
\end{remark}

\subsection{$u_0 = 0$}

When $u_0 = 0$, the functions $A,B,a$ and $b$ can be determined explicitly
\begin{equation}
\begin{aligned}
  A(z) &= \frac{1}{2} \left( 1 + \frac{z}{\lambda(z)} \right), \quad B(z) = \frac{1}{2} \left( 1 - \frac{z}{\lambda(z)} \right),\\
  a(z) & = \frac{z+\lambda(z)}{2 z}, \quad b(z) = \frac{z- \lambda(z)}{2z}.
\end{aligned}
\end{equation}
We display the solution of \eqref{eq:KdV} with $u(x,0) = H_c(x) + c^2$ for various values of $c$, all evaluated at $t = 1$.

\begin{figure}[ht]
  \centering
  \begin{overpic}[width=.95\linewidth]{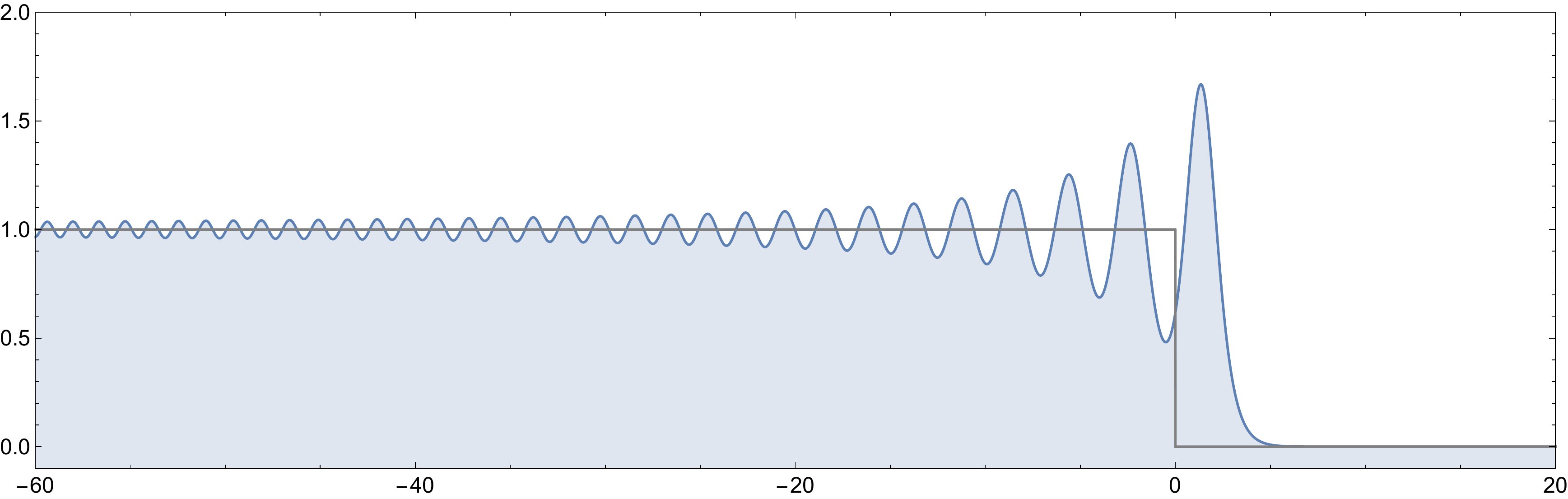}
    \put(50,-2){$x$}
    \put(-3,13){\rotatebox{90}{$u(x,1)$}}
  \end{overpic}
  
  \vspace{.2in}
  
\begin{overpic}[width=.95\linewidth]{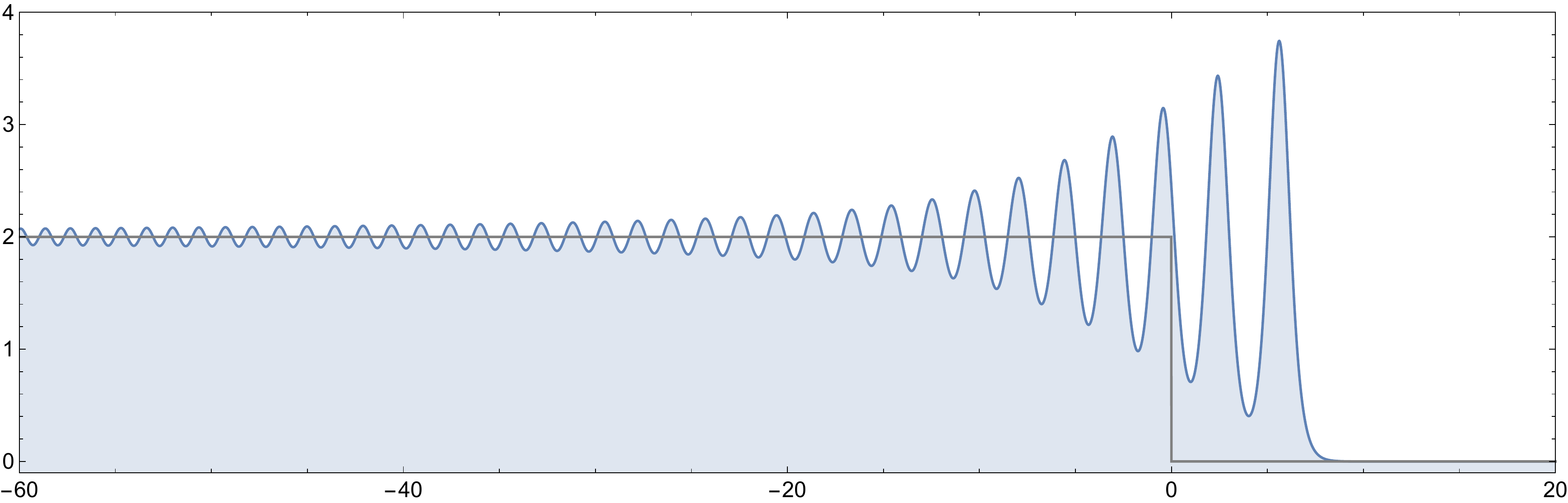}
    \put(50,-2){$x$}
    \put(-3,13){\rotatebox{90}{$u(x,1)$}}
  \end{overpic}

  \vspace{.2in}

  \centering
  \begin{overpic}[width=.95\linewidth]{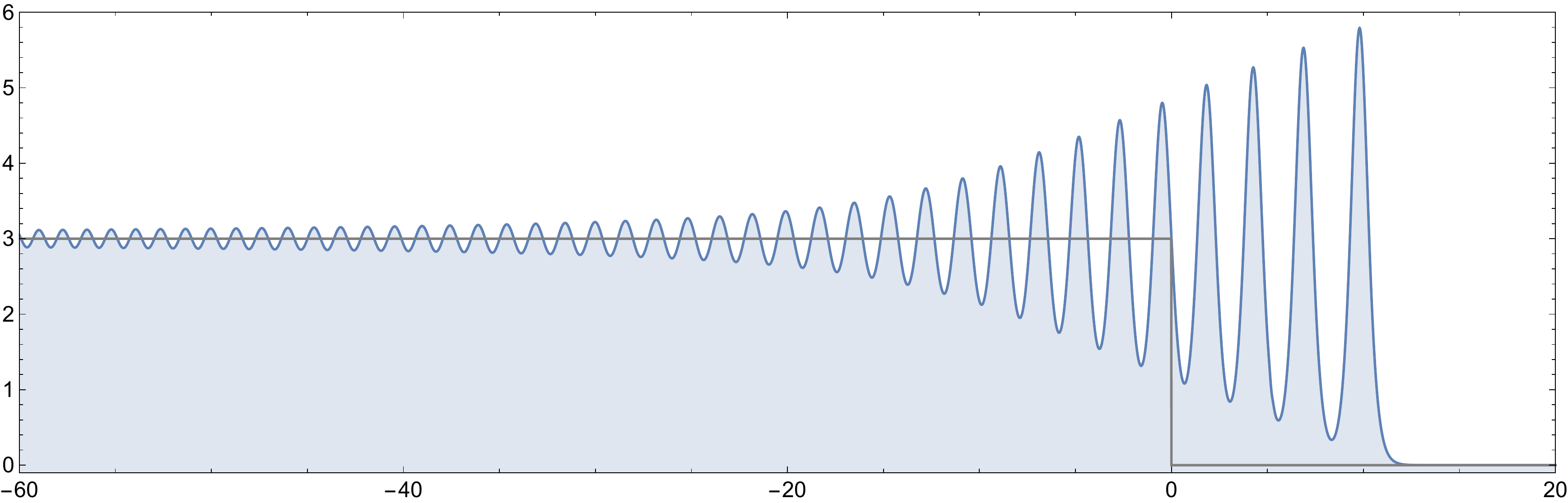}
    \put(50,-2){$x$}
    \put(-3,13){\rotatebox{90}{$u(x,1)$}}
  \end{overpic}
  
  \caption{  The solution of the KdV equation at $t = 1$ when $u(x,0) = H_c(x) + c^2$, $c = 1$ (top), $c = \sqrt{2}$ (middle) and $c = \sqrt{3}$ (bottom).}\label{f:dsw:2}
\end{figure}
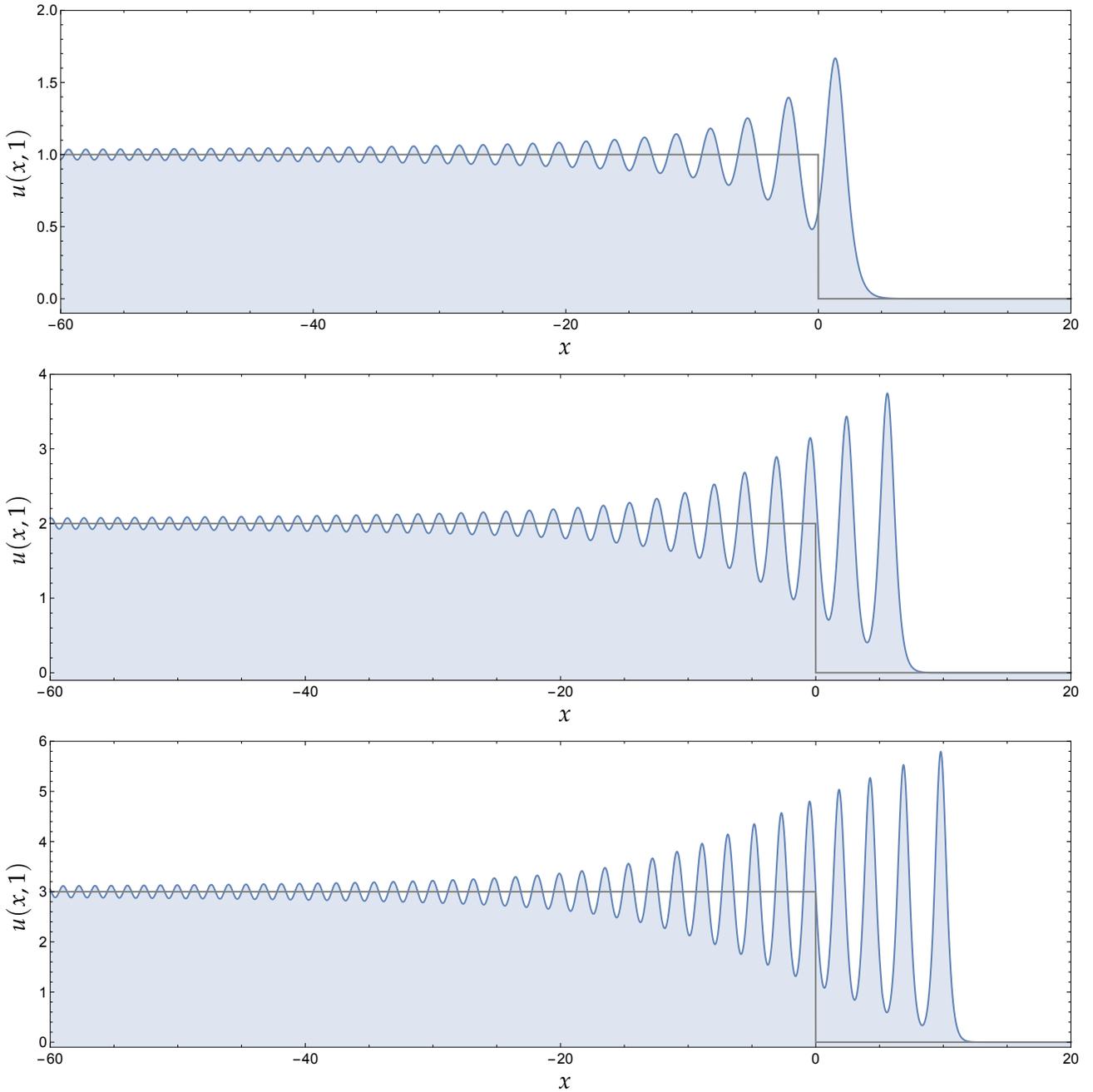

\subsection{Smooth soliton-free data}

An example of smooth data that fits into the described framework is
\begin{align}\label{eq:smoothfree}
u(x,0) = \frac{1}{4} ( 1 + \mathrm{erf}(x))^2,
\end{align}
where $\mathrm{erf}(x)$ is the error function \cite{DLMF}.  In this case, computing $R_{\mathrm l}$ and $R_{\mathrm r}$ is non-trivial.  We display these functions in Figures~\ref{f:freeleft} and \ref{f:freeright}, noting that the decay of $u_0$ makes $A,B,a$ and $b$ analytic functions of $z$ for all $z$ off the cut $[-c,c]$.  The corresponding solution is given in Figure~\ref{f:smoothfree}
\begin{figure}[ht]
    
\begin{minipage}[b]{.45\linewidth}
  \centering

  \vspace{.1in}
  
   \begin{overpic}[width=.94\linewidth]{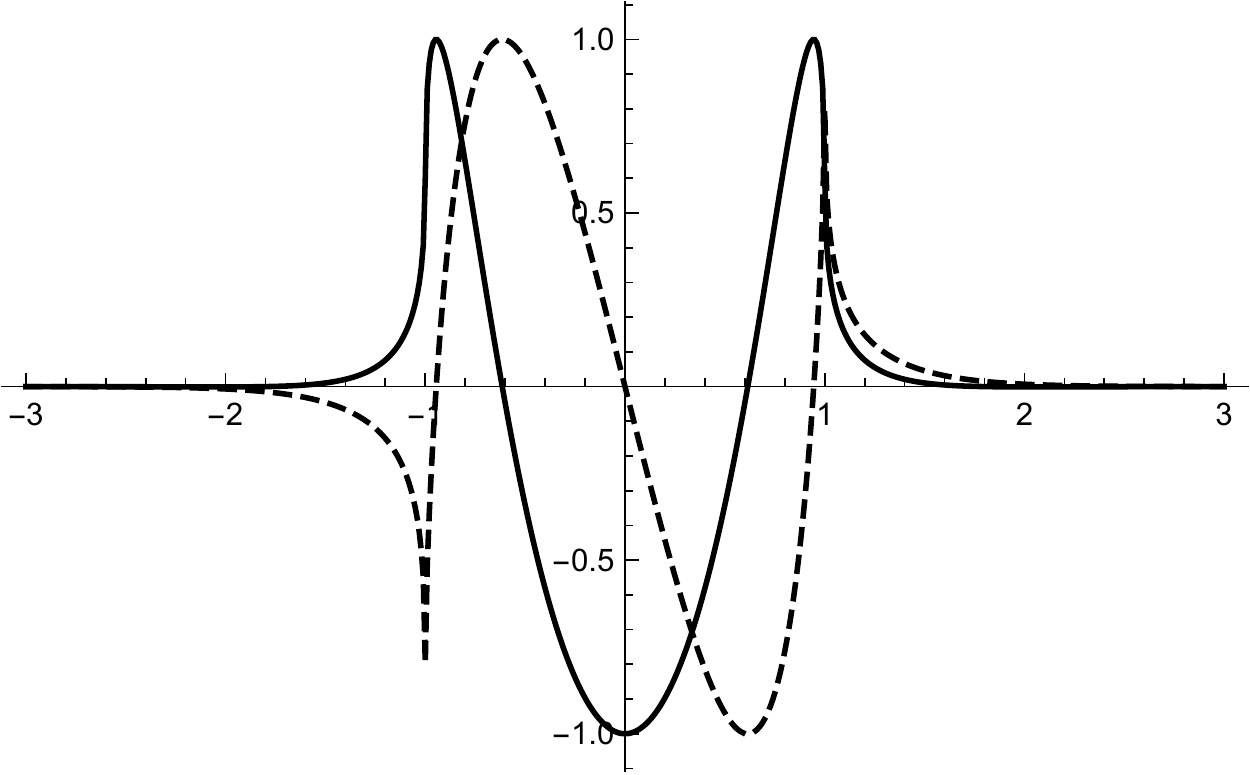}
    \put(46,64){$R_{\mathrm l}(z)$}
    \put(-4,30){$z$}  
  \end{overpic}
  
  \vspace{.1in}
  
\subcaption{The real (solid) and imaginary (dashed) parts of $R_{\mathrm l}(z)$ when $u(x,0)$ is given in \eqref{eq:smoothfree}.}\label{f:freeleft}  
\end{minipage}%
\hspace{.1in}
\begin{minipage}[b]{.45\linewidth}
  \centering

  \vspace{.1in}
  
   \begin{overpic}[width=.94\linewidth]{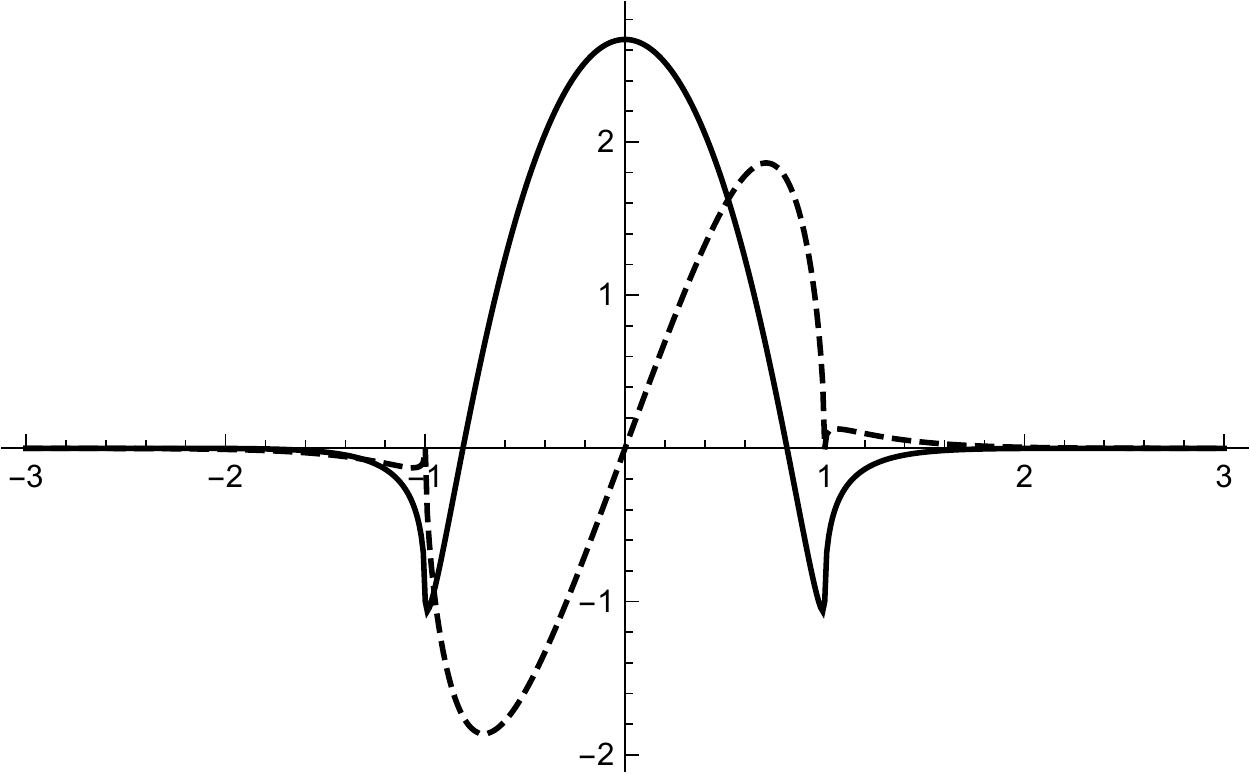}
    \put(46,64){$R_{\mathrm r}(z)$}
    \put(-4,26){$z$}  
  \end{overpic}

  \vspace{.1in}
  
\subcaption{The real (solid) and imaginary (dashed) parts of $R_{\mathrm r}(z)$ when $u(x,0)$ is given in \eqref{eq:smoothfree}.}\label{f:freeright}
\end{minipage}%

  \caption{The right and left reflection coefficients for \eqref{eq:smoothfree}.}
\end{figure}

\begin{figure}[ht]
  \centering
  \begin{overpic}[width=.95\linewidth]{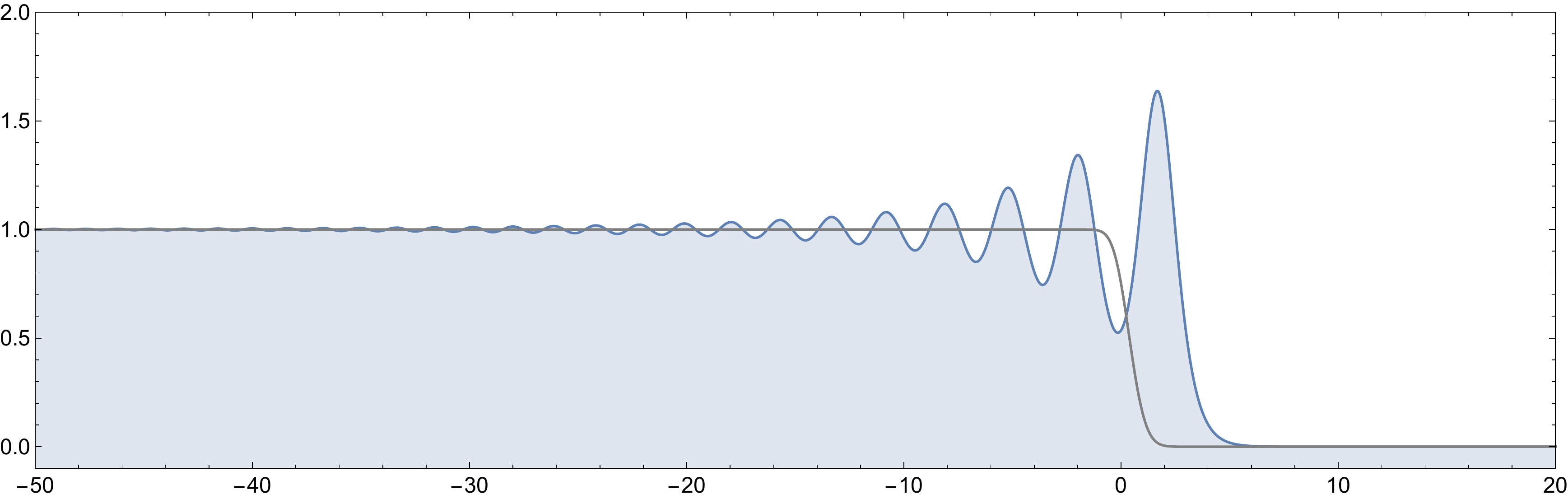}
    \put(50,-2){$x$}
    \put(-3,13){\rotatebox{90}{$u(x,1)$}}
  \end{overpic}
  
  \vspace{.2in}
  
\begin{overpic}[width=.95\linewidth]{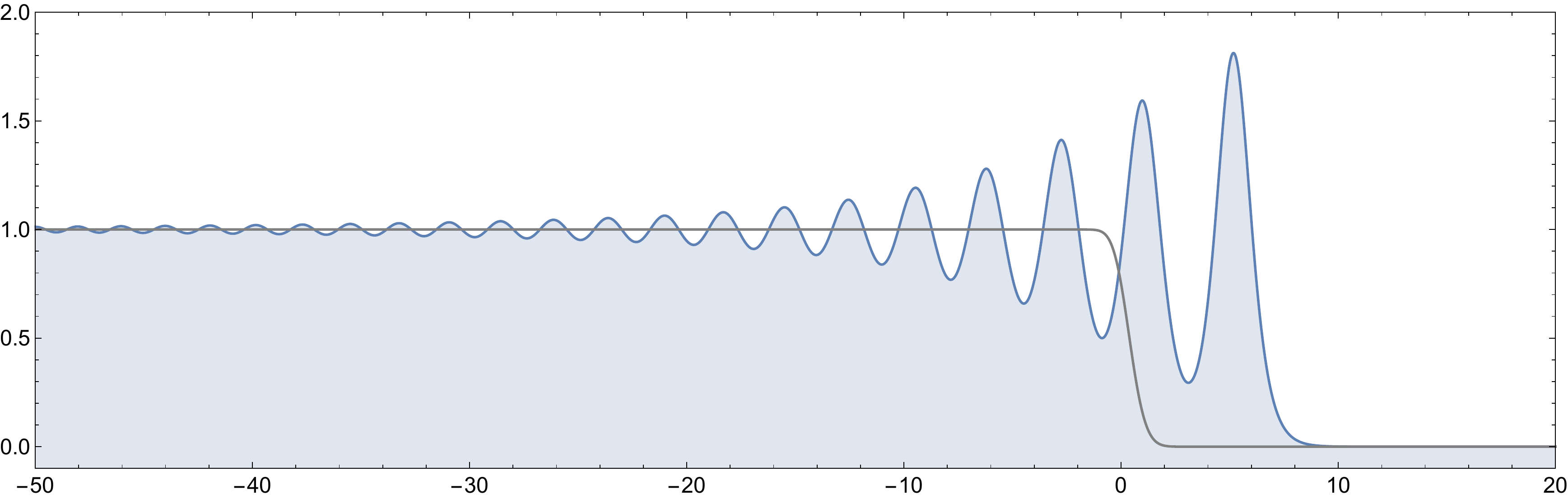}    
    \put(50,-2){$x$}
    \put(-3,13){\rotatebox{90}{$u(x,2)$}}
  \end{overpic}

  \vspace{.2in}
  
\begin{overpic}[width=.95\linewidth]{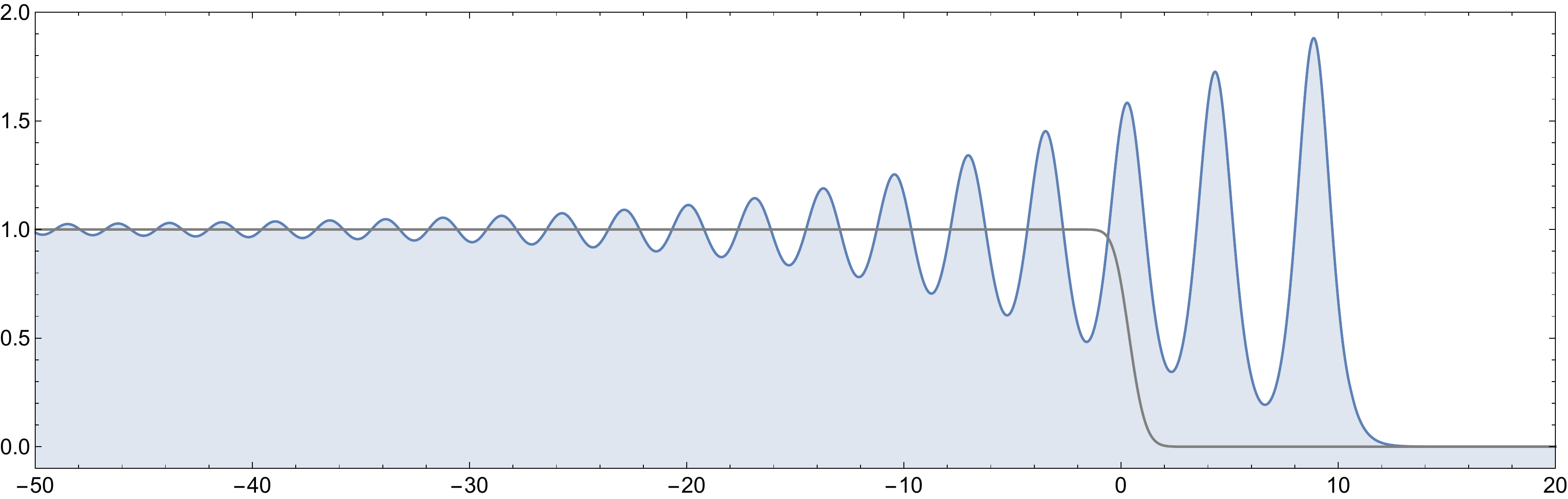}    
    \put(50,-2){$x$}
    \put(-3,13){\rotatebox{90}{$u(x,3)$}}  
  \end{overpic}
  \caption{  The solution of the KdV equation at $t = 1,2,3$ when $u(x,0)$ is given by \eqref{eq:smoothfree}.  The gray curve indicates the initial condition.}\label{f:smoothfree}
\end{figure}

\subsection{Smooth data with a soliton}
An example of smooth data that fits into the described framework but produces a soliton is
\begin{align}\label{eq:smooth}
u(x,0) = \frac{1}{4} ( 1 + \mathrm{erf}(x))^2 + 2 \E^{-x^2/2}.
\end{align}
The reflection coefficients are given in Figures~\ref{f:left} and \ref{f:right}.    The data associated to the pole in the RH problem is given by
\begin{equation}
\begin{aligned}
  z_1 &\approx 0.950681 \I,\\
  c(z_1) &\approx 3.48119 \I,\\
  C(z_1) &\approx 3.90351 \I.  
\end{aligned}
\end{equation}
The corresponding solution is displayed in Figure~\ref{f:smooth}.

\begin{remark}[Soliton speed]
  The speed of the soliton can be easily read off from the RH problem.  For example, the jump on $\Sigma_j$ in \rhref{rhp:1t} is determined by
\begin{equation}
\begin{aligned}
    \E^{-2 \I z_j x - 8 \I z_j^3t } = \E^{-2 \I z_j ( x + 4 z_j^2 t)}.
\end{aligned}
\end{equation}
  This indicates a velocity of $-4 z_j^2$ for $x \ll 0$, in the case of data decaying to $0$ at $-\infty$ and tending to $-c^2$ at $+\infty$. In the current setting, this gives a velocity of $-4 z_j^2 + 6 c^2$.  Similarly, for $x \gg 0$ we consider the exponential in the jump on $\Sigma_j$ in \rhref{rhp:2t}
  \begin{equation}
\begin{aligned}
    \E^{2 \I \lambda(z_j) x + 8 \I \lambda^3(z_j)t + 12 \I c^2 \lambda(z_j) t} = \E^{2 \I \lambda(z_j) ( x + 6 c^2 t + 4 (z_j^2 - c^2) t)}.
 \end{aligned}
\end{equation}
  This indicates a velocity of $-4 z_j^2 - 2c^2$, in the case of data decaying to $0$ at $-\infty$ and tending to $-c^2$ at $+\infty$.  For the current setting of \eqref{eq:smooth}, the velocity is $-4 z_j^2 + 4c^2$, a decrease in velocity of $2 c^2$.
\end{remark}

\begin{figure}[ht]

\begin{minipage}[b]{.45\linewidth}
  \centering

  \vspace{.1in}
  
   \begin{overpic}[width=.94\linewidth]{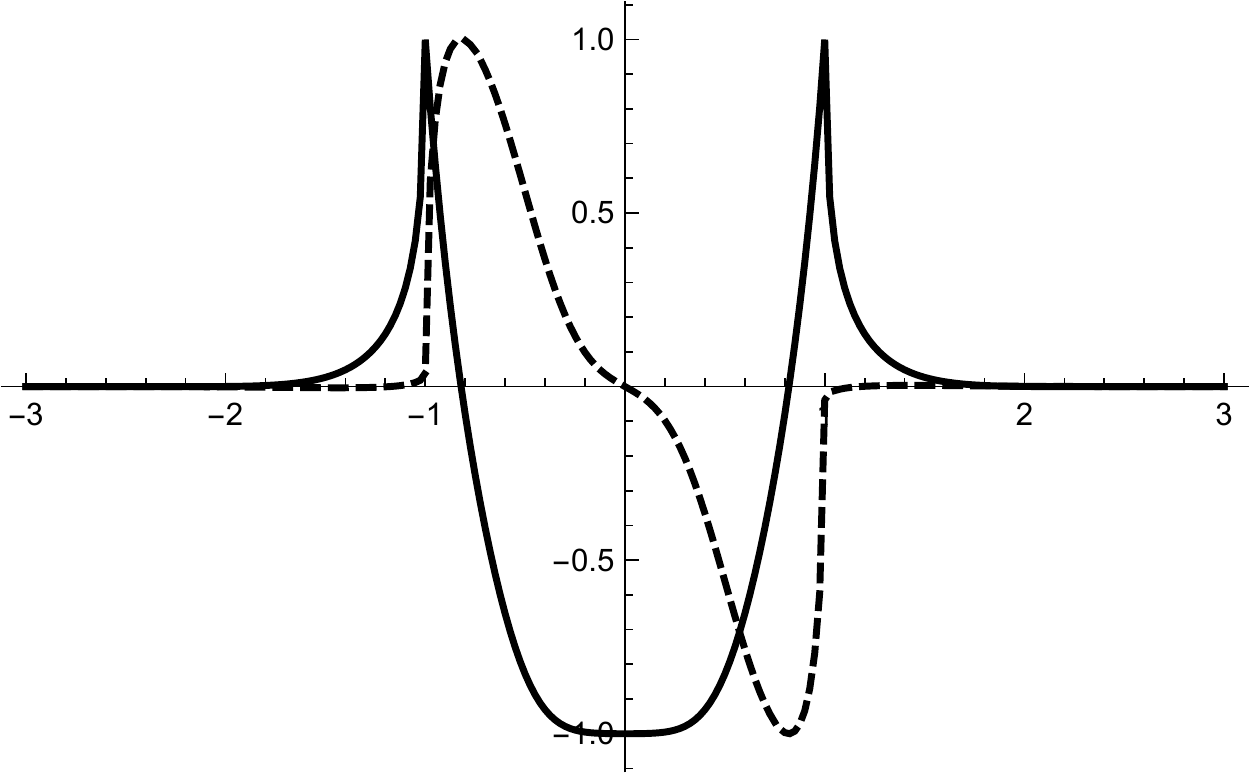}
    \put(46,64){$R_{\mathrm l}(z)$}
    \put(-4,30){$z$}  
  \end{overpic}
  
  \vspace{.1in}
  
\subcaption{The real (solid) and imaginary (dashed) parts of $R_{\mathrm l}(z)$ when $u(x,0)$ is given in \eqref{eq:smooth}.}\label{f:left}  
\end{minipage}%
\hspace{.1in}
\begin{minipage}[b]{.45\linewidth}
  \centering

  \vspace{.1in}
  
   \begin{overpic}[width=.94\linewidth]{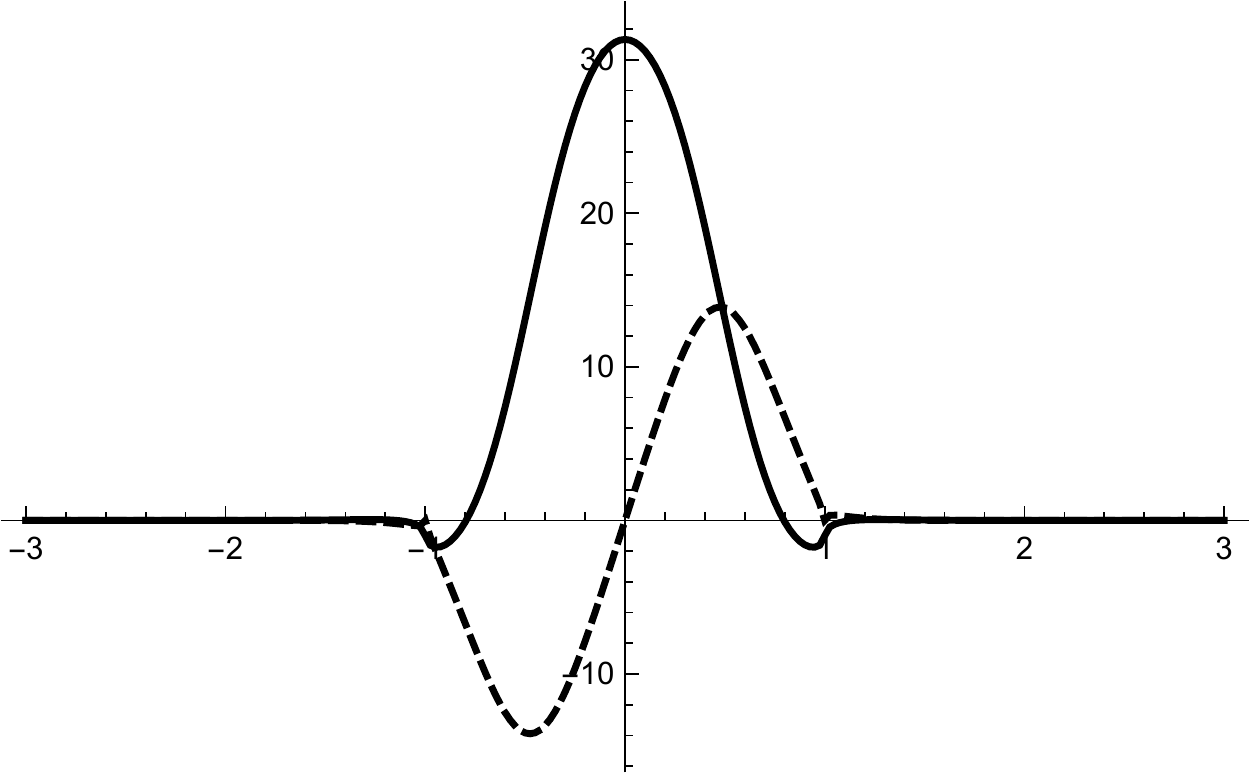}
    \put(46,64){$R_{\mathrm r}(z)$}
    \put(-4,19){$z$}  
  \end{overpic}

  \vspace{.1in}
  
\subcaption{The real (solid) and imaginary (dashed) parts of $R_{\mathrm r}(z)$ when $u(x,0)$ is given in \eqref{eq:smooth}.}\label{f:right}
\end{minipage}%

  \caption{The right and left reflection coefficients for the data \eqref{eq:smooth}.}
\end{figure}

\begin{figure}[ht]
  \centering
  \begin{overpic}[width=.95\linewidth]{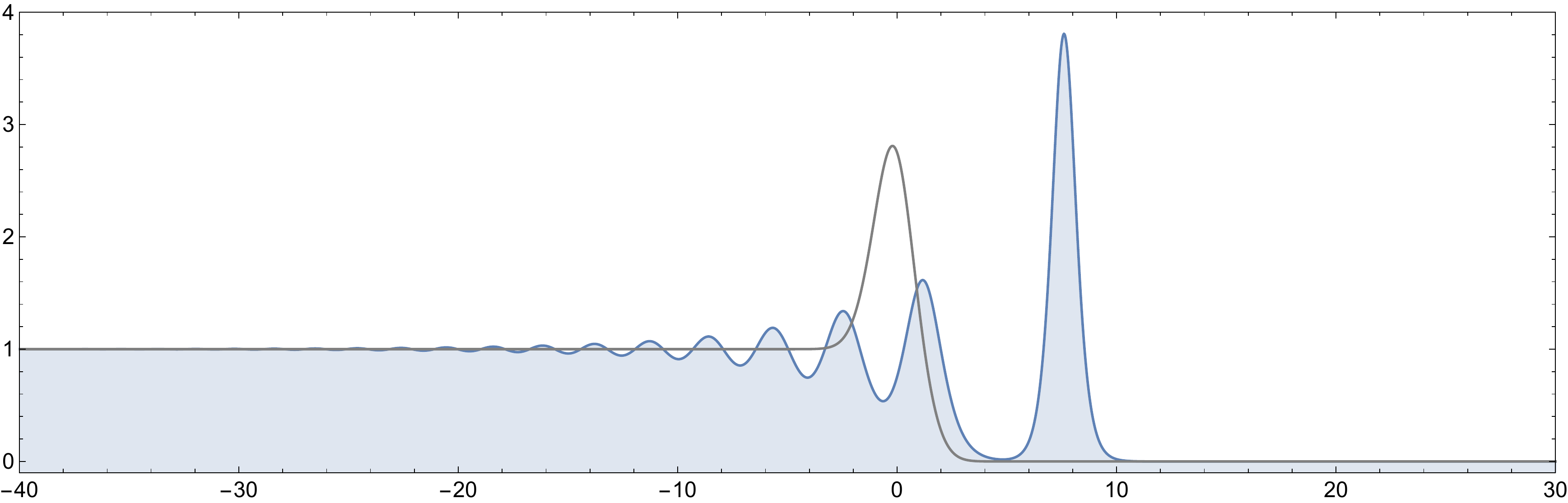}
    \put(50,-2){$x$}
    \put(-3,13){\rotatebox{90}{$u(x,1)$}}
  \end{overpic}
  
  \vspace{.2in}
  
\begin{overpic}[width=.95\linewidth]{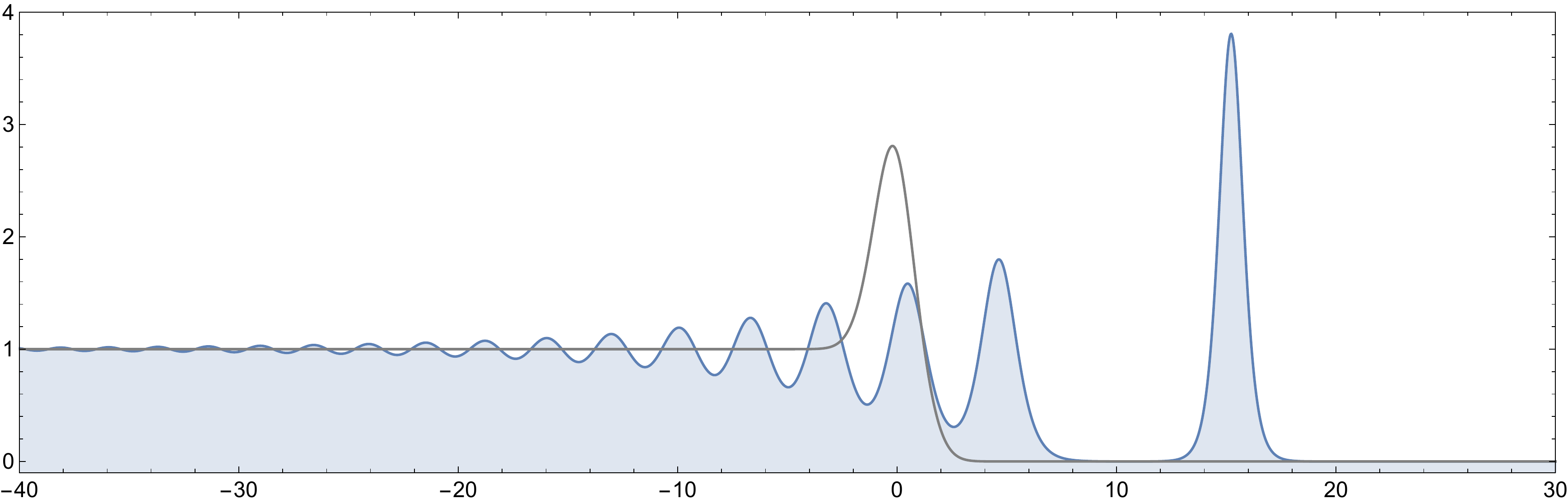}    
    \put(50,-2){$x$}
    \put(-3,13){\rotatebox{90}{$u(x,2)$}}
  \end{overpic}

  \vspace{.2in}
  
\begin{overpic}[width=.95\linewidth]{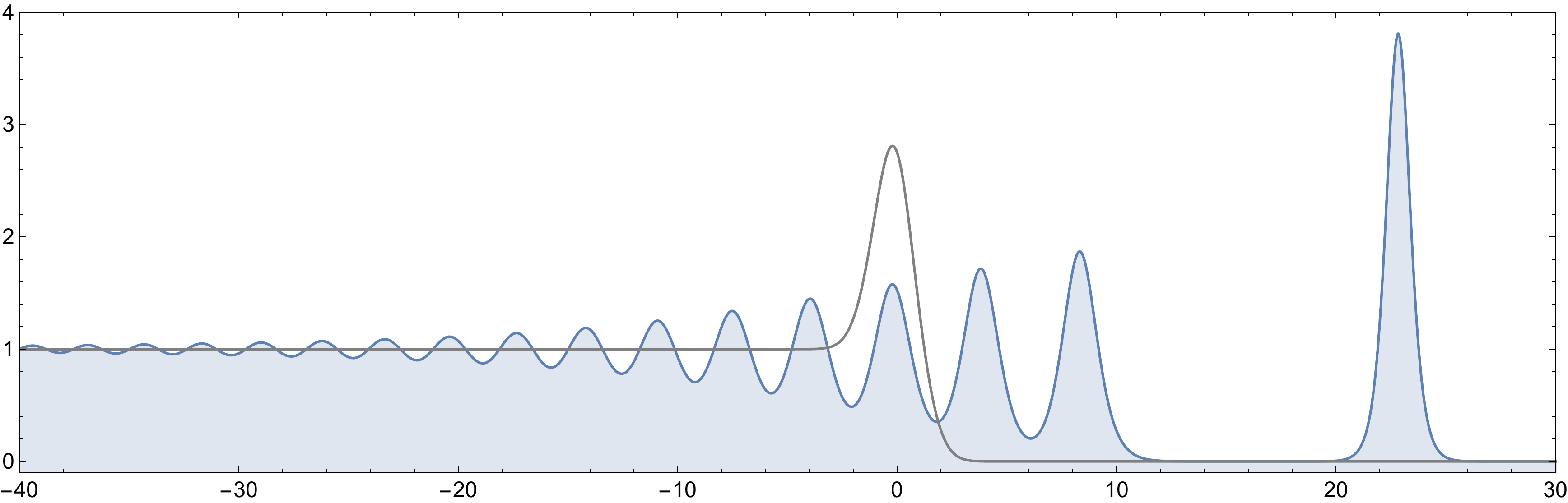}    
    \put(50,-2){$x$}
    \put(-3,13){\rotatebox{90}{$u(x,3)$}}  
  \end{overpic}
  \caption{  The solution of the KdV equation at $t = 1,2,3$ when $u(x,0)$ is given by \eqref{eq:smooth}. The gray curve indicates the initial condition.}\label{f:smooth}
\end{figure}

\appendix

\section{Solitons and time-dependence}
\label{sec:KdV-Time}
We derive time dependence of the scattering data under the assumption that $u_0(\cdot, t) = u(\cdot,t) - H_c(\cdot)$ and its $x$ derivative decay rapidly at infinity for all $t$.  After the time dependence is determined, one can appeal to the so-called Dressing Method to show that if the solution of the RH problem exists and is unique, then it produces a solution of the KdV equation (see \cite[Proposition 12.1]{TrogdonSOBook}, for example).

We have defined the (partial) scattering map $\mathcal Su_0 = (R_{\mathrm l}, R_{\mathrm r})$.  Define $R_{\mathrm r}(z;t)$ and $R_{\mathrm l}(z;t)$ by the mapping
\begin{equation}
  \mathcal S(u(\cdot,t) - H_c)   =  (R_{\mathrm l}(\cdot;t), R_{\mathrm r}(\cdot;t)).
\end{equation}
where $u(x,t)$ is the solution of the KdV equation with initial data $u_0 + H_c$.  The map gives only the partial scattering data because we have not yet incorporated discrete spectrum, i.e., solitons.  Define $a(z;t)$, $b(z;t)$, $A(z;t)$ and $B(z;t)$ to be the  functions corresponding to $u(\cdot,t) - H_c$.

Extend the  solutions $\phi^{\text{p,m}}(z;x)$ and $\psi^{\text{p,m}}(z;x)$ to functions $\phi^{\text{p,m}}(z;x,t)$ and $\psi^{\text{p,m}}(z;x,t)$ by replacing $u_0(x)$ with $u(x,t)$.  These functions satisfy the following scattering and evolution equations (scalar Lax pair):
\begin{equation}
\begin{aligned}
-\phi_{xx} - u(x,t)\phi &= z^2\phi, \\
\phi_t &= ( 4z^2 - 2 u (x,t) )\phi_x + (u_x(x,t) + 1) \phi.
\end{aligned}
\end{equation}
The compatibility condition $\phi_{xx t} = \phi_{txx}$ with the condition $z_t =0$ gives the KdV equation \eqref{eq:KdV}.  Consider, now with time dependence, for $z \in \mathbb R$,
\begin{equation}
\begin{aligned}
   \psi^{\text{p}}(z;x,t) &=  a(z;t) \phi^{\text{p}}(z;x,t)  + b(z;t) \phi^{\text{m}}(z;x,t),\\
   \phi^{\text{m}}(z;x,t) &=  B(z;t)  \psi^{\text{p}}(z;x,t)  + A(z;t) \psi^{\text{m}}(z;x,t).
\end{aligned}
\end{equation}
So, for $t$ and $z^2 > c^2$ fixed, we have
\begin{multline}
  a_t(z;t) \phi^{\text p} + a(z;t) \phi^{\text p} + b_t(z;t) \phi^{\text m} + b(z;t) \phi^{\text m} \\
   = (4 z^2 - 2 u(x,t) ) a(z;t) \phi_x^{\text p} + (4 z^2 - 2 u(x,t) )b(z;t)\phi_x^{\text m} + (u_x(x,t) + 1) (a(z;t)\phi^{\text p} + b(z;t)\phi^{\text m}).
\end{multline}
Then as $x \to - \infty$,
\begin{equation}
   \phi_x^{\text p}(z;x,t) = \I z \phi^{\text p}(z;x,t)( 1 + o(1)), \quad \phi_x^{\text m}(z;x,t) = -\I z \phi^{\text m}(z;x,t)( 1 + o(1)).
\end{equation}
Using that $u(x,t), u_x(x,t) \to 0$ as $x \to - \infty$, we find
\begin{equation}
  (a_t(z;t) - 4 \I z^3 a(z;t) )\phi^{\text p} + (b_t(z;t) + 4 \I z^3 b(z;t)) \phi^{\text m} = o(1), \quad x \to -\infty.
\end{equation}
This implies that
\begin{equation}
  a(z;t) = a(z;0)\E^{4 \I z^3 t}, \quad b(z;t) = b(z;0) \E^{-4 \I z^3 t}.
\end{equation}
and therefore
\begin{equation}
  R_{\mathrm l}(z;t) = R_{\mathrm l}(z;0) \E^{-8 \I z^3 t}.
\end{equation}
This also holds for $-c \leq z \leq c$. Now, consider
\begin{multline}
  B_t(z;t) \psi^{\text p} + B(z;t) \psi_t^{\text p} + A_t(z;t) \psi^{\text m} + A(z;t) \psi_t^{\text m}\\
   = (4 z^2 - 2 u(x,t) ) B(z;t) \psi_x^{\text p} + (4 z^2 - 2 u(x,t) )A(z;t) \psi_x^{\text m} + (u_x(x,t) + 1) (B(z;t)\psi^{\text p} + A(z;t)\phi^{\text m})
\end{multline}
and then as $x \to + \infty$,
\begin{equation}
   \psi_x^{\text p}(z;x,t) = \I \lambda(z) \psi^{\text p}(z;x,t)( 1 + o(1)), \quad \psi_x^{\text m}(z;x,t) = -\I \lambda(z) \psi^{\text m}(z;x,t)( 1 + o(1)),
\end{equation}
and $u(x,t) \to -c^2$, $u_x(x,t) \to 0$.  Therefore as $x \to + \infty$
\begin{equation}
 (B_t(z;t) - \I \lambda(z) (4 z^2 + 2 c^2) B(z;t)) \psi^{\text p} + (A_t(z;t) + \I \lambda(z) (4 z^2 + 2 c^2) A(z;t)) \psi^{\text m} = o(1).
\end{equation}
Therefore,
\begin{equation}
  B(z;t) = B(z;0) \E^{4\I \lambda^3(z) t+\I 6c^2\lambda(z)t}, \quad A(z;t) = A(z;0) \E^{-4\I \lambda^3(z) t-\I 6c^2\lambda(z)t}.
\end{equation}
This then gives for $s^2 > c^2$
\begin{equation}
  R_{\mathrm l}(s;t) = R_{\mathrm l}(s;0) \E^{8\I \lambda^3(s) t+\I 6c^2\lambda(s)t},
\end{equation}
and $R_{\mathrm l}(s;t) = R_{\mathrm l}(s;0) \E^{8\I \lambda_+^3(s) t+\I 6c^2\lambda_+(s)t}$ for $-c \leq s \leq c$.

Next, assume $a(z) = a(z;0)$ (and hence $A(z)$) has a simple zero at $z' \in \mathbb C^+$.  We then must incorporate a residue condition because $\mb N_1$ and $\mb N_2$ will no longer be analytic for $z \not \in \mathbb R$.  So, consider
\begin{equation}
  \mathrm{Res}_{z = z'}\, \mb N_1(z) = \mathrm{Res}_{z = z'}\, \mb L_1(z) \begin{bmatrix} \frac{1}{a(z)} & 0 \\ 0 & 1 \end{bmatrix}   \E^{- \I z x \sigma_3} = \begin{bmatrix} \mathrm{Res}_{z = z'}\, \frac{\psi^{\text p}(z;x,t)}{a(z;t)}\E^{- \I z x} & 0 \end{bmatrix} = \begin{bmatrix} \frac{\psi^{\text p}(z';x,t)}{a'(z';0)} \E^{- \I z x - 4 \I z^3 t} & 0 \end{bmatrix}
\end{equation}
because the second entry is analytic at $z = z'$.  Then the fact that $a(z'x,t) = 0$ implies that there exists  $b_{z'}(t) \in \mathbb C$ such that
\begin{equation}
  \psi^{\text p}(z';x,t) = b_{z'}(t) \phi^{\text m}(z';x,t), \quad b_{z'}(t) = b_{z'}(0) \E^{-4 \I z'^3 t}
\end{equation}
and therefore
\begin{equation}
  \begin{bmatrix} \frac{\psi^{\text p}(z';x,t)}{a'(z';0)} \E^{- \I z x - 4 \I z^3 t} & 0 \end{bmatrix} = \begin{bmatrix} {\phi^{\text m}(z';x,t)}\frac{b_{z'}(0)}{a'(z';0)} \E^{- \I z' x - 8 \I z'^3 t} & 0 \end{bmatrix} = \lim_{z \to z'} \mb N_1(z) \begin{bmatrix} 0 & 0\\  \frac{b_{z'}(0)}{a'(z';0)} \E^{- 2\I z' x - 8 \I z'^3 t} & 0 \end{bmatrix}.
\end{equation}
Similarly, at $z = - z'$
\begin{equation}
\begin{aligned}
  \mathrm{Res}_{z = -z'}\, \mb N_1(z) & =  \mathrm{Res}_{z = -z'} \mb N_1(-z) \sgo = \lim_{z \to -z'} (z+z')\mb N_1(-z) \sigma_1 = \lim_{z \to z'} (-z+z')\mb N_1(z) \sigma_1\\
  & = - \lim_{z \to z'} \mb N_1(z) \begin{bmatrix} 0 & 0\\  \frac{b_{z'}(0)}{a'(z';0)} \E^{- 2\I z' x - 8 \I z'^3 t} & 0 \end{bmatrix} \sgo =  \lim_{z \to -z'}\mb N_1(z) \sgo \begin{bmatrix} 0 & 0\\  -\frac{b_{z'}(0)}{a'(z';0)} \E^{- 2\I z' x - 8 \I z'^3 t} & 0 \end{bmatrix} \sgo.
  \end{aligned}
\end{equation}
Completing the analogous calculation for $\mb N_2(z)$, we find
\begin{equation}
\begin{aligned}
  \mathrm{Res}_{z=z'}\, \mb N_2(z) &=  \mathrm{Res}_{z = z'}\, \mb L_2(z) \begin{bmatrix} 1 & 0 \\ 0 & \frac{1}{A(z;x,t)} \end{bmatrix}\E^{\I \lambda(z) x \sigma_3} = \begin{bmatrix} \mathrm{Res}_{z=z'}\, \frac{\phi^{\text m}(z;x,t)}{A(z;t)} \E^{ \I \lambda(z) x} & 0\end{bmatrix} \\
                                   &= \begin{bmatrix} \frac{\phi^{\text m}(z';x,t)}{A'(z;0)} \E^{ \I \lambda(z') x + 4 \I \lambda(z') t + 6 \I c^2 \lambda(z')t} & 0 \end{bmatrix}\\
                                 & = \lim_{z \to z'} \mb N_2(z) \begin{bmatrix} 0 & 0\\  \frac{1}{b_{z'}(0)A'(z';0)} \E^{ 2\I \lambda(z') x + 8 \I \lambda(z') t + 12 \I c^2 \lambda(z')t} & 0 \end{bmatrix}
\end{aligned}
\end{equation}
and
\begin{equation}
\begin{aligned}
  \mathrm{Res}_{z=-z'}\, \mb N_2(z) &= \lim_{z \to z'} \mb N_2(z) \sgo \begin{bmatrix} 0 & 0\\  -\frac{1}{b_{z'}(0)A'(z';0)} \E^{ 2\I \lambda(z') x + 8 \I \lambda(z') t + 12 \I c^2 \lambda(z')t} & 0 \end{bmatrix} \sgo
\end{aligned}
\end{equation}
For such a value of $z'$, define
\begin{equation}
  c(z') = \frac{b_{z'}(0)}{a'(z';0)}, \quad C(z') = \frac{1}{b_{z'}(0)A'(z';0)}.
\end{equation}

\subsection{From residues to jumps}

It will be inconvenient in what follows for us to treat residue conditions directly.  So, we deform them to jump conditions on small circles.  Assume $\mb N(z)$ is a vector-valued analytic function in a open neighborhood $U$ of $z'$ that satisfies
\begin{equation}
  \mathrm{Res}_{z = z'}\, \mb N(z) = \lim_{z \to z'} \mb N(z)\begin{bmatrix} 0 & 0 \\ -\alpha & 0 \end{bmatrix}, \quad \alpha \in \mathbb C.
\end{equation}
Choose $\epsilon > 0$ small enough so that $\{|z-z'| = \epsilon\} \subset U$ and define
\begin{equation}
  \mb M(z) = \begin{cases} \mb N(z) \begin{bmatrix} 1 & 0 \\ \frac{\alpha}{z-z'} & 1 \end{bmatrix} & |z -z'| < \epsilon,\\
    \mb N(z) & \text{otherwise}. \end{cases}
\end{equation}
Then it follows that $\mb M$ is analytic in $U \setminus \{|z-z'| = \epsilon\}$ and if $\{|z-z'| = \epsilon\}$ is given a clockwise orientation, then
\begin{equation}
  \mb M^{+}(s) = \mb M^-(s) \begin{bmatrix} 1 & 0 \\ \frac{\alpha}{s-z'} & 1 \end{bmatrix}, \quad |s-z'| = \epsilon.
\end{equation}
In such a way, residue conditions are equivalent to rational jump conditions.


\section{Unique solvability of the Riemann--Hilbert problems}

\label{sec:KdV-Unique}

\subsection{Unique solvability of \rhref{rhp:1t}}

Before proving Theorem~\ref{t:uniqueRH1} we establish some elementary facts.
\label{sec:rhp:1t}

\begin{lemma}
  Assume $\Gamma$ is an admissible contour that satisfies $\Gamma = - \Gamma$, with a reversal of orientation.  Then $\mb F(z) = \begin{bmatrix} F_1(z) & F_2(z) \end{bmatrix}$, where $F_1,F_2 \in H^2_\pm(\Gamma)$ satisfies
  \begin{align}\label{eq:sym}
    \mb F(-z) = \mb F(z) \sgo, \quad z \in \mathbb C \setminus \Gamma
  \end{align}
  if and only if $\mb F(z) = \mathcal C_{\Gamma} \mb f(z)$ for some $\mb f \in L^2(\Gamma)$ (componentwise) satisfying
  \begin{align}\label{eq:consym}
    -\mb f(-s) = \mb f(s) \sgo, \quad s \in \Gamma.
  \end{align}
\end{lemma}
\begin{proof}
  Assume $\mb f \in L^2(\Gamma)$ satisfies \eqref{eq:consym}.  And consider, for $z \not\in \Gamma$,
  \begin{equation}
    \mb F(z) = \frac{1}{2 \pi \I} \int_{\Gamma} \frac{\mb f(s)}{s-z} \D s = \frac{1}{2 \pi \I} \int_{-\Gamma} \frac{\mb f(-s)}{s+z} \D s =  \frac{1}{2 \pi \I} \int_{\Gamma} \frac{\mb f(s)}{s+z} \sgo \D s = \mb F(-z) \sgo.
  \end{equation}
  Conversely, we have that $\mb F = \mathcal C_\Gamma \mb f$ for some $\mb f \in L^2(\Gamma)$ and if $\mb F$ satisfies \eqref{eq:sym} then for all $z \in \mathbb C \setminus \Gamma$
  \begin{equation}
    0 = \frac{1}{2 \pi \I} \int_{\Gamma}\left(\mb f(s) + \mb f(-s)  \sgo  \right)\frac{\D s}{s-z}.
  \end{equation}
  Because $\mathcal C^+_\Gamma \mb f(s) - \mathcal C^-_\Gamma \mb f(s) = \mb f(s)$ for a.e. $s \in \Gamma$, we find that \eqref{eq:consym} holds.
\end{proof}

\begin{definition}
  If $\Gamma$ is admissible, define
\begin{equation}
\begin{aligned}
    L_s^2( \Gamma) =  L_{+s}^2( \Gamma) &= \left\{ \mb f = \begin{bmatrix} f_1 & f_2 \end{bmatrix}, ~~ f_1,f_2 \in L^2(\Gamma), ~~ \mb f(s) = - \mb f(-s) \sgo \right\},\\
    L_{-s}^2( \Gamma) &= \left\{ \mb f = \begin{bmatrix} f_1 & f_2 \end{bmatrix}, ~~ f_1,f_2 \in L^2(\Gamma), ~~ \mb f(s) = \mb f(-s) \sgo \right\}.  
\end{aligned}
\end{equation}
  \end{definition}

  \begin{lemma}
    If $\Gamma$ is admissible then
    \begin{equation}
      L^2(\Gamma) = L_s^2(\Gamma) \oplus L_{-s}^2(\Gamma).
    \end{equation}
  \end{lemma}
  \begin{proof}
    For $\mb u \in L^2(\Gamma)$ define
    \begin{equation}
      \mathcal P \mb u(s) = \frac{1}{2} ( \mb u(s) - \mb u(-s) \sgo).
    \end{equation}
    Then $\mathcal P$ is a projection onto $L_s^2(\Gamma)$. It also follows that $\mathcal I-\mathcal P$ maps $L^2(\Gamma)$ onto $L_{-s}^2(\Gamma)$.
  \end{proof}

  \begin{lemma}\label{l:symprops} Suppose $\Gamma$ is admissible. 
    \begin{itemize}
    \item If $\mb u \in L_{\pm s}^2(\Gamma)$ then 
    \begin{equation}
      \mathcal C_{\Gamma}^- \mb u(-s) \sgo = \pm \mathcal C_{\Gamma}^+ \mb u(s),
    \end{equation}
    and therefore
    \begin{equation}
      \mathcal C_{\Gamma}^+ \mb u(-s) \sgo  = \pm \mathcal C_{\Gamma}^- \mb u(s).
    \end{equation}
    
  \item If $\mb M, \mb P: \Gamma \to \mathbb C^{2\times 2}$, $\mb M, \mb P \in L^\infty(\Gamma)$ satisfy
    \begin{align}\label{eq:MPsym}
      \mb M(s) = \sgo \mb P(-s)  \sgo
    \end{align}
    then the operator
    \begin{equation}
      \mb u \mapsto \mathcal C_{\Gamma}^+ \mb u \cdot \mb P - \mathcal C_{\Gamma}^- \mb u \cdot \mb M = \mb u - \mathcal C_{\Gamma}^- \mb u \cdot (\mb P - \mb M)
    \end{equation}
    maps $L_{\pm s}^2(\Gamma)$ to itself.
  \end{itemize}
\end{lemma}
\begin{proof}
  The calculation above implies the first part.  Let $\mb u \in L_{\pm s}^2(\Gamma)$.  Then the second part follows from
 \begin{equation}
\begin{aligned}
    \mathcal C_{\Gamma}^+ \mb u (-s) \mb P(-s) &\sgo - \mathcal C_{\Gamma}^- \mb u(-s) \mb M(-s) \sgo = \mathcal C_{\Gamma}^+ \mb u(-s)\sgo\mb M(s)  - \mathcal C_{\Gamma}^- \mb u(-s) \sgo \mb P(s)\\
    & = \pm \left(  \mathcal C_{\Gamma}^- \mb u(s) \mb M(s) - \mathcal C_{\Gamma}^+ \mb u(s) \mb P(s)\right) = \mp \left(  \mathcal C_{\Gamma}^+ \mb u(s) \mb P(s)-  \mathcal C_{\Gamma}^- \mb u(s) \mb M(s) \right).
\end{aligned}
\end{equation}
\end{proof}

\begin{theorem}\label{t:directsum}
  Suppose $\Gamma$ is admissible and $\mb M, \mb P: \Gamma \to \mathbb C^{2\times 2}$, $\mb M, \mb P \in L^\infty(\Gamma)$ satisfy \eqref{eq:MPsym}.  Further, suppose the operator
    \begin{equation}
      \mb u \mapsto \mathcal C\mb u : = \mb u - \mathcal C_{\Gamma}^- \mb u \cdot (\mb P - \mb M)
    \end{equation}
    is invertible on $L^2(\Gamma)$.  Then $\mathcal C|_{L_s^2(\Gamma)}$ is invertible on $L_s^2(\Gamma)$.
  \end{theorem}
  \begin{proof}
    It suffices to show that if $\mathcal C \mb u = \mb f$ where $\mb f \in L_s^2(\Gamma)$ then $\mb u \in L_s^2(\Gamma)$.  Suppose $\mb u = \mb v_+ + \mb v_-$ where $\mb v_\pm \in L_{\pm s}^2(\Gamma)$, and $\mb v_- \neq 0$.  Then $\mathcal C \mb v_- \in L_{-s}^2(\Gamma)$, and $\mathcal C \mb v_- \neq 0$.  But this contradicts that $\mb f \in L_s^2(\Gamma)$.
  \end{proof}

So, we find that any $L^2$ solution $\mb N_1$ of \rhref{rhp:1t} must satisfy $\mb N_1 = \mathcal C_{\Gamma} \mb u$  for some $\mb u \in L_s^2(\Gamma)$ and
\begin{equation}
\begin{aligned}
  \mb u(s) &- \mathcal C_{\mathbb R}^- \mb u(s) \cdot (\mb J_1(s) - \mb I) = \begin{bmatrix} 1 &  1 \end{bmatrix} \cdot(\mb J_1(s) - \mb I),\\
  \mb J_1(s) &= \begin{cases}\begin{bmatrix} 1 - |R_{\mathrm{l}}(s)|^2  & -\overline{R_{\mathrm{l}}}(s) \E^{2 \I s x + 8 \I s^3 t}  \\ {R_{\mathrm{l}}}(s) \E^{-2 \I s x- 8 \I s ^3 t} & 1 \end{bmatrix} & s \in \mathbb R,\\ \\
    \begin{bmatrix} 1 & 0 \\ -\frac{c(z_j)}{s - z_j}\E^{-2 \I z_j x - 8 \I z_j^3 t} & 1 \end{bmatrix} & s \in \Sigma_j,\\ \\
    \begin{bmatrix} 1 & -\frac{c(z_j)}{s + z_j}\E^{-2 \I z_j x - 8 \I z_j^3 t} \\ 0 &  1 \end{bmatrix} & s \in -\Sigma.
  \end{cases}
\end{aligned}
\end{equation}
We note that the operator $\mb u \mapsto \mb u - \mathcal C_{\Gamma}^- \mb u \cdot (\mb J_1 - \mb I)$ does not map $L_s^2(\Gamma)$ to itself.  So, we need to decompose $\mb J_1$ first.  Write
\begin{equation}
\begin{aligned}
  \mb J_1(s) &= \mb M_1(s) \mb P^{-1}_1(s) = \begin{bmatrix} 1 & - R_{\mathrm{l}}(-s) \E^{2 \I s x + 8 \I s^3 t} \\ 0 & 1 \end{bmatrix} \begin{bmatrix} 1 & 0  \\   R_{\mathrm{l}}(s) \E^{-2 \I s x - 8 \I s^3 t} & 1 \end{bmatrix}, \quad s \in \mathbb R,\\
  \mb J_1(s) &= \mb P^{-1}_1(s), ~~\mb M_1(s) = \mb I, \quad s \in \Sigma_j,\\
  \mb J_1(s) &= \mb M_1(s), ~~\mb P_1(s) = \mb I, \quad s \in -\Sigma_j.
\end{aligned}
\end{equation}

\begin{lemma}
  The operator
  \begin{equation}
  \begin{aligned}
    \mb u \mapsto \mb u \cdot \mb P_1 &- \mathcal C_{\Gamma}^- \mb u \cdot (\mb M_1 - \mb P_1)
    \end{aligned}
  \end{equation}
  is bounded on $L^2_s(\Gamma)$ to itself and if $R_{\mathrm l} \in L^2(\mathbb R)$ then
  \begin{align}\label{eq:inLs}
    \begin{bmatrix} 1 &  1 \end{bmatrix} \cdot (\mb M_1(\cdot) - \mb P_1(\cdot)) \in L_s^2(\Gamma).
  \end{align}
\end{lemma}
\begin{proof}
  It  follows that  
  \begin{equation}
    \sgo \mb M_1(-s) \sgo = \mb P_1(s).
  \end{equation}
  Then from Lemma~\ref{l:symprops} the lemma follows.
  
\end{proof}

\begin{lemma}
  The operator
  \begin{align}\label{eq:theop}
    \mb u \mapsto \mb u \cdot \mb P_1 &- \mathcal C_{\Gamma}^- \mb u \cdot (\mb M_1 - \mb P_1)
  \end{align}
  is Fredholm on $L^2_s(\Gamma)$ with index zero provided that $R_{\mathrm l}$ is continuous and decays at infinity.
\end{lemma}
\begin{proof}
  The fact that this operator is Fredholm on $L^2(\Gamma)$ follows from standard arguments \cite{TrogdonSOBook}.  This implies Fredholm on the invariant subspace $L^2_s(\Gamma)$.  Then replace $R_{\mathrm l}$ with $\alpha R_{\mathrm l}$ for $0 \leq \alpha \leq 1$.  For $\alpha$ sufficiently small, the operator is invertible and is therefore index zero. It must therefore be index zero for all $\alpha$.
\end{proof}

\begin{proof}[Proof of Theorem~\ref{t:uniqueRH1}]
  The unique solvability of \rhref{rhp:1t} is implied by the invertibility of \eqref{eq:theop}.  And to this end, because the Fredholm index of the operator is zero, it suffices to show that the kernel is trivial.  Assume $\mb u \in L^2_s(\Gamma)$ is an element of the kernel and define $\mb N(z) = \mathcal C_{\mathbb R}\mb u \in H^2_\pm(\Gamma)$.  It follows that $\mb N$ solves the $L^2$ RH problem
  \begin{equation}
    \mb N^+(s) = \mb N^-(s) \mb J_1(s), \quad s \in \Gamma, \quad \mb N(z) = \mb N(-z) \sgo, \quad z \in \mathbb C\setminus \Gamma.
  \end{equation}
  We use another symmetry of the contour $\Gamma$. If $U$ is a connected component of $\mathbb C \setminus \Gamma$ then so is $\overline{U} := \{\bar z : z \in U\}$.  Thus for $f \in \mathcal E^2(U)$, $\overline{f(\bar \cdot)} \in \mathcal E^2(\overline{U})$ and if $f \in \mathcal E^2(U)$ and $g \in \mathcal E^2(\overline{U})$ then
  \begin{equation}
    \int_{\partial U} f(s) \overline{ g(\bar s)} \D s = 0.
  \end{equation}
  We select $U$ to be the connected component in the upper-half plane that contains the real axis in its boundary.  The positively oriented boundary for $U$ is then the real axis, and $\cup_j \Sigma_j$ with reversed orientation.  Therefore
  \begin{align}
    0 &=  \int_{\mathbb R} \mb N^+(s) \overline { \mb N^-(s) }^T \D s - \sum_j \int_{\Sigma_j} \mb N^-(s) \overline { \mb N^+(\bar s)  }^T \D s, \label{eq:U}\\
    0 &=  \int_{\mathbb R} \mb N^-(s) \overline { \mb N^+(s) }^T \D s - \sum_j \int_{-\Sigma_j} \mb N^+(s) \overline { \mb N^-(\bar s)  }^T \D s. \label{eq:Ubar}
  \end{align}
  Here the second line arises from similar considerations for $\overline {U}$.  Taking orientation into account and using the symmetry of $\mb N$
  \begin{equation}
    \int_{\Sigma_j} \mb N^-(s) \overline { \mb N^+(\bar s) }^T \D s = -  \int_{-\Sigma_j} \mb N^-(-s) \overline { \mb N^+(-\bar s) }^T \D s  = -  \int_{-\Sigma_j} \mb N^+(s) \overline { \mb N^-(\bar s) }^T \D s.
  \end{equation}
  Thus, adding \eqref{eq:U} and \eqref{eq:Ubar}, we have
  \begin{equation}
    0 =  \Re \int_{\mathbb R} \mb N^+(s) \overline { \mb N^-(s) }^T \D s.
  \end{equation}    
We use this to show that $\mb N(z)= 0$ for $z \not\in \mathbb R$ which implies that $\mb u \equiv 0$.  If we set $\mb N(z) = \begin{bmatrix} N_1(z) & N_2(z) \end{bmatrix}$, we find
  \begin{equation}
\begin{aligned}
    \int_{\mathbb R} \mb N^+(s) \overline { \mb N^-(s) }^T \D s &= \int_{\mathbb R} \left[ \left| N_1^+(s)\right|^2 [ 1 - |R_{\mathrm l}(s)|^2 ] + \left| N_2^+(s)\right|^2  \right. \\ &+ \left.  N_2^+(s) \overline{N_1^+(s)} R_{l}(-s) \E^{2 \I s x + 8 \I s^3 t} +  \overline{N_2^+(s)} {N_1^+(-s)} R_{l}(s) \E^{-2 \I s x - 8 \I s^3 t} \right] \D s.
  \end{aligned}
\end{equation}
  Taking the real part of this expression, we find
  \begin{equation}
    0 = \int_{\mathbb R} \left[ \left| N_1^+(s)\right|^2 [ 1 - |R_{\mathrm l}(s)|^2 ] + \left| N_2^+(s)\right|^2 \right] \D s
  \end{equation}
  implying that $\mb N^+(s) = 0$ and therefore $\mb N(z) = 0$, because $|R_{\mathrm l}(s)| < 1$ for a.e. $s \in \mathbb R$ \cite{Kappeler1986}.
\end{proof}

\subsection{Unique solvability for \rhref{rhp:2t}}
\label{sec:rhp:2t}

\newcommand{\asm}[1]{Assumption~(#1)}
\newcommand{\asms}[1]{Assumptions~(#1)}

The jump matrix for \rhref{rhp:2t} is discontinuous and the Fredholm theory no longer applies.  We have to perform a lengthy regularization process and then we use the fact that \rhref{rhp:1t} has a unique solution to show that \rhref{rhp:2t} has a unique solution.  We peform deformations under the assumptions of Theorem~\ref{t:uniqueRH2}. In this section when we refer to \asm j, we are referring the $j$th assumption in Theorem~\ref{t:uniqueRH2}.  For simplicity we assume $n = 0$, i.e., no solitons.  Because all deformations are performed in a neighborhood of the real axis the result immediately applies to the case of $n > 0$.

The remainder of this section constitutes the proof of Theorem~\ref{t:uniqueRH2}

\begin{proof}[Proof of Theorem~\ref{t:uniqueRH2}]

From \asms{1,4-6,8,10}, $R_{\mathrm r}(s)$ is continuous for $s \in \mathbb R$ and satisfies
\begin{equation}
  R_{\mathrm r}(s) = L_{-c}(s) + E_{-c}(s), \quad E_{-c}(s) =  O(|s+c|), \quad s \to -c,
\end{equation}
and $L_{-c}$ has an analytic extension to a neighborhood $\{ |z + c| < \epsilon, \Im z > 0 \}$.  Note that $\overline R_{\mathrm l}(s) = R_{\mathrm l}(-s)$ follows from \asms{2,3,7,9}.  Then
\begin{equation}
  \mb N_2^+(s) = \mb N_2^-(s) \mb M_2(s) \mb P^{-1}_2(s) = \mb N_2^-(s) \begin{bmatrix} 1 & - R_{\mathrm{r}}(-s) \E^{-2 \I \lambda(s) x - 8 \I \varphi(s) t} \\ 0 & 1 \end{bmatrix} \begin{bmatrix} 1 & 0  \\   R_{\mathrm{r}}(s) \E^{2 \I \lambda(s) x + 8 \I \varphi(s) t} & 1 \end{bmatrix}.
\end{equation}
We factor
\begin{equation}
\begin{aligned}
  \mb M_2(s) &= \begin{bmatrix} 1 & -  L_c(-s) \E^{-2 \I \lambda(s) x - 8 \I \varphi(s) t} \\ 0 & 1 \end{bmatrix} \begin{bmatrix} 1 & -  E_c(-s) \E^{-2 \I \lambda(s) x - 8 \I \varphi(s) t} \\ 0 & 1 \end{bmatrix} = \mb M_{2,o}(s) \mb M_{2,e}(s),\\
  \mb P_2(s) &= \begin{bmatrix} 1 & 0  \\   - L_{-c}(s) \E^{2 \I \lambda(s) x + 8 \I \varphi(s) t} & 1 \end{bmatrix}\begin{bmatrix} 1 & 0\\ -  E_{-c}(s) \E^{2 \I \lambda(s) x + 8 \I \varphi(s) t}  & 1 \end{bmatrix} = \mb P_{2,o}(s) \mb P_{2,e}(s).
\end{aligned}
\end{equation}
Then, consider the jump matrix near $s = -c$, $s > -c$:
\begin{equation}
  \mb N_2^+(s) = \mb N_2^-(s) \mb M_2(s) \begin{bmatrix} 0 & 1 \\ 1 & 0 \end{bmatrix} \mb P^{-1}_2(s).
\end{equation}
Fix $\epsilon > 0$, and for $ z \not\in \mathbb R \cup \{z\colon |z + c| = \epsilon\}$ define
\begin{equation}
  \mb N_{2,1}(z) = \mb N_{2}(z) \begin{cases} \mb I & |z +c| > \epsilon, \\
    \mb M_{2,o}(z) & \Im z < 0 \text{ and } |z + c| < \epsilon,\\
  \mb P_{2,o}(z) & \Im z > 0 \text{ and } |z + c| < \epsilon. \end{cases}
\end{equation}
Then the sectionally analytic function $\mb N_{2,1}$ has the following jumps when we give the circle $ \{s\colon|s + c| = \epsilon\}$ a clockwise orientation:
\begin{equation}
  \mb N_{2,1}^+(s) = \mb N_{2,1}^-(s)
  \begin{cases} \mb M_2(s) \mb P^{-1}_2(s) & s < -c-\epsilon \text{ and } s > c, \\
    \mb M_{2,e}(s) \mb P^{-1}_{2,e}(s) & -c-\epsilon < s < -c,\\
    \mb M_{2,e}(s) \sgo \mb P^{-1}_{2,e}(s) & -c < s < -c + \epsilon, \\
    \mb M_{2,o}(s) & \Im s < 0,~ |s+c| = \epsilon,\\
     \mb P_{2,o}(s) & \Im s > 0,~ |s+c| = \epsilon.
  \end{cases}
\end{equation}
The jump on the real axis, inside the circle, is nearly of the form:
\begin{equation}
  \mb W^+(s) = \mb W^-(s) \begin{cases} \sgo & s > -c,\\
    \mb I & s < -c. \end{cases}
\end{equation}
To find such a solution $\mb W$ we first perform an eigen decomposition
\begin{equation}
\sgo = \frac 1 2 \begin{bmatrix} 1 & 1 \\ -1 & 1 \end{bmatrix} \begin{bmatrix} -1 & 0 \\ 0 & 1 \end{bmatrix} \begin{bmatrix} 1 & -1 \\ 1 & 1 \end{bmatrix}.
\end{equation}
Then we solve a matrix problem (keeping an identity condition at infinity)
\begin{equation}
  \mb V^+(z)  = \mb V^-(z)\begin{bmatrix} -1 & 0 \\ 0 & 1 \end{bmatrix},\quad \mb V(z) = \begin{bmatrix} \sqrt{\frac{z+c}{z-c}} & 0 \\ 0 & 1 \end{bmatrix}.
\end{equation}
We find the solution
\begin{align}\label{eq:W}
  \mb W(z) & = \frac 1 2 \begin{bmatrix} 1 & 1 \\ -1 & 1 \end{bmatrix} \begin{bmatrix} \sqrt{\frac{z+c}{z-c}} & 0 \\ 0 & 1 \end{bmatrix} \begin{bmatrix} 1 & -1 \\ 1 & 1 \end{bmatrix} \\
  & = \frac 1 2 \begin{bmatrix} \sqrt{\frac{z+c}{z-c}} & 1 \\ -\sqrt{\frac{z+c}{z-c}} & 1 \end{bmatrix}\begin{bmatrix} 1 & -1 \\ 1 & 1 \end{bmatrix} = \frac 1 2 \begin{bmatrix} \sqrt{\frac{z+c}{z-c}} + 1 & 1-\sqrt{\frac{z+c}{z-c}} \\ 1-\sqrt{\frac{z+c}{z-c}} & \sqrt{\frac{z+c}{z-c}} + 1 \end{bmatrix}.\notag
\end{align}
 We note that $\mb W(-z)$ is also a solution.  Then, perform the transformation, for $z \not\in \mathbb R \cup \{|z+c| = \epsilon\}$,
\begin{equation}
  \mb N_{2,2}(z) = \mb N_{2,1}(z) \begin{cases} \mb I & |z+c| > \epsilon,\\
\mb  W^{-1}(z) & |z + c| < \epsilon. \end{cases}
\end{equation}
For $-c - \epsilon < z < -c + \epsilon$, $z \neq 0$, the resulting jump for the function $\mb N_{2,2}(z)$ is given by
\begin{equation}
\mb G_{-c}(s) =  \mb W_-(s) \mb M_{2,e}(s) \mb W^{-1}_-(s) \mb W_+(s) \mb P^{-1}_{2,e}(s) \mb W_+^{-1} (s).
\end{equation}
We want this to be continuous and equal to the identity jump at $s = 0$.  Note that for $\kappa(z) = \sqrt{\frac{z+c}{z-c}}$
\begin{equation}
\mb H(s) = \mb W_{\pm}(s) \begin{bmatrix} 1 & f(s) \\ 0 & 1 \end{bmatrix} \mb W_{\pm}^{-1}(s) = \frac{1}{4} \begin{bmatrix} 1 & 1 \\ -1 & 1 \end{bmatrix} \begin{bmatrix} 2 - f(s) & f(s) \kappa_\pm(s) \\ - f(s) \kappa_{\pm}^{-1}(s) & 2 + f(s)  \end{bmatrix} \begin{bmatrix} 1 & -1 \\ 1 & 1 \end{bmatrix}.
\end{equation}
So, if $f(s) = O(|s+c|)$ as $s \to -c$, $\mb H(s) = \mb I + O(|s+c|^{1/2})$ as $s \to -c$.   While the jump condition for $\mb N_{2,2}(z)$ behaves nicely near $z = -c$, we do not know that the solution itself does.
  
Let $\Phi: \mathbb R \to \mathbb R$ be infinitely differentiable, non-negative, $\Phi(s) = 1$ for $|s| < \epsilon/4$ and $\Phi(s) = 0$ for $|s| > \epsilon/2$.  Then consider the $L^2$ RH problem
\begin{rhp}
  \begin{equation}
    \mb L^+(s) = \mb L^-(s) \left[\mb I + \Phi(s-c) (G_{-c}(s) - \mb I) \right], \quad  -c - \epsilon < s < c + \epsilon, \quad \mb L(\cdot) - \mb I \in H_\pm(\mathbb R).
  \end{equation}
\end{rhp}
For $\epsilon$ sufficiently small, it follows that this problem is uniquely solvable because the associated singular integral operator is a near-identity operator.  And because the jump matrix is 1/2-H\"older continuous by \asms{1,4}, so is the solution, giving with 1/2-H\"older continuous boundary values \cite{SIE}.  Furthermore, $\det \mb L(z) \neq 0$. Then set 
\begin{equation}
  \mb N_{2,3}(z) = \mb N_{2,1}(z) \begin{cases} \mb I & |z+c| > \epsilon,\\
\mb  W^{-1}(z) \mb L^{-1}(z) & |z - c| < \epsilon. \end{cases}
\end{equation}
It follows that $\mb N_{2,2}(z)$ has an identity jump in a neighborhood of $z = -c$.
\begin{lemma}
  Let $\Gamma$ be a differentiable curve parameterized by $\gamma: [-1,1] \to \Gamma$, $\gamma(t) = t + \I \ell(t)$, $\ell(0) = 0$ and define $\Gamma_\epsilon = \gamma((-1 + \epsilon, 1 - \epsilon))$.  Assume $g$ is analytic in  an the open set $\bigcup_{0 < |r| < R} (\Gamma_{2\epsilon} + \I r)$ and satisfies
  \begin{align}\label{eq:hest}
    \sup_{- R < 2 r < R, ~r \neq 0} \int_{\Gamma_{\epsilon}} |g(s + \I r)|^2 |\D s| < \infty,
  \end{align}
  for some $R > 0$ and $0 < \epsilon < 1/2$.  Then, assume the branch of $z \mapsto z^{-1/2}$ is chosen so that $h(z) = z^{-1/2}g(z)$ has an isolated singularity at $z = 0$.  Then $h$ is analytic at $z = 0$.
\end{lemma}
\begin{proof}
  First consider $f(z) = z^{1/2}g(z)$.  This has an isolated singularity at $z = 0$ and it satisfies
  \begin{equation}
    \sup_{0 < |r| < R} \int_{\Gamma_{\epsilon}} |f(s + \I r)|^2 |\D s| < \infty.
  \end{equation}
  It then follows that $f \in \mathcal E^2 (C_\pm)$ where $C_\pm = \bigcup_{0 < r < R/2} (\Gamma_\epsilon \pm \I r)$. For sufficiently small $\epsilon > 0$
\begin{equation}
\begin{aligned}
    \int_{\partial B(0,\epsilon)} f(z) \D z & = \int_{\partial B(0,\epsilon)\cap C_+} f(z) \D z + \int_{\partial B(0,\epsilon) \cap C_-} f(z) \D z \\
    & \quad + \int_{{\Gamma_{\epsilon}} \cap B(0,\epsilon)} f(z) \D z - \int_{{\Gamma_{\epsilon}} \cap B(0,\epsilon)} f(z) \D z = 0.
\end{aligned}
\end{equation}
  The same is true for $z^k f(z)$ for all integers $k > 0$.  Thus $f$ is analytic at $z = 0$.  We now claim that $f(0) = 0$.  Assume
  \begin{equation}
     f(z) = c + o(1), \quad z \to 0, \quad c \neq 0.
  \end{equation}
  There exists $\delta > O$, so that for $|z| < \delta$, $|f(z)| \geq |c|/2$.  Then $|h(z)| \geq |c| |z|^{-1/2}/2$ for $|z| < \delta$.  Then consider for $0 < r < R$
  \begin{equation}
    \int_{{\Gamma_{\epsilon}} \cap B(0,\delta)} |h(z + \I r)|^2 |\D z| \geq \frac{|c|^2}{4} \int_{{\Gamma_{\epsilon}} \cap B(0,\delta)} |z + \I r|^{-1} |\D z|.
  \end{equation}
  Then using the parameterization
  \begin{equation}
    \int_{{\Gamma_{\epsilon}} \cap B(0,\delta)} |z + \I r|^{-1} |\D z| \geq \int_{{\Gamma_{\epsilon}} \cap B(0,\delta)} \frac{|\D z|}{|z| + r}  \geq \int_{t_1}^{t_2} \frac{\D t}{\sqrt{t^2 + \ell^2(t)} + r}, ~~ t_1 < 0 < t_2.
  \end{equation}
  Then because $\ell(t)$ is differentiable and and satisfies $\ell(0) = 0$, we have $|\ell(t)| \leq C|t|$, $t_1 \leq t \leq t_2$ and we are left estimating
  \begin{equation}
    \int_{t_1}^{t_2} \frac{\D t}{\sqrt{t^2 + \ell^2(t)} + r} \geq \int_{t_1}^{t_2} \frac{\D t}{\sqrt{1 + C^2}|t| + r} \geq \frac{1}{\sqrt{1 + C^2}} \log \left( 1 + \frac{\sqrt{1 + C^2}t_1}{r} \right).
  \end{equation}
  This right-hand side tends to $\infty$ as $r \to 0$, contradicting \eqref{eq:hest}.  Thus $f(0) = 0$.  Then it follows that $\int_{\partial B(0,\epsilon)} z^k h(z) \D z = 0$ for all positive integers $k$ and $h$ must be analytic at $z = 0$.
\end{proof}

Applying this lemma to $\mb N_{2,3}(z)$ near $ z= -c$ we find that it is indeed analytic in a neighborhood of $z = -c$.  Specifically, each component of $\mb N_{2,3}$, inside the circle $|z + c| < \epsilon$ will be of the form
\begin{equation}
  h_1(z)\phi_1(z) + \frac{h_2(z)\phi_2(z)}{\sqrt{z+c}},
\end{equation}
where $\phi_j$ are bounded analytic functions for $\Im z \neq 0$ and $h_j$ satisfy the estimate $\sup_{0 < r < R} \int_{-\delta}^\delta |h_j(s \pm \I r)|^2 \D s < \infty$ for some $\delta > 0$, $R > 0$.  So we apply the lemma to
\begin{equation}
  g(z) = \sqrt{z+c}h_1(z)\phi_1(z) + {h_2(z)\phi_2(z)}.
\end{equation}
We are led to the following $L^2$ RH problem for $\mb N_{2,3}$:
\begin{rhp}
  Giving the circle $ \{|s + c| = \epsilon\}$ a clockwise orientation
\begin{equation}
  \mb N_{2,3}^+(s) = \mb N_{2,3}^-(s) \mb J_{2,3}(s) = \mb N_{2,3}^-(s)\begin{cases} \mb M_2(s) \mb P^{-1}_2(s) & s < -c-\epsilon \text{ and } s > c, \\
    \mb L_-(s) \mb G_{-c}(s) \mb L_+^{-1}(s) & -c-\epsilon < s < -c + \epsilon,\\
    \mb M_{2}(s) \sgo \mb P^{-1}_{2}(s) & -c + \epsilon < s < c, \\
    \mb M_{2,o}(s)\mb  W^{-1}(s) \mb L^{-1}(s) & \Im s < 0,~ |s+c| = \epsilon,\\
     \mb P_{2,o}(s)\mb  W^{-1}(s) \mb L^{-1}(s) & \Im s > 0,~ |s+c| = \epsilon.
  \end{cases}
\end{equation}
with $\mb N_{2,3}(\cdot) - \mb I \in H_\pm^2 (\mathbb R \cup \{|s+c| = \epsilon\} )$.
\end{rhp}  
To complete the proof of Theorem~\ref{t:uniqueRH2} we perform the following steps:
\begin{enumerate}
\item We perform a similar deformation of \rhref{rhp:2t} near $z = c$ using symmetry considerations.
\item Then we show the resulting singular integral operator is Fredholm, and show that it is index zero using a homotopy argument.
\item Then to show the kernel is trivial, we show that every distinct element of the kernel results in a distinct vanishing solution of \rhref{rhp:1t}.
\end{enumerate}
Step (1) is given as a RH problem.  We separate (2)-(4) into three lemmas.  The fact that
\begin{equation}
 \sgo \mb M_2(-s) \sgo = \mb P_2(s)
\end{equation}
implies
\begin{equation}
   \sgo \mb M_2(-z) \mb P_2^{-1}(-z) \sgo =  \sgo \mb M_2(-z) \sgo \sgo \mb P_2^{-1}(-z) \sgo = \mb P_2(z) \mb M_2^{-1}(z) = \left( \mb M_2(z) \mb P_2^{-1}(z) \right)^{-1}.
\end{equation}
This similarly holds for
\begin{equation}
\sgo \mb M_2(-s) \sgo \mb P_2^{-1}(-s) \sgo = \left( \mb M_2(s) \sgo \mb P_2^{-1}(s) \right)^{-1}.
\end{equation}
This is a necessary condition for $\mb N_2(-z)\begin{bmatrix} 0 & 1 \\ 1 & 0 \end{bmatrix} = \mb N_2(z)$ when $\mb N_2$ is a solution of \rhref{rhp:2t}.


Orient the circle $\{|s-c| = \epsilon\}$ with a clockwise orientation and define an $L^2$ RH problem that is regular at $\pm c$.
\begin{rhp} \label{rhp:N24} The function $\mb N_{2,4}(\cdot) - \begin{bmatrix} 1 & 1 \end{bmatrix} \in H_\pm^2(\Gamma)$
  \begin{equation}
    \mb N_{2,4}^+(s) = \mb N_{2,4}^-(s) \mb J_{2,4}(s), \quad s \in \Gamma,
  \end{equation}
  where
  \begin{align}\label{eq:circs}
    \Gamma = \mathbb R \cup \{|s+c| = \epsilon\} \cup \{|s-c| = \epsilon\},
  \end{align}
  and
  \begin{equation}
    \mb J_{2,4}(s) = \begin{cases} \mb J_{2,3}(s) & \Re s \leq 0,\\
      \sgo \mb J_{2,3}^{-1}(-s) \sgo  & \Re s > 0.\end{cases}
  \end{equation}
  Furthermore, $\mb N_{2,4}$ satisfies the symmetry condition
  \begin{equation}
    \mb N_{2,4}(-z) \begin{bmatrix} 0 & 1 \\ 1 & 0 \end{bmatrix} = \mb N_{2,4}(z), \quad z \in \mathbb C \setminus \Gamma.
  \end{equation}
\end{rhp}

\begin{lemma} The operator
  \begin{align}\label{eq:symop}
  \mb u \mapsto \begin{cases} \mb u(s) - \mathcal C_{\Gamma}^- \mb u(s) (\mb J_{2,4}(s) - \mb I) & s \in \Gamma, ~\Re s \leq 0,\\
    \mb u(s)\mb J_{2,4}^{-1}(s) - \mathcal C_{\Gamma}^- \mb u(s) (\mb I - \mb J_{2,4}^{-1}(s) ) & s \in \Gamma, ~\Re s > 0,
    \end{cases}
  \end{align}
  is Fredholm on $L^2_s(\Gamma)$ where $\Gamma$ is given in \eqref{eq:circs}.  Furthermore, the Fredholm index is zero.
  \end{lemma}  
\begin{proof}
  This RH problem satisfies the zeroth-order product condition \cite[Definition 2.55]{TrogdonSOBook} with continuous jump matrices.  Furthermore, $R_{\mathrm r}$ in addition to being continuous,  decays at infinity by \asm{11}, thus the operator
  \begin{equation}
  \mb u \mapsto \mb u - \mathcal C_{\Gamma}^- \mb u \cdot (\mb J_{2,4} - \mb I)
\end{equation}
is Fredholm on $L^2(\Gamma)$.  This implies that the operator \eqref{eq:symop} is also Fredholm on $L^2(\Gamma)$.  Because of the enforced symmetry of $\mb J_{2,4}$, this operator also maps $L_s^2(\Gamma)$ to itself (see Lemma~\ref{l:symprops}), and is therefore Fredholm on $L_s^2(\Gamma)$.  Now, to show that the index is zero, we replace $R_{\mathrm r}$ with $\alpha R_{\mathrm r}$ for $0 \leq \alpha \leq 1$.  It follows that $\mb J_{2,4}(s) \to \mb J_\infty(s)$, uniformly for $s \in \Gamma$, as $\alpha \to 0$ where $\mb J_\infty(s)$ for $\Re s \leq 0$ is given by
\begin{equation}
  \mb J_\infty(s) = \begin{cases} \sgo & -c + \epsilon < s  \leq 0,\\
  \mb W^{-1}(s) & |s + c | = \epsilon.\end{cases}
\end{equation}
and
\begin{equation}
  \mb J_\infty(s) = \sgo \mb J_\infty^{-1}(-s)  \sgo, \quad \Re s > 0.
\end{equation}
We construct the inverse operator to
\begin{align}\label{eq:Jinf}
  \mb u \mapsto \mb u - \mathcal C_{\Gamma'}^- \mb u \cdot (\mb J_\infty - \mb I) = \mathcal C_{\Gamma'}^+ \mb u - \mathcal C_{\Gamma'}^- \mb u \cdot J_\infty,\\\Gamma' = [-c + \epsilon, c -\epsilon] \cup \{|s + c| = \epsilon\} \cup \{|s - c| = \epsilon\},\notag
\end{align}
explicitly, and use this to show that the index of \eqref{eq:symop} is zero.\\

Consider the operator
\begin{align}\label{eq:Jinfinv}
\mb u \mapsto \mathcal C_{\Gamma'}^+ ( \mb u \cdot \mb W_+^{-1} ) \mb W_+ - \mathcal C_{\Gamma'}^- ( \mb u  \cdot \mb W_+^{-1} ) \mb W_-,
\end{align}
and its composition with \eqref{eq:Jinf} by considering
\begin{equation}
\begin{aligned}
  \mathcal C_{\Gamma'}^+ ( (\mathcal C_{\Gamma'}^+ \mb u - \mathcal C_{\Gamma'}^- \mb u \cdot \mb J_\infty ) \cdot \mb W_+^{-1} ) \mb W_+ &= \mathcal C_{\Gamma'}^+ \mb u - \mathcal C_{\Gamma'}^+ ( \mathcal C_{\Gamma'}^- \mb u \cdot \mb W_-^{-1} ) \mb W_+ = \mathcal C_{\Gamma'}^+ \mb u,\\
  \mathcal C_{\Gamma'}^- ( (\mathcal C_{\Gamma'}^+ \mb u - \mathcal C_{\Gamma'}^- \mb u \cdot \mb J_\infty ) \cdot \mb W_+^{-1} ) \mb W_- &= \mathcal C_{\Gamma'}^- \mb u.
\end{aligned}
\end{equation}
This shows that \eqref{eq:Jinfinv} is the left inverse of \eqref{eq:Jinf}.  Similar considerations show it is also the right inverse.  Now, this implies an inverse for \eqref{eq:symop} on $L^2(\Gamma')$ when $s = 0$:
\begin{equation}
  \mb u \mapsto  \mathcal C_{\Gamma'}^+ ( \mb u \cdot \hat {\mb W}  ) \mb W_+ - \mathcal C_{\Gamma'}^- ( \mb u  \cdot \hat{ \mb W} ) \mb W_-, \quad \hat{\mb W}(z) = 
  \begin{cases} \mb W_{+}^{-1}(z) & \Re z \leq 0,\\ \mb W_-^{-1}(z) & \Re z > 0. \end{cases}
\end{equation}
It is then enough to show that this operator maps $L_s^2(\Gamma')$ to itself\footnote{Note that \eqref{eq:symop} is the identity operator on $\Gamma\setminus\Gamma'$ for $s = 0$.}. This follows from Theorem~\ref{t:directsum}.
\end{proof}  

\begin{lemma}\label{l:N24}
  The kernel of the operator \eqref{eq:symop} is trivial, and therefore \rhref{rhp:N24} has a unique $L^2$ solution for any $\epsilon> 0$ sufficiently small.
\end{lemma}
\begin{proof}

  The following transformation essentially maps the function $\mb N_2$ to $\mb N_1$, with the exception of the exponentials,
\begin{equation}
  \mathcal T \mb N_2(z) := \begin{cases} \mb N_2(z) \E^{-( \I \lambda(z) x + 4 \I \varphi(z) t) \sigma_3}\begin{bmatrix}  A(z) & 0 \\ 0 & 1 \end{bmatrix} \sgo \begin{bmatrix}  \frac{1}{a(z)} & 0 \\ 0 & 1 \end{bmatrix} & \Im z > 0,\\ \\
    \mb N_2(z) \E^{-( \I \lambda(z) x + 4 \I \varphi(z) t) \sigma_3}\begin{bmatrix} 1 & 0 \\ 0 & A(-z) \end{bmatrix} \sgo \begin{bmatrix} 1 & 0 \\ 0 & \frac{1}{a(-z)} \end{bmatrix} & \Im z < 0.
    \end{cases}
\end{equation}
This should be equal to $\mb N_1(z)\E^{(\I x z + 4 \I x z^3 t)\sigma_3}$.  So, let $\mb u$ be an element of the kernel of \eqref{eq:symop}.  Define for $z \not\in \Gamma$, ($\Gamma$ is given in \eqref{eq:circs})
  \begin{equation}
    \mb Y(z) = \begin{cases} \mathcal C_{\Gamma} \mb u (z) & |z + c| < \epsilon, ~ |z - c| < \epsilon,\\
      \mathcal T \mathcal C_{\Gamma} \mb u (z), & \text{otherwise}. \end{cases}
  \end{equation}
   Of particular interest are the jumps on $|s \pm c| = \epsilon$.  On this circle for $\Im s > 0$
  \begin{equation}
    \mb Y^+(s) = \mb Y^-(s) \sgo \begin{bmatrix} \frac{1}{A(s)} & 0 \\ 0 & {a(s)} \end{bmatrix} \mb P_{2,o}(s) \mb W^{-1}(z) \mb L^{-1}(z) := \mb Y^-(s) \mb R(s).
  \end{equation}
  We must compute the inverse of this jump matrix
  \begin{align}\label{eq:j-prod}
    \mb R_{-c,+}(z) = \mb L(z) \mb W(z) \mb P_{2,o}(z) \E^{-(2 \I \lambda(z) x + 8 \I \varphi(z) t) \sigma_3}  \begin{bmatrix} {A(z)} & 0 \\ 0 & \frac{1}{a(z)} \end{bmatrix} \sgo
  \end{align}
  and we focus on the product, with the notation $f(z) = -L_{-c}(z) \E^{2 \I \lambda(z) x + 8 \I \varphi(z) t}$
  \begin{equation}
\begin{aligned}
    \mb W(z) \mb P_{2,o}^{-1}(z) &= \frac{1}{2} \begin{bmatrix} 1 & 1 \\ -1 & 1 \end{bmatrix} \begin{bmatrix} \sqrt{\frac{z + c}{z -c }} & 0 \\ 0 & 1 \end{bmatrix} \begin{bmatrix} 1 & -1 \\ 1 & 1 \end{bmatrix} \begin{bmatrix} 1 & 0 \\ - f(z) & 1 \end{bmatrix}\\
                            & = \frac{1}{2} \begin{bmatrix} 1 & 1 \\ -1 & 1 \end{bmatrix} \begin{bmatrix} \sqrt{\frac{z + c}{z -c }} & 0 \\ 0 & 1 \end{bmatrix} \begin{bmatrix} 1 + f(z) & -1 \\ 1 - f(z) & 1 \end{bmatrix}\\
    & =  \frac{1}{2} \begin{bmatrix} 1 & 1 \\ -1 & 1 \end{bmatrix}  \begin{bmatrix} ( 1 + f(z)) \sqrt{\frac{z + c}{z -c }}  & - \sqrt{\frac{z + c}{z -c }} \\ 1 - f(z) & 1 \end{bmatrix}.
 \end{aligned}
\end{equation}
  We know that $A(z)$ blows up as a square root at $z = -c$ by \asms{2,7}, so for \eqref{eq:j-prod} to be bounded for $|z + c| \leq \epsilon$, $\Im z > 0$, $f(-c) = 1$ is required, and because $R_{\mathrm r}$ is $1/2$-H\"older continuous, we have $L_{-c}(z) = -1 + O(|z + c|^{1/2})$.  This shows that \eqref{eq:j-prod} is a bounded analytic function.  Similarly,
\begin{equation}
\begin{aligned}
    \mb W(z) \mb M^{-1}_{2,o}(z) &= \frac{1}{2} \begin{bmatrix} 1 & 1 \\ -1 & 1 \end{bmatrix} \begin{bmatrix} \sqrt{\frac{z + c}{z -c }} & 0 \\ 0 & 1 \end{bmatrix} \begin{bmatrix} 1 & -1 \\ 1 & 1 \end{bmatrix} \begin{bmatrix} 1 & -f(-z) \\ 0 & 1 \end{bmatrix}\\
                            & = \frac{1}{2} \begin{bmatrix} 1 & 1 \\ -1 & 1 \end{bmatrix} \begin{bmatrix} \sqrt{\frac{z + c}{z -c }} & 0 \\ 0 & 1 \end{bmatrix} \begin{bmatrix} 1 & - f(-z) - 1 \\ 1 & 1 - f(-z) \end{bmatrix}\\
    & = \frac{1}{2} \begin{bmatrix} 1 & 1 \\ -1 & 1 \end{bmatrix}  \begin{bmatrix} \sqrt{\frac{z + c}{z -c }} & - \sqrt{\frac{z + c}{z -c }}( f(-z) + 1) \\ 1 & 1 - f(-z) \end{bmatrix},\\
\end{aligned}
\end{equation}
  shows that
\begin{equation}\label{eq:j-prod-}
    \mb R_{-c,-}(z) = \mb L(z) \mb W(z) \mb M^{-1}_{2,o}(z) \E^{-(2 \I \lambda(z) x + 8 \I \varphi(z) t) \sigma_3}\begin{bmatrix} \frac{1}{a(-z)} & 0 \\ 0 & A(-z) \end{bmatrix} \sgo,
\end{equation}
is a bounded analytic function for $\{|z + c| < \epsilon\}$, $\Im z < 0$ because $L_{-c}(-z) = -1 + O(|z+c|^{1/2})$. If we define for $\Re z \leq 0$
  \begin{equation}
    \mb Z(z) = \begin{cases} \mathcal C_{\Gamma} \mb u (z) \mb R_{-c,+}(z) & |z + c| < \epsilon, \Im z > 0,\\
      \mathcal C_{\Gamma} \mb u (z) \mb R_{-c,-}(z) & |z + c| < \epsilon, \Im z < 0,\\
      \mathcal T \mathcal C_{\Gamma} \mb u (z), & \text{otherwise}. \end{cases}
  \end{equation}
  and $\mb Z(z) = \mb Z(-z)\sgo$ for $\Re z >0$,   we obtain a function with $L^2(\mathbb R)$ boundary values and no jumps on $|s \pm c| = \epsilon$.  Then it follows that
  \begin{equation}
    \mb Z(z) \E^{(2 \I z x + 8 \I z^3 t) \sigma_3},
  \end{equation}
  is a solution of \rhref{rhp:1t}, by \eqref{eq:transl} and \eqref{eq:Btob}, with $\mb Z \in H_{\pm}^2(\mathbb R)$, and therefore $\mb Z = 0$.  This implies $\mb u = 0$.
\end{proof}
The last step is to establish the following injection.
\begin{lemma}\label{l:correspond}
  Every $L^2$ solution of \rhref{rhp:N24} corresponds to one and only one solution of \rhref{rhp:2t}.
\end{lemma}
\begin{proof}
  The careful derivation of \rhref{rhp:N24} implies that each solution of \rhref{rhp:2t} can be deformed to a solution of \rhref{rhp:N24} for any $\epsilon$ sufficiently small.  Because the functions $\mb L(z) \mb W(z) \mb P_{2,o}(z)$ and $\mb L(z) \mb W(z) \mb M_{2,o}(z)$ are bounded analytic functions in the domains $\{ |z + c| < \epsilon, \Im z > 0 \}$ and $\{ |z + c | < \epsilon, \Im z < 0\}$, respectively.  This allows the inversion of the deformations, so that each $L^2$ solution of \rhref{rhp:N24} gives an $L^2$ solution of \rhref{rhp:2t}.    
\end{proof}

  Given two distinct solutions $\mb N_2^{(1)}$ and $\mb N_2^{(2)}$ of \rhref{rhp:2t}, they  must differ at some point $z^*$, $\mb N_2^{(1)}(z^*) \neq \mb N_2^{(2)}(z^*)$, $z^* \not \in \mathbb R$.  Then $\min\{|z^*-c|, |z^* + c|\} > \delta$ for some $\delta >0$.  We perform the deformation to \rhref{rhp:N24} for $0 < \epsilon < \delta$ for each solution, and Lemma~\ref{l:N24} gives a contradiction, and establishes uniqueness. The existence is also guaranteed by Lemmas~\ref{l:N24} and \ref{l:correspond}.

\end{proof}

\bibliographystyle{plain}
\bibliography{libraryTT}  
\end{document}